\documentclass{amsart}
\usepackage{amsmath}
\usepackage{amsfonts}
\usepackage{amssymb}
\usepackage[T1]{fontenc}

\usepackage{xcolor}
\usepackage{amsthm}
\usepackage{amssymb}
\usepackage{amsfonts}
\usepackage{mathtools}
\usepackage{ esint }
\usepackage{latexsym}
\usepackage{graphicx}
\usepackage{mathabx} 
\usepackage{cleveref}
\usepackage[utf8]{inputenc}

\crefformat{section}{\S#2#1#3} 
\crefformat{subsection}{\S#2#1#3}
\crefformat{subsubsection}{\S#2#1#3}

\usepackage{cite}

\usepackage[normalem]{ulem}

\newtheorem{theorem}{Theorem}[section]

\newtheorem{corollary}[theorem]{Corollary}
\newtheorem{example}[theorem]{Example}
\newtheorem{proposition}[theorem]{Proposition}
\newtheorem{lemma}[theorem]{Lemma}

\theoremstyle{definition}
\newtheorem{definition}[theorem]{Definition}

\theoremstyle{remark}
\newtheorem{remark}[theorem]{Remark}

\newcommand{\con}{\operatorname{Cond}}

\newcommand{\R}{\mathbb{R}}

\newcommand{\Rd}{\mathbb{R}^d}

\newcommand{\Z}{\mathbb{Z}}

\newcommand{\T}{\mathbb{T}}
\newcommand{\TT}{\mathbb{T}^2}

\newcommand{\C}{\mathbb{C}}

\newcommand{\N}{\mathbb{N}}

\newcommand{\Hi}{\mathcal{H}}

\DeclareMathOperator\supp{supp}
\newcommand\norm[1]{\left\lVert#1\right\rVert}
\newcommand\abs[1]{\left\lvert#1\right\rvert}

\newcommand{\spp}{\textbf{supp }}

\newcommand{\Cn}{\text{Cond}}
\newcommand{\U}{\mathbb{U}}
\renewcommand{\P}{\mathcal{P}}
\renewcommand{\L}{\sout{L}}

\allowdisplaybreaks

\title[Strichartz Estimates and GWP for Periodic Quintic NLS in 1D]{Strichartz Estimates and Small-Mass Global Well-posedness for the periodic quintic NLS in 1D}
\author{Nikolaos Skouloudis} 
\address{LIDS, Massachusetts Institute of Technology, Cambridge, MA 02139-4307, USA}
\email{nskoulou@mit.edu}

\author{Jiahui Yu}
\address{Department of Mathematics, Massachusetts Institute of Technology, Cambridge, MA 02142-4307, USA}
\email{jiahu878@mit.edu}

\begin{document}

\begin{abstract}
We consider the periodic quintic nonlinear Schrödinger and prove small-mass global well-posedness in $H^s(\T)$ for $s>0$. The proof relies on a new derivative-loss-free $L^6_{t,x}$ Strichartz estimate which is established using the high-low method, an asymmetric superlevel set estimate and a new refined broad-narrow argument. Although our $L^6_{t,x}$ Strichartz estimate is not sharp, being valid on slightly shorter time scales than the optimal logarithmic scale, combining it with the $I$-method enables the extension of local solutions to arbitrary times.
\end{abstract}
\maketitle
\section{Introduction}
We consider the quintic nonlinear Schrödinger equation (NLS) on the unit circle $\T=\R/\left(2\pi \Z \right)$,
\begin{equation}
    \label{eq:NLS}
    \begin{cases}
       i \partial_tu +\Delta u=\mu \abs{u}^4u, \\
       u(x,0)=u_0(x) \in H^{s}\left(\T\right),
    \end{cases}
    \tag{NLS}
\end{equation}
where $\mu = 1$ corresponds to the defocusing case and $\mu = -1$ to the focusing case. For smooth solutions \(u\), the mass \(M(u(t))\) and energy \(E(u(t))\) are conserved along the flow of \eqref{eq:NLS}, where
\begin{subequations}
\begin{align}
 \label{def: mass} 
    M\left(u(t)\right) &:= \norm{u(t)}_{L_x^2(\T)},\tag{mass} \\
 \label{def: energy}
    E(u(t)) &:= \frac{1}{2}\norm{\partial_x u(t)}_{L_x^2(\T)}^2 + \frac{\mu}{6} \norm{u(t)}^6_{L_x^6(\T)}.   \tag{energy}  
\end{align}
\end{subequations}
Moreover, \eqref{eq:NLS} enjoys a scaling symmetry which is critical with respect to the mass. Specifically, for any $\lambda>0$, we map a solution of \eqref{eq:NLS} to another solution, $u \mapsto u^\lambda$ via the transformation
\begin{equation*}
    u^\lambda(x,t):= \frac{1}{\lambda^\frac{1}{2}} u\left(\frac{x}{\lambda}, \frac{t}{\lambda^{2}}\right),
\end{equation*}
and stress that $u^\lambda$ is defined on $\T_\lambda \times \R$ where $\T_\lambda := \R /(2\pi\lambda\Z)$. We say that \eqref{eq:NLS} is mass-critical as the rescaling leaves the mass invariant.

To formulate our main results, we recall the standard notion of well-posedness. The initial value problem \eqref{eq:NLS} is \textit{locally well-posed} if, for any $u_0 \in H^s(\T) $, there exists a time $T>0$, an open ball $B$ in $H^s(\T)$ containing $u_0$, and a closed subspace, $X$, of $C^0([0,T] \rightarrow H^s(\T))$, such that for each $u_0 \in B$ there exists a unique $u \in X$ satisfying
\begin{equation*}
    u(t)=S(t)u_0+i \mu  \int_0^t S(t-t') \abs{u(t')}^4u(t') dt',
\end{equation*}
where $S(t):= e^{it \Delta} $ is the Schrödinger group. Moreover, we require the map $u_0 \mapsto u$ to be uniformly continuous as a map from $B$ (endowed with the $H^s(\T)$ topology) to X (endowed with the $C^0([0,T] \rightarrow H^s(\T))$ topology). If we can take $T>0$ to be arbitrarily large, we say that \eqref{eq:NLS} is \textit{globally well-posed}. 

The main objective of this paper is to establish global well-posedness of the Cauchy problem \eqref{eq:NLS} for $s>0$. Our approach is based on establishing a derivative-loss-free $L^6_{t,x}$ Strichartz estimate and combining this estimate with the $I$-method. We begin by stating the Strichartz estimate.

\begin{theorem}
\label{theorem:L^6 Strichartz}
There exists a constant $C>1$ such that for any $N \geq 1$ and any $\phi \in L^2(\T)$ with $\supp \hat{\phi} \subset [-N, N]$, the following estimate holds
\[
\norm{S(t) \phi}_{L^6_{t,x}\left( \left[0, \left(\log N \right)^{-C}\right] \times \T \right)} \leq C\norm{\phi}_2.
\]
\end{theorem}

Theorem~\ref{theorem:L^6 Strichartz} is proved via a lossless small-cap decoupling estimate. A key ingredient in the argument is an estimate on the size of the intersection of superlevel sets of functions whose Fourier supports are contained in caps $\Sigma_1$ and $\Sigma_2$. In the regime where the length of $\Sigma_1$ is much smaller than the distance between $\Sigma_1$ and $\Sigma_2$, we establish an asymmetric superlevel set estimate which yields an additional gain reflecting this scale separation (see Lemma \ref{asym} for details). This gain is used to eliminate logarithmic losses in the decoupling argument. As a further consequence of this analysis, we recover a refined trilinear Strichartz estimate due to McConnell\cite[Proposition 5.10]{McConnell_Mass_critical_d_1}, up to an arbitrarily small polynomial loss in the frequency scale. We refer the reader to Appendix~\ref{sec:appendix A} for a proof using our method.

The derivative-loss-free estimate in Theorem~\ref{theorem:L^6 Strichartz} provides the key input for the nonlinear analysis. Combining this estimate with the $I$-method, we prove the following small-mass global well-posedness result for \eqref{eq:NLS}.

\begin{theorem}
\label{theorem:main_gwp_theorem}
There exists a constant $\delta>0$ such that for any $s>0$ and any initial data $u_0 \in H^s(\T)$ satisfying $\norm{u_0}_2 \leq \delta$, the Cauchy problem (\ref{eq:NLS}) is globally well-posed.
\end{theorem}

\subsection{Context and previous work}

In his seminal work \cite{Bourgain_lwp}, Bourgain established local well-posedness of \eqref{eq:NLS} for $s > 0$. This result is sharp (up to endpoint) in the sense that the data-to-solution map fails to be uniformly continuous for $s<0$ \cite{Burq_illposedness_T}. A key component of Bourgain's \cite{Bourgain_lwp} proof was the following Strichartz estimate: there exists a constant $c>0$ such that for all $\phi \in L^2(\T)$ satisfying $\supp \widehat{\phi} \subset \left[-N,N\right]$, the following estimate holds
\begin{equation}
\norm{S(t) \phi }_{L^6_{t,x}\left(\T \times \T\right)} \leq C_N \norm{\phi}_2, \text{ where } C_N:= c \exp\left(c \frac{\log N}{\log \log N}\right). 
\label{eq: bourgain Strichartz}
\end{equation}
This Strichartz estimate has since played a central role in the analysis of the nonlinear Schrödinger equation on the torus and has served as a starting point for a number of refinements and applications. We briefly review these developments, focusing first on improvements to \eqref{eq: bourgain Strichartz} and the methods underlying their proofs, and then on applications to global well-posedness.

We begin with a discussion of previous improvements to \eqref{eq: bourgain Strichartz}. While the estimate \eqref{eq: bourgain Strichartz} was sufficient to establish local well-posedness for $s>0$, it is not expected to be sharp in terms of the growth of $C_N$. Indeed, by considering $\phi = \sum_{|k| \leq N} e^{ikx}$, it was conjectured in \cite{Bourgain_lwp,Kishimoto_sharp_CN} that the derivative loss could be reduced to $C_N \sim (\log N)^{1/6}$. Despite significant effort, this conjecture remains open. Partial progress has been achieved through refinements of the decoupling techniques underlying the $\ell^2$ decoupling theorem of Bourgain and Demeter \cite{Bourgain_Demeter_decoupling}. For instance, Guth--Maldague--Wang \cite{GuthMaldagueWang_Improved_Dec_Parabola} applied the high-low method introduced by Guth--Solomon--Wang \cite{Guth_Solomon_Wang} to obtain bounds of the form $C_N \sim (\log N)^c$ for some absolute constant $c \gg 1$. This was subsequently improved by Guo--Li--Yung \cite{GuoLiYung_Strichartz}, who established an asymmetric bilinear decoupling inequality to show that $C_N \lesssim_\epsilon (\log N)^{2+\epsilon}$ for every $\epsilon>0$.

An alternative approach to obtaining estimates of the form of \eqref{eq: bourgain Strichartz} is to establish derivative-loss-free Strichartz estimates on shorter, frequency-dependent time intervals. More precisely, one seeks the largest time $T>0$ such that there exists a constant $c>0$ with the property that, for all $\phi \in L^2(\T)$ satisfying $\supp \hat{\phi} \subset [-N,N]$, the estimate \begin{equation} 
\|S(t)\phi\|_{L^6_{t,x}([0,T]\times\T)} \leq c \|\phi\|_2 
\label{eq:derivative_free_L6} 
\end{equation} holds. If one could establish \eqref{eq:derivative_free_L6} with $T=(\log N)^{-1}$, then by decomposing $[0,2\pi]$ into intervals of length $(\log N)^{-1}$ one would immediately obtain \eqref{eq: bourgain Strichartz} with $C_N\sim (\log N)^{1/6}$. In this direction, Burq--Gérard--Tzvetkov \cite{Burq_Gerard_Tzvetkov} proved \eqref{eq:derivative_free_L6} with $T = N^{-1}$. Their result in fact extends to arbitrary compact Riemannian manifolds in higher dimensions, for the same time scale. In the periodic setting, Herr and Kwak \cite{Herr_Kwak_2024_mass_critical} obtained the $L^4_{t,x}$ analogue of \eqref{eq:derivative_free_L6} on the two-dimensional square torus
\begin{equation}
\norm{S(t) \phi }_{L^4_{t,x}\left( \left[0, \frac{1}{\log N}\right] \times \TT\right)} \leq c \norm{\phi}_2,
\label{eq: Herr_Kwaq Strichartz}
\end{equation}
by using the Szemerédi--Trotter theorem to count rectangles with vertices in a finite set in the plane. This in turn yields a sharp Strichartz estimate for the higher-dimensional $L^4_{t,x}$ analogue of \eqref{eq: bourgain Strichartz} and also plays an important role in their proof of small-mass global well-posedness. Returning to the one-dimensional periodic setting, the best known result for the maximal time for which \eqref{eq:derivative_free_L6} holds is $T = N^{-131/208}$, due to McConnell \cite{McConnell_Mass_critical_d_1}, who used number theoretic methods to count the number of lattice points in thin annuli. 

In this paper, we prove that \eqref{eq:derivative_free_L6} holds with $T = (\log N)^{-10000}$. While we make no attempt to obtain a sharp exponent, this time scale already suffices for our purposes of establishing small-mass global well-posedness. We also note that our result does not improve the best known bound \cite{GuoLiYung_Strichartz} for the constant $C_N$ in \eqref{eq: bourgain Strichartz}. Our approach is based on small-cap decoupling, in contrast with the number-theoretic methods of McConnell \cite{McConnell_Mass_critical_d_1} and the incidence-geometric approach of Herr and Kwak \cite{Herr_Kwak_2024_mass_critical}. Small-cap decoupling may be viewed as a refinement of the classical $\ell^2$ decoupling theory adapted to the study of Strichartz estimates on time intervals shorter than a single period. The subject was initiated by Demeter--Guth--Wang \cite{DGW} and further developed in the two-dimensional setting by Fu--Guth--Maldague \cite{FGM} and Guth--Maldague \cite{amplitude} using the high--low method.

We now turn to the global well-posedness problem for \eqref{eq:NLS}. On the real line, Dodson \cite{Dodson_JEDP_2011} proved global well-posedness of \eqref{eq:NLS} in $H^s(\mathbb{R})$ for all $s \ge 0$, by establishing that solutions scatter in $L^2(\mathbb{R})$. In the periodic setting, as scattering is not expected to hold due to the lack of dispersion \cite{CKSTT2010_energy_transfer}, a different approach is required. In this context, derivative-loss-free Strichartz estimates have played a central role in recent advances toward establishing global well-posedness for mass-critical nonlinear Schr\"odinger equations.

In a breakthrough result, Herr and Kwak \cite{Herr_Kwak_2024_mass_critical, Herr_Kwak_2025_mass_critical_large_mass} established global well-posedness for the mass-critical NLS on $H^s(\TT)$ for $s>0$. Their approach proceeds in two stages: first \cite{Herr_Kwak_2024_mass_critical} they proved small-mass global well-posedness using the derivative-free Strichartz estimate \eqref{eq: Herr_Kwaq Strichartz} and subsequently \cite{Herr_Kwak_2025_mass_critical_large_mass} they removed the small assumption and proved that their result is sharp in the sense that the data-to-solution map fails to be uniformly continuous for initial data in $L^2(\TT)$. The goal of the present paper is to establish the one-dimensional analogue of the small-mass global well-posedness result proved in \cite{Herr_Kwak_2024_mass_critical}.

We now briefly recall the argument in \cite{Herr_Kwak_2024_mass_critical}. Under a suitable smallness assumption on the mass, the derivative-free Strichartz estimate \eqref{eq: Herr_Kwaq Strichartz} yields local well-posedness in the complete metric space
\[
X_N:=\left\{ u\in C^0\left( J_N \rightarrow H^s(\TT) \right) \cap Y^s_{J_N}: \norm{u}_{Z_N} \lesssim 1 \right\},
\]
where $N \gg 1$ is sufficiently large constant, $J_N:=\left[0,\frac{1}{\log N} \right]$ is the existence time interval, the space $Y^s_{J_N}$ is the time-restricted $Y^s$ space (see \cref{sec:4.1} for details) and
\[
\norm{u}_{Z_N}^2:= \norm{u}_{Y^0_{J_N}}^2+ N^{-2s} \norm{u}_{Y^s_{J_N}}^2.
\]
As the length of the existence time interval scales with $\frac{1}{\log N}$, local solutions may be continued for arbitrary time intervals, thereby establishing global well-posedness. In fact, Herr and Kwak \cite{Herr_Kwak_2024_mass_critical} proved that for every $k \in \N$, local solutions can be extended from $[0,T_k]$ to $[0,T_{k+1}]$ where 
\[ T_k := \sum_{j=0}^{k-1} \frac{1}{2 \log\!\bigl(K^j N_0\bigr)}, \]
with fixed constants $K \gg 1$ and $N_0 \gg_s 1$. This procedure can be iterated for arbitrarily large time intervals since $\lim_{k \rightarrow \infty} T_k = \infty$. Crucially, this conclusion relies on the fact that the length of each extension step scales like $\frac{1}{\log N}$. If instead the length of existence interval scaled with $\frac{1}{(\log N)^{c}}$ for any $c>1$, the above summation would converge and the iteration scheme would fail. Therefore, we stress that this argument would not work if \eqref{eq: Herr_Kwaq Strichartz} were valid for shorter time intervals of length $\frac{1}{\left(\log N\right)^c}$ for any $c>1$.

An alternative approach to extending local solutions to global ones is based on controlling the growth of the $H^s$ norm of smooth solutions to \eqref{eq:NLS} via the $I$-method. Originally introduced in \cite{I-method_1,I-method_2}, the $I$-method has become a canonical tool in obtaining global well-posedness for infinite energy initial data. This method relies on constructing families of modified Hamiltonians which are almost conserved and has been successfully applied to a range of dispersive equations, including the nonlinear Schrödinger equations \cite{mass_critical_gwp_staffilani_tzirakis, Schippa_mass_critical, LiWuXu_Mass_critical_d_1,fractional_global_wellposedness_T}, the derivative NLS \cite{I-method_2}, and the KdV equation \cite{KdV}, both in periodic and Euclidean settings.

In this direction, the $I$-method has been combined with refined bilinear and trilinear estimates together with \eqref{eq:derivative_free_L6} to establish global well-posedness for \eqref{eq:NLS} for low-regularity initial data. In fact, Bourgain \cite{Bourgain_trilinear} developed an approach combining the $I$-method with normal form reductions to prove global well-posedness below the energy space. Subsequent developments of this framework led to progressively lower regularity thresholds. Specifically, De Silva~\emph{et al.}~\cite{mass_critical_gwp_staffilani_tzirakis, Staffalani_mass_critical_errata} established a refined bilinear Strichartz estimate which yielded global well-posedness for $s > \tfrac{4}{9}$, while Li~\emph{et al.}~\cite{LiWuXu_Mass_critical_d_1} further lowered the regularity threshold to $s > \tfrac{2}{5}$ via a resonant decomposition of the modified energy.

More recently, derivative-free Strichartz estimates have been incorporated into the $I$-method framework to obtain further improvements in the regularity threshold for global well-posedness under an additional small-mass assumption. For example, Schippa \cite{Schippa_mass_critical} combined the derivative-free Strichartz estimate of Burq--Gérard--Tzvetkov \cite{BurqGerardTzvetkov_Strichartz}, valid up to time $T=N^{-1}$, with the refined bilinear Strichartz of De Silva~\emph{et al.}~\cite{mass_critical_gwp_staffilani_tzirakis} to prove global well-posedness for $s > \tfrac{1}{3}$. McConnell \cite{McConnell_Mass_critical_d_1} further improved this threshold to $s > \tfrac{131}{624}$ by establishing \eqref{eq:derivative_free_L6} with $T= N^{-\frac{131}{208}}$ and by proving the following refined asymmetric Strichartz estimate
\begin{equation}
\norm{S(t)\phi_1 (S(t)\phi_2)^2}_{L^2_{t,x}([0,T]\times\T)}^2 \lesssim \left(T^{1/2}+\frac{N_2}{N_1}\right) \norm{\phi_1}_2^2 \norm{\phi_2}_2^4,
\label{eq: strichartz_asym}
\end{equation}
for all $\phi_1, \phi_2 \in L^2(\T)$ satisfying $\supp \widehat{\phi}_j \subset [N_j, 2N_j]$ for $j=1,2$ with $N_1 \gg N_2$. This estimate should be viewed as a refinement of \eqref{eq:derivative_free_L6}, reflecting an additional gain arising from the frequency separation between the interacting functions. In this paper we are able to recover \eqref{eq: strichartz_asym} up to a loss of $N_2^\epsilon$, where $\epsilon>0$ is arbitrarily small, as a consequence of the decoupling argument used in the proof of Theorem~\ref{theorem:L^6 Strichartz}. We include a proof in Appendix~\ref{sec:appendix A} for completeness. While \eqref{eq: strichartz_asym} does not enter into our proof of small-mass global well-posedness, it is natural to ask whether estimates of this type could help remove the small-mass assumption in the defocusing case. We leave this as an open problem for future work.

To summarize, in this paper we follow a similar approach to ~\cite{Schippa_mass_critical,McConnell_Mass_critical_d_1} to obtain small-mass global well-posedness of \eqref{eq:NLS} in $H^s(\T)$ for $s>0$. Our main contributions are twofold: (i) we prove \eqref{eq:derivative_free_L6} for longer times \(T = \frac{1}{(\log N)^c} \) for some \(c \gg 1\) and (ii) we do not rely on resonant decompositions of the modified energy or on refined bilinear/trilinear Strichartz estimates — our proof depends solely on bounding the first modified energy.
    
\subsection{Proof outline}
Theorem \ref{theorem:L^6 Strichartz} follows from a lossless small-cap decoupling estimate for functions whose Fourier support lies in an $R^{-1}$ neighborhood of the parabola. Roughly speaking, one decomposes the Fourier support into caps $\gamma$ of length $1/N$ and seeks to establish an $L^6$ decoupling estimate without any logarithmic loss. Accordingly, let
\[
f=\sum_\gamma f_\gamma
\]
be such a decomposition where each $f_\gamma$ has essentially constant amplitude on $B_R$. We refer the reader to \cref{sec:2} for the precise geometric setup and notation. We prove the following lossless small-cap decoupling result
\begin{equation}
\label{eq: main_L6_dec_intro}
    \|f\|_{L^6(B_R)}^6
\lesssim
N^2R
\left(\sum_\gamma a_\gamma^2\right)^3,
\end{equation}
where $\sqrt{R}(\log N)^{10000}\leq N\leq R$ and $a_\gamma$ denotes the amplitude of $f_\gamma$ on $B_R$ (see Theorem \ref{main} for the precise formulation).

The proof of \eqref{eq: main_L6_dec_intro} begins with a standard pigeonholing argument which reduces matters to the case where $a_\gamma \sim 1$ or $a_\gamma=0$ for each $\gamma$. The main step is to estimate the size of the superlevel sets
\[
U_a(f):=\{x\in B_R : \frac{a}{2}\leq \abs{f(x)}\leq a\}.
\]
More precisely, we prove that there exist a constant $C>1$ such that for all $R/N^2 \lesssim (\log N)^{-2C}$, one has
\begin{equation}
\label{eq: main decoupling estimate on Ua}
|U_a|a^6
\lesssim
\max\left(
\log\left(\frac{\lambda}{a}\right)^{O(1)}
\frac{a}{\lambda},
(\log N)^{-C}
\right)
N^2R\lambda^3,
\end{equation}
where $\lambda$ denotes number of caps $\gamma$ such that $\abs{f_\gamma}\sim 1$. The estimate in \eqref{eq: main_L6_dec_intro} then follows by summing over dyadic values of $a$.

To establish \eqref{eq: main decoupling estimate on Ua}, we mainly follow the proof of \cite[Theorem 4]{amplitude}. As a corollary to their theorem, we see that up to a $\log(N)^{O(1)}$ loss, the only possible sharp contribution to $\|f\|^6_{L^6}$ comes from $|U_\lambda|\lambda^6$, i.e., the highest possible peaks. In \cite{amplitude}, the result was first proved for the broad case and a broad--narrow argument was then used to obtain the general case. To improve their argument and remove the $\log N$ loss, we need to overcome two obstacles: (i) replace the $\log N$ loss in the broad estimate by a $\log \lambda/a$ loss and (ii) remove any potential $\log N$ losses coming from the broad--narrow argument. 

The broad estimate applies when the caps $\Sigma_1$ and $\Sigma_2$ are comparable in size and separated by the same scale. Following the approach of Guth and Maldague \cite{amplitude}, we first establish a broad estimate using the high--low method and then combine it with a bilinear restriction estimate. By the high--low method, we mean an argument based on distinguishing whether the functions are dominated by high or low frequencies. In practice, this involves a dyadic decomposition in the frequency domain and treating the resulting pieces differently: the low-frequency part is controlled in $L^\infty$, while the high-frequency contributions are handled using almost orthogonality. To overcome the first obstacle and replace the $\log N$ loss arising at each scale by the smaller factor $\log \frac{\lambda}{a}$, we carefully restrict the number of high--low scales that contribute to the superlevel set estimate.

To overcome the second obstacle, we develop a refined broad--narrow argument. The broad--narrow decomposition has become a canonical tool in harmonic analysis and appears, for instance, in the proof of the $\ell^2$ decoupling theorem \cite{Bourgain_Demeter_decoupling} and in \cite{amplitude}. However, a direct implementation of the standard broad--narrow argument from Guth and Maldague \cite{amplitude} would still incur a $\log N$ loss, due to the presence of $\log N$ distinct broad and narrow scales. Our refinement avoids this loss by exploiting the fact that contributions coming from sufficiently separated broad--narrow scales interact only weakly. More precisely, we prove an asymmetric superlevel estimate that deals with intersections of superlevel sets associated with functions $f_1$ and $f_2$ whose respective Fourier supports are contained in caps $\Sigma_1$ and $\Sigma_2$ such that the length of the shorter cap is much shorter than the distance between them. Roughly speaking, for caps $\Sigma_1$ and $\Sigma_2$ that are at least $l$ separated and for $a,b$ in the appropriate ranges (see Lemma \ref{asym} for a precise statement), one has
\begin{equation}
\label{eq: asym_superlevel}
    \abs{U_{a}(f_{1}) \cap U_{b}(f_{2})} \lessapprox \max\left(\frac{l(\Sigma_2)}{l}, \frac{\sqrt{R}}{N}\right)\frac{N^2R\lambda(\Sigma_1)\lambda(\Sigma_2)^2}{a^2b^4},
\end{equation}

where $l(\Sigma_j)$ denotes the length of the cap $\Sigma_j$, $\lambda(\Sigma_j)$ is the number of caps $\gamma \subset \Sigma_j$ such that $\abs{f_\gamma} \sim 1$ and $\lessapprox$ should be understood as bound up to logarithmic losses. The additional gain $\max\left(\frac{l(\Sigma_2)}{l}, \frac{\sqrt{R}}{N}\right)$ is precisely what allows us to control interactions between distant broad--narrow scales and thereby eliminate the logarithmic losses. The asymmetric estimate is established using a high--low analysis similar to that employed in the proof of the broad estimate. In addition to its role in the proof of Theorem~\ref{theorem:L^6 Strichartz}, \eqref{eq: asym_superlevel} may have further applications beyond the present work. For instance, it is a key input in recovering the asymmetric Strichartz estimate in Appendix~\ref{sec:appendix A}.

We now turn to the proof of Theorem \ref{theorem:main_gwp_theorem}. The proof of Theorem~\ref{theorem:main_gwp_theorem} is based on the $I$-method and broadly follows the approach developed in \cite{mass_critical_gwp_staffilani_tzirakis,Schippa_mass_critical,McConnell_Mass_critical_d_1}. As we are assuming that the mass is small, it suffices to consider the defocusing case ($\mu=1$) by the Gagliardo-Nirenberg \cite{Nirenberg} inequality. Next, by exploiting the rescaling symmetry of the equation, we analyze solutions $u^\lambda$ on the space-time domain $\T_\lambda \times \R$. We introduce a Fourier multiplier $I$ which acts as the identity on low frequencies $\abs{k} \lesssim N$ and damps high frequencies $\abs{k}\gg N$ at the rate $\left(\abs{k}/N \right)^{s-1}$, where $N$ is a large parameter to be determined. We define the $I$-system by applying the $I$-operator to \eqref{eq:NLS} and the modified energy $E^1(u) := E(Iu)$ is then naturally interpreted as the energy of the $I$-system. Under a small mass assumption, we establish a local existence result for the $I$-system for time intervals of length $\lambda^2 \log(\lambda N)^{-c}$ where $c>0$ is a fixed constant. The main challenge is to extend this local result to arbitrary times. This is achieved by proving that $E^1(u(t))$ is "almost conserved", in the sense that the time-averaged derivative of $E^1$ exhibits mild growth. This slow growth of the modified energy allows us to iterate the local existence argument and propagate the $H^s$-bound for $u^\lambda$ over arbitrarily long times. Undoing the rescaling then yields polynomial-in-time bounds for the original solution $u$, which establishes global well-posedness.

\subsection{Organization of the paper.}
Our paper is organized as follows: in \cref{sec:2} we introduce the geometric setup for decoupling, in \cref{sec:3} we prove Theorem \ref{theorem:L^6 Strichartz} using the high-low method and a refined broad-narrow argument, in \cref{sec:4} we transfer our Strichartz estimates to $Y^s$ spaces and in \cref{sec:5} we give the proof of Theorem \ref{theorem:main_gwp_theorem} using the $I$-method. In Appendix~\ref{sec:appendix A} we give an alternative proof of the asymmetric Strichartz estimate first established in \cite{McConnell_Mass_critical_d_1}.

\subsection{Notation}
We  use $A \lesssim B$ to indicate an estimate of the form $A \leq CB$ for some constant $C>0$. When we want to stress that the constant depends on a parameter, say $p$, we write  $A \lesssim_p B$. If  both $A \lesssim B$ and $B \lesssim A$ hold, we write $A \sim B$. Furthermore, $A \ll B$ signifies that $A \leq cB$ for some small constant $0<c<1$ and we use $a\pm$ to denote $a \pm \epsilon$ where $0<\epsilon \ll 1$. When we say that a property $Q(A,B)$ holds whenever $A \ll B$, we mean that that there exist a small enough absolute constant $0<c<1$ such that $Q(A,B)$ holds whenever $A \leq cB$. For any measurable $A \subset \Rd$, we denote its characteristic function by $\chi_A$ and its Lebesgue measure by $\abs{A}$. Finally, $\mathbb{N}$ denotes the set of natural numbers and we write $\mathbb{N}_0:=\mathbb{N} \cup \{0\}$ for the set of non-negative integers.

\section{Geometric Setup and Preliminary Lemmas for Decoupling}\label{sec:2}

Let's first go through some definitions and standard results in decoupling. 
Let 
\[
\P := \{(x,x^2) : x \in [0,1]\}
\]
denote the truncated parabola, and let $\Gamma:=N_{1/R}(\P)$ be its $1/R$-neighborhood. 
We write $B_r(x)$ for the ball of radius $r$ centered at $x \in \R^2$, and $B_r$ for the ball of radius $r$ centered at the origin. 
We now record some basic notations.

\begin{definition}
A \emph{cap} $\tau$ of length $l(\tau)$ centered at $c(\tau)$ is defined as
\[
\tau := N_{1/R}(\P) \cap B_{l(\tau)}(c(\tau)).
\]
The \emph{direction} $d(\tau)$ of $\tau$ is the direction of the normal to $\P$ at $c(\tau)$.
\end{definition}

\begin{definition}
We say that a set $S$ has \emph{rectangular shape} if there exist rectangles. 
\[
R_1 \subseteq S \subseteq R_2
\]
such that $|R_2| \leq 100\,|R_1|$ where $|R|$ denotes the area of the rectangle $R$. 

Given a rectangular shape $S$, we say that a function $W$ \emph{decays rapidly off $S$} if
\[
W(x) \lesssim \left(\frac{1}{1+\operatorname{dist}(x,S)}\right)^{100}.
\]

The \emph{dual set} $S^*$ of $S$ is defined by
\[
S^* := \{x \in \R^2 : |x \cdot y| \leq 1 \;\; \text{for all } y \in S \}.
\]

We use $W_S$ to denote an ($L^1$-normalized) \emph{weight function} on $S$, i.e.\ a nonnegative function such that
\begin{enumerate}
    \item $W_S \sim \tfrac{1}{|S|}$ on $S$,
    \item $\int W_S \sim 1$,
    \item $W_S$ decays rapidly off $S$,
    \item $\widecheck{W_S}$ and $\widehat{W_S}$ are supported in $S^*$.
\end{enumerate}
By $\|f\|_{\L^p(W_S)}$, we mean $\left(\int |f|^p W_S\right)^{1/p}$. 
\end{definition}
By the argument in section 2.1 in \cite{amplitude}, given any set $S$ with rectangular shape, there exists a weight function $W_S$ on $S$. 
\begin{lemma}(Locally Constant)\label{locally constant}
    Let $S$ be a rectangular shape and suppose that $\spp \hat{g} \subset S^*$, then $g$ is locally constant on $S$. In other words, for any translation $S'$ of $S$, 
    $$\|g\|_{L^\infty (S')} \leq \|g\|_{\L^1 (W_{S'})}$$
    for some weight function $W_{S'}$ on $S'$. 
\end{lemma}
\begin{proof}
There exists a function $\eta_{S^*}$ such that $\eta_{S^*} \sim 1$ on $S^*$ and is supported on $cS^*$ and $\check{\eta_{S^*}}\sim \frac{1}{|S|}$ on $S$ and decays rapidly off $S$. Thus,
$g(x)=g*\check{\eta_{S^*}}$. 
    Therefore, for any $x \in S'$, we have 
    \begin{eqnarray*}
    |g(x)|
       &=&  |g *\check{\eta}_{S^*}(x)|\\
       &\leq& \int |g(x-y)||\check{\eta}_{S^*}(y)|  dy\\
    &\leq& \int |g|W_{S'}
    \end{eqnarray*}
We put $W_{S'}(y):\sim \max_{x \in S'}|\check{\eta}_{S^*}(x-y)|$ to be the weight function. Note that,  $W_{S'} \sim \frac{1}{|S|}$ on $S+S'\sim S'$ and decays rapidly off $S+S'$.   
\end{proof}
 
\begin{definition}
    Given a function $f$ with $\spp \hat{f} \subseteq N_{1/R}(\P^1)$, we define $f_\tau$, the Fourier restriction of $f$ to $\tau$, to be 
    $$f_\tau: = f* \check{\chi}_\tau.$$
\end{definition}
\begin{definition}
    Let $N\gg 1$ and $R\gg 1$ such that 
    $N \leq R< N^2$. 
    We say that a function $f$ with $\spp \hat{f} \subset \N_{1/R}(\P)$ satisfies $\Cn$ if   $f=\sum_{\gamma} f_\gamma$ such that each $f_\gamma$ has fourier support on a cap $\gamma$ of length $1/N$ and for each $\gamma$ there exists a real number $a_\gamma$ such that $\frac{a_\gamma}{100} \leq|f_\gamma| \leq a_\gamma$ on $B_R$ and $f_\gamma$ decays rapidly off $B_R$. We say that $f$ satisfies $\Cn_0$ if for all $\gamma$ we have either $a_\gamma=1$ or $a_{\gamma}=0$. 
\end{definition}

We shall first prove our main result for the special case in which $f$ satisfies $\Cn_0$, and then extend it to the general case. 
Let $f$ satisfies $\Cn_0$. 
Define $l_k=3^{-k}$. Then, $l_{k-1}/l_k=3$ and 
$$R^{-1/2}=l_M <l_{M-1} <\cdots < l_1 <l_0=1.$$ Let $c$ be the smallest number such that $l_c> N/R$. 
For $1 \leq k \leq M-1$, We cut $\P$ into caps $\tau_k$ of length $l_k$ and for $c \leq k \leq M-1$, we cut $\P$ into caps $\omega_k$ of length $l_k^{-1}R^{-1}$. Also, we cut $\P$ into caps $\theta$ of length $R^{-1/2}$ and into caps $\gamma$ of length $1/N$. 
Given any cap $\tau$, we define 
\begin{align*}
    \lambda(\tau)&:=\#\{\gamma: \gamma \subset \tau, |f_\gamma|\sim 1\}, \\
    \lambda(l_k)&:=\max_{\tau_k}\#\{\gamma \subset \tau_k: |f_\gamma|\sim 1 \text{ on } B_R\}, \\
\lambda(l_k^{-1}R^{-1})&:=\max_{\omega_k}\#\{\gamma \subset \omega_k: |f_\gamma|\sim 1 \text{ on } B_R\}
\end{align*}

Also, given a cap $\theta$, we define 
$$\lambda(\theta)=\#\{\gamma \subset \theta: |f_\gamma|\sim 1 \text{ on }B_R\}.$$

For any $k$ and any $\tau_k \in \P$, we tile $B_R$ by $l_kR \times R$ tubes, $U_{\tau_k}$, whose long axis is in the direction $d(\tau_k)$. We call this set of tubes $\U_{\tau_k}$. For $U_{\tau_k} \in \U_{\tau_k}$, let $W_{U_k}$ be a weight function on $U_{\tau_k}$.

Let $\{\eta_k\}_{0 \leq k \leq M} \cup \{\eta_{l_k}\}_{0 \leq k \leq M} $ be a smooth partition of unity on $B_1$ such that 
\begin{enumerate}
\item $$\spp \eta_{l_k} \subseteq \left\{x: \frac{l_{k}}{2} \leq |x| \leq 2l_{k-1}\right\} \text{ for } 1  \leq k \leq M$$
    \item $$\spp \eta_k \subseteq \left\{x: \frac{R^{-1}l_{k-1}^{-1}}{2} \leq |x| \leq 2R^{-1}l_{k}^{-1}\right\} \text{ for } 1  \leq k \leq M$$
and 
$\spp \eta_0 \subseteq \left\{x:  |x| \leq 2R^{-1}\right\}.$
\item $\check{\eta}_k(x)\sim R^{-2}l_k^{-2}$ on $B_{Rl_{k}}$ and decays rapidly off $B_{Rl_{k-1}}$ and $\check{\eta}_{l_k}(x)\sim_K l_k^{2}$ on $B_{l_{k-1}^{-1}}$ and decays rapidly off $B_{l_{k}^{-1}}$
\item $\int |\check{\eta}_k| \sim 1$ and $\int |\check{\eta}_{l_k}| \sim 1$
\end{enumerate}

In this paper, we shall look at the blurred version of $f$ at many different scales.
For $1\leq k \leq M$, define $g_k=\sum_{\tau_k} |f_{\tau_k}|^2$ and define $g=\sum_{\theta}|f_\theta|^2$ and $g_{\tau}=\sum_{\theta \subset \tau}|f_\theta|^2$. Then, $g_k$ is basically $|f|^2$ blurred at scale $l_k^{-1}$, $g$ is $f$ blurred at scale $R^{1/2}$ and $g_\tau$ is $|f_\tau|^2$ blurred at scale $R^{1/2}$. Moreover, we can view $g$ as the sum of the characteristic functions of a bunch of $R \times R^{1/2}$ tubes. 

Then, intuitively $|U_{\tau_k}|\|f_{\tau_k}\|_{\L^2(W_{U_k})}^2\sim |U_{\tau_k}|\|g_{\tau_k}\|_{\L^1(W_{U_k})}$ counts the number of $R\times R^{1/2}$ tubes contained in $U_{\tau_k}$. 
Now, let's prove some properties of $\|f\|^2_{\L^2(W_{U_{\tau_k})}}$.

\begin{lemma}(Local Orthogonality)\label{localorth}
$$\|g_{\tau_k}\|_{\L^1(W_{U_{\tau_k})}}\sim \|f_{\tau_k}\|_{\L^2(W_{U_k})}^2 \lesssim \begin{cases}
\lambda(l_k)\lambda(l_k^{-1}R^{-1}) \text{ if } k \geq c\\
    \lambda(l_k)\text{ if } k < c
\end{cases}
$$    
\end{lemma}
\begin{proof}
By Plancherel's theorem, we have
    \begin{eqnarray*}
        \int |f_\theta|^2 W_{U_\tau}  &=& \int \hat{f_\theta}(\xi) \hat{\bar{f_\theta}}*\hat{W}_{U_{\tau_k}}(\xi) \, d\xi\\
        &=&\sum_{\omega \subset \theta} \sum_{\omega' \subset \theta} \int \hat{f_\omega}(\xi) \hat{\bar{f_{\omega'}}}*\hat{W}_{U_{\tau_k}}(\xi) \, d\xi\\
    \end{eqnarray*}
    Note that the summand from $(\omega, \omega')$ is non-zero only if $\omega \cap (-\omega' +U_{\tau_k}) \not=\emptyset$. Thus, we need $|c(\omega)-c(\omega')| \leq R^{-1}l_k^{-1}$. In this case, we say $\omega \sim \omega'$. As $l(\omega )\geq R^{-1}l_k^{-1}$ for each $\omega$ there is $\sim 1$ many $\omega'$ such that $\omega \sim \omega'$. Thus, 
    \begin{eqnarray*}
        \int |f_\theta|^2 W_{U_\tau} 
        &=&\sum_{\omega \subset \theta} \sum_{\omega' \sim \omega} \int \hat{f_\omega}(\xi) \hat{\bar{f_{\omega'}}}*\hat{W}_{U_{\tau_k}}(\xi) \, d\xi\\
        &=&\sum_{\omega \subset \theta} \sum_{\omega' \sim \omega} \int f_\omega\bar{f_{\omega'}}W_{U_{\tau_k}} \text{ by Plancherel }\\
        &\lesssim &\sum_{\omega \subset \theta} \int |f_\omega|^2W_{U_{\tau_k}} \text{ by Cauchy-Schwartz}\\
    \end{eqnarray*}

    If $k \leq c$, then $1/N \geq l_k^{-1}R^{-1}$, so 
    $$\int |f_\theta|^2 W_{U_\tau} \lesssim \sum_{\gamma \subset \theta} \int |f_\gamma|^2W_{U_{\tau_k}} \lesssim (\# \gamma \subseteq \theta) |U_{\tau_k}|.$$
    Thus, 
    $$\|g_{\tau_k}\|_{\L^1(W_{U_{\tau_k}})} \lesssim (\# \gamma \subseteq \tau_K) \sim \lambda(l_k).$$
    If $k > c$, then 
    $$\int |f_\theta|^2 W_{U_\tau} \lesssim \sum_{\omega \subset \theta} \int |f_\omega|^2W_{U_{\tau_k}} $$
    where $\omega$ are caps of length $R^{-1}l_k^{-1}$. 
    Note that $\|f_\omega\|_\infty \leq (\# \gamma \subseteq \omega) \sim \lambda(R^{-1}l_k^{-1})$.
    Thus, 
    $$\sum_{\omega \subseteq \theta} \int |f_\omega|^2W_{U_{\tau_k}} \sim (\# \omega \subseteq \theta) (\# \gamma \subseteq \omega )^2|U_{\tau_k}|.$$
    Therefore, 
    \begin{equation*}
     \|g_{\tau_k}\|_{\L^1(W_{U_{\tau_k}})} \lesssim (\# \gamma \subseteq \tau)(\# \gamma \subseteq \omega) \sim \lambda(l_k)\lambda(l_k^{-1}R^{-1}). \qedhere    
    \end{equation*}

\end{proof}

Next, we prove a slightly generalized version of the bilinear restriction. 
\begin{lemma}(Bilinear Restriction)\label{bilresg}
Let $T$ be a $L_1$ by $L_2$ rectangle. Suppose that $d(\tau_1)$ is perpendicular to the side of length $L_1$ and $d(\tau_1)$ is perpendicular to the side of length $L_2$. Let $\omega_1$ be caps of length $\geq L_1^{-1}$ and $\omega_2$ be caps of length $\geq L_2^{-1}$
     Then, 
    $$\int_{T}|f_{\tau_1}|^2|f_{\tau_2}|^2 \lesssim \int \left(\sum_{\omega_1 \subset \tau_1} |f_{\omega_1}|^2\right) \left(\sum_{\omega_2 \subset \tau_2} |f_{\omega_2}|^2\right)W_T.$$
\end{lemma}
\begin{proof}
    The proof is essentially a slight variation of the bilinear restriction proof. Let $\eta_T$ be a smooth cut-off function on $T$ that has Fourier support in $T^*$. Then, 
    \begin{eqnarray*}
       && \int_{T} |f_{\tau_1}|^2 |f_{\tau_2}|^2 \\
       & \lesssim & \int |f_{\tau_1}|^2 |f_{\tau_2}|^2 \eta_T\\
       & = & \int |\hat{f}_{\tau_1}|^2 |\hat{f}_{\tau_2}|^2 \eta_T \text{ by Plancherel }\\
       & \leq & \sum_{\omega_1} \sum_{\omega_1'} \sum_{\omega_2} \sum_{\omega_2'} \int \hat{f}_{\omega_1}*\hat{\bar{f}}_{\omega_1'}(\xi) \hat{f}_{\omega_2}\ast \hat{\bar{f}}_{\omega_2'}\ast \hat{\eta}_T(\xi)
    \end{eqnarray*}
    Note that $\hat{f}_{\omega_1}*\hat{\bar{f}}_{\omega_1'}$ has support in $\omega_1-\omega_1'$ and $\hat{f}_{\omega_2}*\hat{\bar{f}}_{\omega_2'}*\hat{\eta}_T$ has support in $\omega_2-\omega_2'+T$. 
Thus, in order for $$\int \hat{f}_{\omega_1}*\hat{\bar{f}}_{\omega_1'}(\xi) \hat{f}_{\omega_2}\ast \hat{\bar{f}}_{\omega_2'}\ast \hat{\eta}_T(\xi)\not=0,$$ we need 
$$(\omega_1-\omega_1') \cap (\omega_2-\omega_2'+T)\not=\emptyset.$$

Let $(\xi_1,\xi_1^2),(\xi_1',\xi_1'^2), (\xi_2,\xi_2^2), (\xi_2',\xi_2'^2)$ be the center of $\omega_1, \omega_1', \omega_2, \omega_2'$. 
Note that $\omega_1-\omega_1'$ is a $R^{-1/2}L_2 \times R^{-1}$ tube centering at $(\xi_1-\xi_1', \xi_1^2-\xi_1'^2)$ and $\omega_2-\omega_2'+S$ is a $R^{-1/2}L_2 \times R^{-1/2}$ tube centering at $(\xi_2-\xi_2', \xi_2^2-\xi_2'^2)$ 
Then, 
we need 
$$|(\xi_1-\xi_1')-(\xi_2-\xi_2')| \leq R^{-1/2}L_2$$
 $$|(\xi_1^2-\xi_1'^2)-(\xi_2^2-\xi_2'^2)| \leq R^{-1/2}L_1.$$

Thus, 
\begin{eqnarray*}
    R^{-1/2}L_1 &\geq& |(\xi_1^2-\xi_1'^2)-(\xi_2^2-\xi_2'^2)| \\
    &=&|(\xi_1+\xi_1'-\xi_2-\xi_2')(\xi_2-\xi_2')+(\xi_1+\xi_1')((\xi_1-\xi_1')-(\xi_2-\xi_2'))| \\
    &\geq&|(\xi_1+\xi_1'-\xi_2-\xi_2')(\xi_2-\xi_2')|-|(\xi_1+\xi_1')((\xi_1-\xi_1')-(\xi_2-\xi_2'))| \\
    &\geq&|(\xi_1+\xi_1'-\xi_2-\xi_2')(\xi_2-\xi_2')|-C R^{-1/2}L_2 \\
\end{eqnarray*}
Also, note that $(\xi_1+\xi_1'-\xi_2-\xi_2') \gtrsim 1$ as $\Gamma_1$ and $\Gamma_2$ are $\sim 1$ distance apart. Thus, we have $|\xi_2-\xi_2'| \lesssim R^{-1/2}L_1$. In this case, we say that $\omega_2' \in G(\omega_2)$. If we fix $\omega_1,\omega_2, \omega_2'
$, as $|(\xi_1-\xi_1')-(\xi_2-\xi_2')| \leq R^{-1/2}L_2$, there is $\lesssim 1$ choice for $\omega_1'$. In this case, we call $\omega_1' \in G(\omega_1,\omega_2, \omega_2')$. 
Thus, 
\begin{align*}
   & \int_{T} |f_1|^2 |f_2|^2 \\
   & \leq  \sum_{\omega_2} \sum_{\omega_2' \in G(\omega_2)}\sum_{\omega_1} \sum_{\omega_1' \in G(\omega_1,\omega_2, \omega_2')}  \int \hat{f}_{\omega_1}*\hat{\bar{f}}_{\omega_1'}(\xi) \hat{f}_{\omega_2}\ast \hat{\bar{f}}_{\omega_2'}\ast \hat{\eta}_T(\xi)\\
   & =  \sum_{\omega_2} \sum_{\omega_2' \in G(\omega_2)}\sum_{\omega_1} \sum_{\omega_1' \in G(\omega_1,\omega_2, \omega_2')}  \int f_{\omega_1} \bar{f}_{\omega_1'} f_{\omega_2} \bar{f}_{\omega_2'} \eta_T \text{ by Plancherel}\\
   & =  \sum_{\omega_2} \sum_{\omega_2' \in G(\omega_2)}\sum_{\omega_1} \sum_{\omega_1' \in G(\omega_1,\omega_2, \omega_2')}  \int  (|f_{\omega_1}|^2+ |{f}_{\omega_1'}|^2) f_{\omega_2} \bar{f}_{\omega_2'} \eta_T\\
    & =  \sum_{\omega_2} \sum_{\omega_2' \in G(\omega_2)} \int_S (\sum_{\omega_1} |f_{\omega_1}|^2) f_{\omega_2} \bar{f}_{\omega_2'} \eta_T \text{ as }|G(\omega_1,\omega_2, \omega_2')| \lesssim 1 \\
    & =  \int  (\sum_{\omega_1} |f_{\omega_1}|^2) (\sum_{\omega_2} |f_{\omega_2}|^2)\eta_T  \text{ as }|G(\omega_2)| \lesssim 1 \qedhere 
\end{align*}
\end{proof}
Next, we recall some high-low lemmas from \cite{amplitude} that will be useful in our proof. 

\begin{lemma}\label{break}
    $$\sum_{\tau' \subset \tau} |f_{\tau'}|^2*\check{\eta}_{\leq r} \lesssim \sum_{\tau''\subset \tau} |f_{\tau''}|^2 *\check{\eta}_{\leq r}$$
    and 
    $$\sum_{\tau' \subset \tau} |f_{\tau'}|^2*\check{\eta}_{\sim r} \lesssim \sum_{\tau''\subset \tau} |f_{\tau''}|^2 *\check{\eta}_{\sim r}$$
   whenever $r\leq l(\tau'')\leq l(\tau')$. 
\end{lemma}
\begin{proof}
 It is enough to prove that 
 $|f_\tau|^2*\check{\eta}_{\leq r} \lesssim \sum_{\tau'\subset \tau} |f_{\tau'}|^2 *\check{\eta}_{\leq r}$ whenever $r<l(\tau')<l(\tau)$.
 Note, we can write 
$$|f_{\tau}|^2*\check{\eta}_{\leq r}=\sum_{\tau_1'} \sum_{\tau_2'} (f_{\tau_1'}\overline{f_{\tau_2'}})*\check{\eta}_{\leq r}$$
 Note that $(f_{\tau_1'}\overline{f_{\tau_2'}})*\check{\eta}_{\leq r}\equiv 0$ if $(\tau_1' -\tau_2') \cap B(0,r)=\emptyset$. Thus, for each $\tau_1'$ there are only finitely many $\tau_2'$ such that $(f_{\tau_1'}\overline{f_{\tau_2'}})*\check{\eta}_{\leq r} \neq 0$. Thus, by Cauchy-Schwartz, 
$$|f_{\tau}|^2*\check{\eta}_{\leq r}\lesssim \sum_{\tau'\subset \tau}  |f_{\tau'}|^2 *\check{\eta}_{\leq r}.$$
The proof for the $\check{\eta}_{\sim r}$ case is basically the same. 
\end{proof}
The following lemma essentially states that the low-frequency components have small $L^\infty$ norms. 
 \begin{lemma}\label{Low}(Low Lemma)
        Let $R^{-1} \leq r \leq R^{-1/2}$. 
      $$|g_{\tau}*\check{\eta}_{\leq r}|\lesssim \begin{cases}
            \lambda(\tau) \text{ if }r\leq 1/N\\
            \lambda(\tau)rN \text{ if }r>1/N
        \end{cases}$$      
    \end{lemma}
    \begin{proof}
    Note that $\check{\eta}_{\leq r}(x) \sim r^2$ on a ball of radius $1/r$.
    Note that $||f_\theta|^2*\check{\eta}_{\leq r}|$ are locally contained on $R \times 1/r$ tubes $U$. 
    Thus, 
    $$\||f_\theta|^2 * \check{\eta}_{\leq r}(x)\|_{\infty}\sim \max_{U}\||f_\theta|^2 * \check{\eta}_{\leq r}(x)\|_{L^1(U)} \lesssim \begin{cases}
            \lambda_0 \text{ if }r\leq 1/N\\
            \lambda_0rN \text{ if }r>1/N
        \end{cases}.$$
    Note that the number of $\theta$'s is $\lambda(\tau)/\lambda_0$ so that 
    \begin{equation*}
    |g*\check{\eta}_{\leq r}|\lesssim \begin{cases}
            \lambda(\tau) \text{ if }r\leq 1/N,\\
            \lambda(\tau)rN \text{ if }r>1/N.
        \end{cases}    \qedhere 
    \end{equation*}
    \end{proof}
The following lemma exploits the transversality for the high-frequency part. 
\begin{lemma}(high-low)\label{hl}
   Let $\tau_0$ be a cap. Let $R^{-1}<s<R^{-1/2}$. Then, 
    $$\int |g_{\tau_0}*\check{\eta}_{\sim s} |^2 \lesssim \sum_{\substack{\tau\subset \tau_0\\ l(\tau)=(sR)^{-1}}}\sum_{U \in \U_\tau}|U|\|f_\tau\|^4_{\L^2(U)}.$$
\end{lemma}
\begin{proof}
    Define $g_{\tau}=\sum_{\theta \subset \tau} |f_\theta|^2$. Then, $g_{\tau_0}=\sum_{\tau\subset \tau_0}g_\tau$. Note that the support of $\hat{g_\tau}$ is contained in $\cup_{\theta \subset \tau} \theta-\theta$ which is a  $R^{-1/2} \times R^{-1/2}l(\tau)$ rectangle in the same direction as $\tau$. Note that unless $dist(\tau,\tau')\lesssim l(\tau)$, 
    $$\hat{g}_\tau \cap \hat{g}_{\tau'} \cap \{\xi: |\xi|>s\}=\emptyset$$
    so $\int |g_\tau * \check{\eta}_{\sim s}||g_{\tau'} * \check{\eta}_{\sim s}|=0$. Thus, 
    \begin{align*}
    \int_{B_R} |g*\check{\eta}_{\sim s} |^2 &=\int_{B_R} |\sum_{\tau} g_{\tau}*\check{\eta}_{\sim s} |^2 \\
    &\lesssim  \sum_{\tau}\int_{B_R}| g_{\tau}*\check{\eta}_{\sim s} |^2 \\
    &\lesssim  \sum_{\tau}\sum_{U \subset \U_{\tau}}\int_{U}| g_{\tau}*\check{\eta}_{\sim s} |^2 \\
    &\lesssim  \sum_{\tau}\sum_{U \in \U_{\tau}}|U|\|g_\tau*\check{\eta}_{\sim s}\|^2_{\L^2(U)}\\
    &\lesssim  \sum_{\tau}\sum_{U \in \U_{\tau}}|U|\|g_\tau*\check{\eta}_{\sim s}\|^2_{\L^1(W_U)} \text{ as }g_\tau*\check{\eta}_{\sim s} \text{ is locally constant on }U\\
    &\sim  \sum_{\tau}\sum_{U \in \U_{\tau}}|U|\|f_\tau\|^4_{\L^2(W_U)} \text{ by local orthogonality} \qedhere
    \end{align*}
\end{proof}

\section{Proof of Theorem \ref{theorem:L^6 Strichartz}}\label{sec:3}
In this section, we prove Theorem~\ref{theorem:L^6 Strichartz}. By standard reductions, it suffices to establish the following lossless small-cap decoupling estimate.

\begin{theorem}[Lossless small-cap decoupling]
\label{main}
Suppose that $\sqrt{R}(\log N)^{10000} \leq N \leq R$ and Let $f$ be a function with $f=\sum_\gamma f_\gamma$ with $\alpha_\gamma/100 \leq |f_\gamma| \leq \alpha_\gamma$ on $B_R$ for some $\alpha_\gamma \in \R$. Then, 
    $$\|f\|_{L^6(B_R)}\lesssim \left(\sum_{\gamma} \|f_\gamma\|^2_{L^6(B_R)}\right)^{1/2}\left(\frac{N^2}{R}\right)^{1/6}.$$
\end{theorem}

The proof of Theorem~\ref{main} is divided into two parts. First, we reduce matters to the special case in which $f$ satisfies $\Cn_0$. We then prove this special case.

\subsection{Reducing to the special case}

In order to prove the Theorem \ref{main}, we first prove that the general result follows from a special case when $f$ satisfies $\Cn_0$. 
Define the superlevel set $$U_a(f):=\{x \in B_R: a/2 \leq |f(x)| \leq a \}.$$
\begin{lemma}\label{special0}
Suppose that $f$ satisfies $\Cn_0$, then 
    $$|U_a(f)|a^6 \lesssim  \log\left(\frac{\lambda}{a}\right)^{O(1)}\frac{a}{\lambda} N^2R\lambda^3.$$
\end{lemma}

Now let's prove that Theorem \ref{main} follows from Lemma \ref{special0}. 
\begin{definition}
    We say that a function $f$ satisfies condition $\con_0(\lambda)$ if $f=\sum_{\gamma}f_\gamma$ such that $\frac{1}{100}\leq |f_r(x)| \leq 1$ for $x \in B_R$ and $f_r$ decays rapidly off $B_R$ or $f_\gamma \equiv 0$ and we let $\lambda$ be the number of non-zero $\gamma$'s.  
    We say that a function $f$ satisfies $\con_1(\alpha,\lambda)$ if $f/\alpha$ satisfies $\con_0(\lambda)$.   
\end{definition}

By lemma \ref{special0} and a rescaling argument, if $f$ satisfies $\Cn_1(\alpha,\lambda)$, then 
\begin{equation}\label{con1}
  |U_a(f)|a^6 \lesssim \log\left(\frac{\alpha\lambda}{a}\right)^{O(1)} \frac{a}{\alpha\lambda}N^2R\alpha^6\lambda^3.  
\end{equation}

\begin{definition}
    We say that $f$ satisfies $\Cn_2(Q,S)$ if there exists a set $J \subset \Z$ with $|J|=S$ such that 
    $f=\sum_J f_j$ and each $f_j$ satisfied $\Cn_1(\alpha_j,\lambda_j)$ such that $Q/(100)^3 \leq \alpha_j^6\lambda_j^2 \leq  Q$ and $100^{j-1} \leq \alpha_j\lambda_j \leq 100^{j}$. 
\end{definition}
\begin{lemma}\label{beattri}
    Let $J \subset \Z$. If $h_k=2^k \chi_{S_k}$, then
    $$ \|\sum_{k \in J} h_k\|_r^r \lesssim \sum_k  \|h_k\|_r^r$$
    for any $r\geq 1$. 
    (Note that we do not require $S_k$ to be disjoint.)
\end{lemma}
\begin{proof}
    At each $x$, $\sum_k h_k (x) \lesssim \max_k h_k(x).$
    Thus, 
    \[\int (\sum h_k)^r \lesssim \int (\max_k h_k(x))^k \leq \int \sum_k h_k^r. \qedhere \]
\end{proof}
\begin{proposition}\label{cond2}
    Suppose that $f$ satisfies $\Cn_2(Q,S)$, then
    $$\|f\|_{L^6(B_R)}^6 \lesssim N^2RQS$$
\end{proposition}
\begin{proof}
    First define $f_{j,k}=100^k \chi_{U_{100^k}(f_j)}$. Note $f_{j,k}=0$ if $k>j$. 
    For any $l \geq 0$, by Lemma \ref{beattri}, we have
    \begin{eqnarray*}
        &&\|\sum_{j \in J}f_{j,j-l}\|^6_{L^6(B_R)}\\
        &\lesssim &\sum_{j \in J}\|f_{j,j-l}\|^6_{L^6(B_R)}\\
        &\lesssim &\sum_{j \in J}10^{-l}QN^2R.\text{ by equation }\eqref{con1}\\
        &=&10^{-l}QN^2RS.
    \end{eqnarray*}
    Summing over all the $l$'s we have
    \begin{align*}
        \|f\|_{L^6(B_R)}^6 &\sim \|\sum_{j\in J}\sum_{l \geq 0}f_{j,j-l}\|_{L^6(B_R)}^6\\
        &\lesssim \left(\sum_{l\geq 0}\|\sum_{j\in J}f_{j,j-l}\|_{L^6(B_R)}\right)^6 \text{ triangle inequality }\\
        &\lesssim \left(\sum_{l \geq 0} (10^{-l}QN^2RS)^{1/6}\right)^6 \\
        &\lesssim  N^2RQS \qedhere
    \end{align*}
\end{proof}
\begin{proposition}\label{general}
    Suppose $f=\sum_k f_k$ with $f_k$ satisfying $\con_2(Q_k, S_k)$ with $Q_k$ taking dyadic values. Then, 
    $$\|f\|_{L^6}^6 \leq \left(\sum_{\gamma} \|f_\gamma\|_{L^6}^2\right)^3 \left(\frac{N^2}{R}\right).$$   
\end{proposition}
\begin{proof}
    By triangle inequality and proposition \ref{cond2}, we have
    $$\|f\|_{L^6}\lesssim \sum_{k} \|f_k\|_{L^6} \lesssim (N^2R)^{1/6}\sum_k Q_k{S_k}^{1/6}.$$
    Thus, by Holder's inequality 
    \begin{eqnarray*}
    \|f\|_{L^6}^6 & \leq & (N^2R) \left(\sum_k Q_k{S_k}^{1/6}\right)^6\\
    & = & (N^2R) \left(\sum_k \left(Q_k^{1/3}{S_k}^{1/6}\right) Q_k^{2/3}\right)^6\\
    & \leq & (N^2R) \left(\sum_k Q_k^2{S_k}\right)\left(\sum_k Q_k^{4/5}\right)^5\\
    & \leq & (N^2R) \left(\sum_k Q_k^2{S_k}\right)^3
    \end{eqnarray*}
    as $Q_k$ are dyadic $\left(\sum_k Q_k^{4/5}\right)^5 \lesssim \left(\max_k Q_k\right)^4 \leq \left(\sum_k S_k Q_k^2\right)^2$.
\end{proof}
Theorem \ref{main} follows quickly from proposition \ref{general}. 
\begin{proof}[Proof of Theorem \ref{main}]
    First, we group all the $f_\gamma$ with similar $\alpha_\gamma$'s together. We write $f=\sum_{j\in J} f_j$ such that each $f_j$ satisfy $\con_1(100^j,\lambda_j)$ for some $\lambda_j$. 
    Then we group the $f_j$'s with similar $L_2$ norms. Define
    $$K_k:=\{j:  100^{k-1}\leq \alpha_j^2\lambda_j \leq 100^k\} \text{ and } f_k=\sum_{j \in K_k}f_{j}.$$
    Note that for each $j \in K_k$, we have $f_j$ satisfies $\Cn_1(\alpha_j,\lambda_j)$ with $\alpha_j=100^j$  and $100^{k-j-1}\leq \alpha_j\lambda_j \leq 100^{k-j}$. 
    Then, each $f_k$ satisfy $\con_2(100^{3k}, |K_k|)$ and $f=\sum_k f_k$. Thus, we can apply proposition \ref{general} to $f$ to get the result we want.
\end{proof}  
\subsection{Special Case}
The goal of this section is to prove the Lemma \ref{special0}, which follows quickly from the following Lemma.
\begin{lemma}\label{special00}
There exists a constant $c_0$ such that 
$$|U_a|a^6 \lesssim (\log N)^{c_0} \max\left(\frac{a^2}{\lambda^2},\frac{R}{N^2}\right)N^2R\lambda^3 .$$
We can put $c_0$ to be $100$, say. 
Moreover, suppose that $a >\lambda (\log N)^{-100}$,
    Then,
    $$|U_a|a^6 \lesssim  \log\left(\frac{\lambda}{a}\right)^{O(1)}\frac{a}{\lambda} N^2R\lambda^3$$
\end{lemma}

First, we derive the first part of the about lemma as a quick corollary from the superlevel bound from the wave envelop paper \cite[theorem 4]{amplitude}. The second part of the lemma is our key contribution. 

Let's recall a superlevel set estimate by Guth and Maldague that we heavily rely on in this paper. \begin{theorem}(Superlevel Set Bound)\label{super}
    $$|U_a|a^4 \lesssim(\log N)^{O(1)}K^{O(1)} \sum_k \sum_{\tau_k} \sum_{\substack{U_{\tau_k} \\ \|g_{\tau_k}\|^{1/2}_{\L_1(W_{U_{\tau_k}})} \geq \frac{K^{-1}b \lambda(l_k)}{\lambda} }}|U_{\tau_k}|\|g_{\tau_k}\|^2_{\L^1(W_{U_{\tau_k}})}.$$
\end{theorem}

\begin{remark}
    Note that this is not exactly what was stated in \cite[theorem 4]{amplitude}, they had $\lesssim_\epsilon R^\epsilon$ instead of $\lesssim(\log N)^{O(1)}K^{O(1)}$ because in their write up they put $K$ to be $R^\epsilon$ but they proof actually  gives $\lesssim(\log N)^{O(1)}K^{O(1)} $.  
\end{remark}
The above theorem implies the following corollary. 
\begin{corollary}\label{oldsup}
There exists a constant $c_0$ such that if $f$ satisfy $\Cn_0$, then 
    $$|U_a^{\lambda_0}|a^6 \lesssim (\log N)^{c_0} \max\left(\frac{\lambda_0}{R^{-1/2}N}\frac{a^2}{\lambda^2},\frac{R}{N^2}\right)\lambda^3 N^2R.$$
    We can put $c_0$ to be $100$, say. 
\end{corollary}
\begin{proof}
Note by local orthogonality, 
$\|g_{\tau_k}\|_{\L^1(W_{U_{\tau_k}})} \sim \|f_{\tau_k}\|_{\L^2(U_k)}^2$.
If $l_k<N/R$, 
    \begin{eqnarray*}
        && \sum_{\tau_k} \sum_{\substack{U_k \\ \|g_{\tau_k}\|^{1/2}_{\L^1(W_{U_{\tau_k}})} \geq \frac{K^{-1}a \lambda(l_k)}{\lambda} }}|U_k|\|g_{\tau_k}\|^{2}_{\L^1(W_{U_{\tau_k}})}\\
        &\lesssim&\max_{\tau_k, U_{\tau_k}} \|f_{\tau_k}\|_{\L^2(U_k)}^2\sum_{\tau_k} \sum_{\substack{U_k \\ \|g_{\tau_k}\|^{1/2}_{\L^1(W_{U_{\tau_k}})}\geq \frac{K^{-1}a \lambda(l_k)}{\lambda} }}|U_k|\|g_{\tau_k}\|_{\L^1(W_{U_{\tau_k}})}\\
        &\lesssim&\max_{\tau_k, U_{\tau_k}} \|f_{\tau_k}\|_{\L^2(U_k)}^2\sum_{\tau_k} \|f_{\tau_k}\|^2_{L^2(B_R)}\\
        &\lesssim&\max_{\tau_k, U_{\tau_k}} \|f_{\tau_k}\|_{\L^2(U_k)}^2\|f\|^2_{L^2(B_R)}\\
&\lesssim&\lambda(l_k)\lambda(l_k^{-1}R^{-1})\|f\|^2_{L^2(B_R)}\\
    &\lesssim&\lambda_0(l_kR^{1/2})(l_k^{-1}R^{-1}N)\|f\|^2_{L^2(B_R)}\\
    \end{eqnarray*}
    as $\lambda(l_k)\lesssim \lambda_0(l_kR^{1/2})$ and $\lambda(l_k^{-1}R^{-1})\lesssim (l_k^{-1}R^{-1}N)$. 
    Thus, if $l_k<N/R$, we have 
    $$\sum_{\tau_k} \sum_{\substack{U_k \\ \|g_{\tau_k}\|^{1/2}_{\L^1(W_{U_{\tau_k}})}\geq \frac{K^{-1}a \lambda_2(l_k)}{\lambda_2} }}|U_k|\|g_{\tau_k}\|^{2}_{\L^1(W_{U_{\tau_k}})}\lesssim \lambda_0R^{-1/2}N\lambda R^2. $$
    On the other hand, if $l_k \geq N/R$, then 
    \begin{eqnarray*}
        && a^2\sum_{\tau_k} \sum_{\substack{U_k \\ \|g_{\tau_k}\|^{1/2}_{\L^1(W_{U_{\tau_k}})} \geq \frac{K^{-1}a \lambda(l_k)}{\lambda} }}|U_k|\|g_{\tau_k}\|^{2}_{\L^1(W_{U_{\tau_k}})}\\
        &\lesssim&K^2\lambda^2\frac{1}{\lambda(l_k)^2}\sum_{\tau_k} \sum_{U_k}|U_k|\|g_{\tau_k}\|^{3}_{\L^1(W_{U_{\tau_k}})}\\
        &\lesssim&K^2\lambda^2\frac{1}{\lambda(l_k)^2}\max_{\tau_k, U_{\tau_k}} \|f_{\tau_k}\|_{\L^2(U_k)}^4\sum_{\tau_k} \sum_{\substack{U_k \\ \|S_{U_{\tau_k}}\|_{\L_2} \geq \frac{R^{-\epsilon}a \lambda(l_k)}{\lambda} }}|U_k|\|g_{\tau_k}\|_{\L^1(W_{U_{\tau_k}})}\\
        &\lesssim&K^2\lambda^3R^2.\\
    \end{eqnarray*}
    where in the last line we used that $$\|f_{\tau_k}\|^2_{\L^2(U_k)}\lesssim \lambda(l_k)$$ if $l_k \geq N/R$ using local orthogonality. 
    Thus, 
    \begin{align*}
        |U_a|a^6 &\lesssim (\log N)^{O(1)} K^{O(1)} \max(a^2\lambda_0 R^{-1/2}N\lambda R^2, \lambda^3 R^2)\\
        &=(\log N)^{O(1)} K^{O(1)}\max\left(\frac{\lambda_0}{R^{-1/2}N}\frac{a^2}{\lambda^2},\frac{R}{N^2}\right)\lambda^3 N^2R. \qedhere
    \end{align*}
\end{proof}
In particular, summing up the different dyadic choices of $\lambda_0\leq \frac{N}{\sqrt{R}}$, we have the following corollary.
\begin{corollary}\label{oldsup2}
There exists a constant $c_0$ such that that if $f$ satisfy $\Cn_0$ then
    $$|U_a|a^6 \lesssim (\log N)^{c_0} \max\left(\frac{a^2}{\lambda^2},\frac{R}{N^2}\right)\lambda^3 N^2R.$$
    We can put $c_0$ to be $100$, say. 
\end{corollary}

To deal with the second part of lemma \ref{special0}, we pigeonhole on $\lambda(\theta)$ and write $$f=\sum_{\substack{1\leq\lambda_0 \leq R^{-1/2}N\\ dyadic}} f^{\lambda_0}$$
with each $f^{\lambda_0}$ has $ \lambda_0/2\leq\lambda(\theta)\leq\lambda_0$.
Define 

$$U_{a}^{\lambda_0}=U_{a}(f_{\lambda_0}).$$
Note that, 
$$\int |f|^6 \leq \sum_{\substack{ 1 \leq a \leq \lambda\\a \text{ dyadic}}} |a|^6 |U_a|.$$

As $1+ \frac{1}{2^2}+\frac{1}{3^2}+\cdots \lesssim 1$ and the scale for $\lambda_0$ is dyadic, 
we have 
$$\sum_{1\leq \lambda_0 \leq R^{-1/2}N} \frac{1}{(\log (R^{-1/2}N/\lambda_0))^2} \sim 1+\frac{1}{2^2}+\frac{1}{3^2}+\cdots +\frac{1}{(\log R^{-1/2}N)^2}\lesssim 1.$$

Thus, 
$$\|f\|^6_{L^6(B_R)} \leq \sum_{\lambda_0}\sum_a |U_{a/(\log (R^{-1/2}N/\lambda_0))^2}^{\lambda_0}|a^6.$$

From now on in this subsection, we shall focus on the case when $a >\lambda (\log N)^{-100}$ and $\lambda_0>R^{-1/2}N(\log N)^{-100}$. We prove the following key lemma, which would immediately imply the second part of Lemma \ref{special00} by summing over all the dyadic $\lambda_0$'s. 
\begin{lemma}\label{key}
    Suppose that $a >\lambda (\log N)^{-100}$ and $\lambda_0>R^{-1/2}N(\log N)^{-100}$.
    Then,
    $$|U_a|a^6 \lesssim  \left((\log\left(\frac{\lambda}{a}\frac{NR^{-1/2}}{\lambda_0}\right)\right)^{O(1)}\frac{a}{\lambda}\frac{NR^{-1/2}}{\lambda_0} N^2R\lambda^3$$
\end{lemma}
Lemma \ref{key} follows from two key lemmas: the broad estimate (see Lemma \ref{br}) and the asymmetric estimate (see Lemma \ref{asym}), which both estimate the size of the intersection of the superlevel set of $f_{\Sigma_1}$ and $f_{\Sigma_2}$. The broad estimate deals with the case when $\Sigma_1$ and $\Sigma_2$ have the same length, and the length is similar to the distance between the two caps, and the asymmetric estimate deals with the case when the length of the shorter cap is much smaller than the distance between the two caps. 

Let $\Sigma_1$ and $\Sigma_2$ be two caps. 
Within both proofs, we shall use $\lambda_i$ to denote $\lambda(\Sigma_i)$ and use $l_i$ to denote $l(\Sigma_i)$ for $i=1,2$. 
Define 
    $$U_{a,b}(f_1,f_2):=\{x\in B_R: |f_1|>a \text{ and }|f_2|>b\}$$
    to be the joint superlevel set of $f_1$ and $f_2$. 
Define $V(k,\beta):=\{x: g_k \gtrsim \beta\}$ and $V(\beta):=\{x: g \gtrsim \beta\}$ to be the superlevel set of the blurred versions of $f$. Define 
$V'(k,\beta):=\{x: g_k \sim \beta\}$ and $V'(\beta):=\{x: g \sim \beta\}$.

\subsubsection{Broad Estimate} The overall strategy of the proof for the broad estimate follows that of \cite{amplitude}, with the key additional ingredient being a more careful control on the number of scales that can contribute to the high–low estimate. To bound $U_{a,b}(f_{\Sigma_1},f_{\Sigma_2})$, we shall first bound the high frequency part of $g_k$, which can be viewed as the $|f|^2$ blurred at scale $l_k^{-1}$. Then, we shall see that $U_{a,b}(f_{\Sigma_1},f_{\Sigma_2})$ is contained in $V(k,\beta)$ the superlevel set of $g_k$ for a big enough $\beta$. Then, we shall show that when $\beta$ is big enough, $g_k$ is necessarily dominated by its high-frequency part. Finally, we related the size of $U_{a,b}(f_1,f_2)$ with that of $V(k,\beta)$ using bilinear restriction and the fact that the Fourier support of $f_1$ and $f_2$ are transversal.

\begin{lemma}(Broad Estimate)\label{br}
  Suppose $\Sigma_1$ and $\Sigma_2$ are caps of length $l$ and are $l$ distance apart also suppose that 
    $a> \lambda(\Sigma_1)(\log N)^{-4000}$, $b> \lambda(\Sigma_2)(\log N)^{-4000}$,
    $\lambda_0> NR^{-1/2}(\log N)^{-100}$ then
  $$U_{a,b}(f_{\Sigma_1},f_{\Sigma_2}) \lesssim \log\left(\frac{\lambda(\Sigma_2)}{b} \frac{NR^{-1/2}}{\lambda_0}\right)\frac{N^2R(\lambda(\Sigma_1)+\lambda(\Sigma_2))}{a^2b^2}\frac{\lambda_0}{NR^{-1/2}}.$$
\end{lemma}
\begin{proof}
    By a rescaling argument, it suffices to prove the case where $l\sim 1$. 
    Now, let's get an estimate of the high-dominated part. 
\begin{proposition}(Bounding High dominated part)\label{boundhigh}
For any $0 \leq k<M$, 
    $$\int| \sum_{\tau_k} |f_{\tau_k}|^2*\check{\eta}_{\geq R^{-1}l_k^{-1}}(x)|^2dx \lesssim K^{O(1)}(M+1-k)\sum_{m=k}^M \sum_{\tau_m}|U_m|\|f_{\tau_m}\|^2_{\L^4(U_m)}.$$
\end{proposition}

\begin{proof}
Let $\eta_{\sim r}$ be the frequency support on the annulus $\{\xi: r/K \leq |\xi| \leq r\}$ and is about $r^2$ on a ball of radius $\sim 1/r$ and decays rapidly off that ball.
\begin{eqnarray*}
     \int| \sum_{\tau_k} |f_{\tau_k}|^2*\check{\eta}_{\geq R^{-1}l_k^{-1}}(x)|^2dx &\sim& \sum_{m=k}^M \int| \sum_{\tau_m} |f_{\tau_m}|^2*\check{\eta}_{\sim l_m^{-1}R^{-1}}(x)|^2dx\\
     &+& \sum_{m=k}^M \int| \sum_{\tau_m} |f_{\tau_m}|^2*\check{\eta}_{\sim l_m}(x)|^2dx.
\end{eqnarray*}
   By  high-low lemma \ref{hl}, we can bound 
\begin{eqnarray*}
    &&\int| \sum_{\tau_m} |f_{\tau_m}|^2*\check{\eta}_{\sim l_m^{-1}R^{-1}}(x)|^2dx\\ 
    &\lesssim &\int| \sum_{\theta} |f_{\theta}|^2*\check{\eta}_{\sim l_m^{-1}R^{-1}}(x)|^2dx \text{ by lemma }\ref{break}\\ 
    &\sim&  K^{O(1)}\sum_{\tau_m}\sum_{U_{\tau_m}} |U_{\tau_m}| \|f_{\tau_m}\|^4_{\L^2(U_{\tau_m})}.
\end{eqnarray*}
Also note that for $m \geq k$, 
   \begin{eqnarray*}
       &&\int| \sum_{\tau_k} |f_{\tau_k}|^2*\check{\eta}_{\sim l_m}(x)|^2dx\\
       &\sim & K^{O(1)} \int| \sum_{\tau_m} |f_{\tau_m}|^2*\check{\eta}_{\sim l_m}(x)|^2dx\\
       &\lesssim & K^{O(1)}\sum_{\tau_m} \int |f_{\tau_m}|^4dx \text{ by disjoint support in high frequency}\\
       &\lesssim & K^{O(1)}\sum_{\tau_m} \int|  \sum_{\theta \subset \tau_m}|f_\theta|^2|^2dx \text{ by Cordoba }L^4 \text{ estimate \cite{cor77}}\\
   \end{eqnarray*}
   Let's bound $\int|  \sum_{\theta \subset \tau_m}|f_\theta|^2|^2dx$.
   Note that
   $$\int|  \sum_{\theta \subset \tau_m}|f_\theta|^2|^2dx=\int|  \sum_{\theta \subset \tau_m}|f_\theta|^2*\check{\eta}_{ \leq l_m^{-1}R^{-1}}|^2dx+\sum_{n=m}^M \int|  \sum_{\theta \subset \tau_m}|f_\theta|^2*\check{\eta}_{ \sim l_n^{-1}R^{-1}}|^2dx.$$
   By the low lemma \ref{Low} and orthogonality, we have
   $$\sum_{\theta\subset \tau_m}|f_\theta|^2* \check{\eta}_{s \leq l_m^{-1}R^{-1}}(x) \sim \sum_{U_m} 1_{U_m} \|f_{\tau_m}\|_{\L^2(U_m)}^2$$
   so 
   $$\int|\sum_{\theta \subset \tau_m}|f_\theta|^2*\check{\eta}_{s \leq l_m^{-1}R^{-1}}|^2dx\lesssim K^{O(1)}\sum_{U_{\tau_m}} |U_{\tau_m}| \|f_{\tau_m}\|^4_{\L^2(U_{\tau_m})}.$$
   Finally, note that 
   \begin{eqnarray*}
       &&\int|  \sum_{\theta \subset \tau_m}|f_\theta|^2*\check{\eta}_{s \sim l_n^{-1}R^{-1}}|^2dx\\
       &\lesssim & K^{O(1)} \int|  \sum_{\tau_n \subset \tau_m}\left(\sum_{\theta \subset \tau_n}|f_\theta|^2\right)*\check{\eta}_{s \sim l_n^{-1}R^{-1}}|^2dx\\
       &\lesssim & K^{O(1)} \sum_{\tau_n \subset \tau_m}\sum_{U_{\tau_n}} |U_{\tau_n}| \|f_{\tau_n}\|^4_{\L^2(U_{\tau_n})} \text{ by lemma }\ref{hl}\\
   \end{eqnarray*}
   Thus, we have 
   \begin{align*}
       & \int| \sum_{\tau_k} |f_{\tau_k}|^2*\check{\eta}_{\geq R^{-1}l_k^{-1}}(x)|^2dx \\
       & \lesssim  K^{O(1)}\sum_{m=k}^M\left(\sum_{\tau_m}\sum_{U_{\tau_m}} |U_{\tau_m}| \|f_{\tau_m}\|^4_{\L^2(U_{\tau_m})}+\sum_{n=m}^M \sum_{\tau_n }\sum_{U_{\tau_n}} |U_{\tau_n}| \|f_{\tau_n}\|^4_{\L^2(U_{\tau_n})} \right)\\
       & \lesssim  K^{O(1)}(M+1-k)\sum_{m=k}^M\sum_{\tau_m}\sum_{U_{\tau_m}} |U_{\tau_m}| \|f_{\tau_m}\|^4_{\L^2(U_{\tau_m})} \qedhere
   \end{align*}
\end{proof}
Next, we show that if $a$ and $b$ are big enough, then the joint superlevel set $U_{a,b}(f_1,f_2)$ is contained in $V(k,\beta)$ for a big enough $\beta$ and show that when $\beta$ is big enough, $g$ is high-dominated. 
\begin{proposition}(Condition for high dominated) \label{highdom}
    Suppose that 
    $a> \lambda_1(\log N)^{-4000}$, $b> \lambda_2(\log N)^{-4000}$,
    $\lambda_0> NR^{-1/2}(\log N)^{-100}$ and  
    $l_k\gg R^{-1/2}\frac{\lambda^2}{a^2}\frac{R^{-1/2}N}{\lambda_0}$,
    Then, we are in the high dominated case, i.e. $$|U_{a,b}(f_1,f_2)|a^2b^2 \lesssim \int_{B_R}| \sum_{\tau_k\subset \tau_1\cup \tau_2} |f_{\tau_k}|^2*\check{\eta}_{\geq R^{-1}l_k^{-1}}(x)|^2dx.$$
\end{proposition}
\begin{proof}
First, we show that $U_a(f_1) \subset V(k,\beta,\Sigma_1)$ with $\beta=\frac{a^2}{\# (\tau_k\subset \Sigma_1)}$. This just follows from Cauchy-Schwartz, as $x \in U_a(f_1)$ implies that $\sum_{\tau_k \subset \Sigma_1}|f_{\tau_k}(x)|^2 \geq \frac{a^2}{\# (\tau_k\subset \Sigma_1)}$. 

Next, we show that for $x \in V(k,\frac{a^2}{\# (\tau_k\subset \Sigma_1)},\Sigma_1)$, 
we have $$\sum_{\tau_k\subset \Sigma_1}|f_{\tau_k}(x)|^2\lesssim \sum_{\tau_k\subset \Sigma_1}|f_{\tau_k}(x)|^2*\check{\eta}_{\geq R^{-1}l_k^{-1}}(x).$$
This basically follows from the fact that the low-frequency part cannot have a big $L^\infty$ norm.
By lemma \ref{break} and low lemma \ref{Low}, 
    \begin{eqnarray*}
    |g_{k,\Sigma_1}*\check{\eta}_{\leq R^{-1}l_k^{-1}}(x)|
        \lesssim \begin{cases}
\lambda_1 \text{ if }l_k \leq N/R \\
  \lambda_1\lambda(l_k^{-1}R^{-1}) \text{ if }l_k >N/R\\
    \end{cases}
    \end{eqnarray*}
    Note that
    \begin{align*}
      \lambda_1 \#(\tau_k \subset \tau_1)\leq \lambda_1\frac{\lambda_1}{\lambda_0}&\leq \frac{\lambda_1^2}{R^{-1/2}N}(\log N)^{100}\\
      &\leq \lambda_1^2(\log N)^{-9900}\\
      &\ll (\lambda_1 (\log N)^{-4000})^2 \\
      &\ll a^2.  
    \end{align*}
    Suppose that $l_k<N/R$, then as $l_k\gg R^{-1/2}\frac{\lambda_1^2}{a^2}\frac{R^{-1/2}N}{\lambda_0}$, 
    $$\lambda_1\lambda(l_k^{-1}R^{-1})\# \tau_k \leq \frac{\lambda_1^2}{\lambda_0}(l_k^{-1}R^{-1}N)\ll \frac{\lambda_1^2}{\lambda_0}(R^{-1}N)R^{1/2}(a^2/\lambda_1^2)\frac{\lambda_0}{R^{-1/2}N}=a^2.$$
Thus, in both cases 
$$ |g_{k,\Sigma_1}*\check{\eta}_{\leq R^{-1}l_k^{-1}}(x)|\ll \frac{a^2}{\# (\tau_k\subset \Sigma_1)}. $$
As $$g_{k,\Sigma_1}=g_{k,\Sigma_1}*\check{\eta}_{\leq R^{-1}l_k^{-1}}(x)+g_{k,\Sigma_1}*\check{\eta}_{\geq R^{-1}l_k^{-1}}(x),$$
so 
$|g_{k,\Sigma_1}(x)\lesssim g_{k,\Sigma_1}*\check{\eta}_{\geq R^{-1}l_k^{-1}}(x)|$
if $x \in V(k,\frac{a^2}{\# (\tau_k\subset \Sigma_1)},\Sigma_1)$ as $x \in U_a(f_1)\subseteq  V(k,\frac{a^2}{\# (\tau_k\subset \Sigma_1)},\Sigma_1)$, the inequality holds when $x \in U_a(f_1)$. 
Similarly, $x \in U_b(f_2)$ implies that 
$|g_{k,\Sigma_2}(x)\lesssim g_{k,\Sigma_2}*\check{\eta}_{\geq R^{-1}l_k^{-1}}(x)|$
Now we first convert bounds for $f_1$ and $f_2$ to bounds for $g_{k,\Sigma_1}$ and $g_{k,\Sigma_2}$ using bilinear restriction lemma \ref{bilresg}.
\begin{align*}
    |U_{a,a}(f_1,f_2)|a^4 &\leq \sum_{U_a \cap B_{l_k^{-1}}\not=\emptyset} \int_{B_{l_k^{-1}}}|f_1|^2|f_2|^2 \\
    &\lesssim \sum_{U_a \cap B_{l_k^{-1}}\not=\emptyset} \int_{B_{l_k^{-1}}}\left(\sum_{\tau_k \subset \Sigma_1} |f_{\tau_k}|^2 \right)\left(\sum_{\tau_k \subset \Sigma_2} |f_{\tau_k}|^2 \right), 
\end{align*}
where the second line follows from bilinear restriction lemma \ref{bires}. Next, we exploit the fact that $g_{k,\Sigma_1}$, $g_{k,\Sigma_2}$, $g_{k,\Sigma_1}*\check{\eta}_{\geq R^{-1}l_k^{-1}}$ and $g_{k,\Sigma_2}*\check{\eta}_{\geq R^{-1}l_k^{-1}}$ are basically constant on $B_{l_k^{-1}}$. 
\begin{align*}
    & \sum_{U_a \cap B_{l_k^{-1}}\not=\emptyset} |B_{l_k^{-1}}|\left\|\left(\sum_{\tau_k \subset \Sigma_1} |f_{\tau_k}|^2 \right)\left(\sum_{\tau_k \subset \Sigma_2} |f_{\tau_k}|^2 \right)\right\|_{L^\infty(B_{l_k^{-1}})}\\
    &\lesssim \sum_{U_a \cap B_{l_k^{-1}}\not=\emptyset} |B_{l_k^{-1}}|\left\|\left(\sum_{\tau_k \subset \Sigma_1} |f_{\tau_k}|^2 *\check{\eta}_{\geq R^{-1}l_k^{-1}}\right)\left(\sum_{\tau_k \subset \Sigma_2} |f_{\tau_k}|^2 *\check{\eta}_{\geq R^{-1}l_k^{-1}}\right)\right\|_{L^\infty(B_{l_k^{-1}})}\\
    &\lesssim \sum_{U_a \cap B_{l_k^{-1}}\not=\emptyset} |B_{l_k^{-1}}|\left\|\left(\sum_{\tau_k \subset \Sigma_1\cup \Sigma_2} |f_{\tau_k}|^2 *\check{\eta}_{\geq R^{-1}l_k^{-1}}\right)^2\right\|_{L^\infty(B_{l_k^{-1}})}\\
    &\lesssim \sum_{U_a \cap B_{l_k^{-1}}\not=\emptyset} \left\|\sum_{\tau_k \subset \Sigma_1\cup \Sigma_2} |f_{\tau_k}|^2 *\check{\eta}_{\geq R^{-1}l_k^{-1}}\right\|_{L^2(W_{B_{l_k^{-1}}})} \\
    &\lesssim  \int_{B_R}|\sum_{\tau_k\subset \tau_1 \cup \tau_2}|f_{\tau_k}|^2*\check{\eta}_{\geq R^{-1}l_k^{-1}}(x)|^2
\end{align*}
where the third line follows from Cauchy Schwartz, the fourth line follows by locally constant principle lemma \ref{locally constant} and the last line follows as the tail of $W_{B_{l_k^{-1}}}$ is rapidly decaying.
\end{proof}
Combining Propositions \ref{boundhigh} and \ref{highdom}, we have 
$$|U_{a,b}(f_1,f_2)|a^2b^2 \leq K^{O(1)}(M+1-k)\sum_{m=k}^M \sum_{\tau_m \subset \Sigma_1 \cup \Sigma_2}|U_{\tau_m}|\|f_{\tau_m}\|^4_{\L^2(U_{\tau_m})}$$
with $k$ chosen so that $l_k=K' R^{-1/2}\frac{\lambda^2}{a^2}\frac{R^{-1/2}N}{\lambda_0}$ for some big constant $K'$. Note that $(M-k)\lesssim \log(\frac{\lambda^2}{a^2}\frac{R^{-1/2}N}{\lambda_0})$ so there are $\sim \log(\frac{\lambda^2}{a^2}\frac{R^{-1/2}N}{\lambda_0})$ terms in the above sum. Thus,
$$|U_{a,b}(f_1,f_2)|a^2b^2 \lesssim \left(\log(\frac{\lambda^2}{a^2}\frac{R^{-1/2}N}{\lambda_0})\right)^2\max_{k \leq m\leq M} \sum_{\tau_m \subset \Sigma_1 \cup \Sigma_2}\sum_{U_{\tau_m}}|U_{\tau_m}|\|f_{\tau_m}\|^4_{\L^2(U_{\tau_m})}.$$
Next, we bound  
\begin{align*}
\sum_{\substack{\tau_m \subset \Sigma_1 \cup \Sigma_2 \\ U_{\tau_m}}} |U_{\tau_m}|
\|f_{\tau_m}\|^4_{\L^2(U_{\tau_m})}
&\leq \max_{\tau_m,U_{\tau_m}}\|f_{\tau_m}\|^2_{\L^2(U_{\tau_m})}
\sum_{\substack{\tau_m \subset \Sigma_1 \cup \Sigma_2 \\ U_{\tau_m}}} |U_{\tau_m}|
\|f_{\tau_m}\|^2_{\L^2(U_{\tau_m})} \\
&\leq \max_{\tau_m,U_{\tau_m}}\|f_{\tau_m}\|^2_{\L^2(U_{\tau_m})} \sum_{\tau_m \subset \Sigma_1 \cup \Sigma_2}\|f_{\tau_m}\|^2_{L^2(B_R)}\\
&\leq \lambda(l_m)\lambda(l_m^{-1}R^{-1})(\lambda_1+\lambda_2)R^2\\
&\leq \lambda_0 \frac{l_m}{R^{-1/2}}(l_m^{-1}R^{-1}N)(\lambda_1+\lambda_2)R^2\\
&\leq \lambda_0 \frac{\lambda_0}{NR^{-1/2}}(\lambda_1+\lambda_2)N^2R \qedhere
\end{align*}   
\end{proof}

\subsubsection{Asymmetric estimate}
Next, we shall prove the following asymmetric estimate. 
\begin{lemma}(Asymmetric Estimate)\label{asym}
Let $\Sigma_1$ and $\Sigma_2$ be caps that are $\sim l$ distance apart. 
Suppose $ \lambda(\Sigma_2)/\lambda_0^{1/2}  \ll b \leq \lambda(\Sigma_2)$, then
  $$\abs{U_{a,b}(f_{\Sigma_1},f_{\Sigma_2})} \lesssim \log\left(\frac{\lambda(\Sigma_2)}{b} \frac{NR^{-1/2}}{\lambda_0}\right)\max\left(\frac{l(\Sigma_2)}{l}, \frac{\sqrt{R}}{N}\right)\frac{N^2R\lambda(\Sigma_1)\lambda(\Sigma_2)^2}{a^2b^4}.$$
\end{lemma}
By a rescaling argument, the lemma above follows from the following lemma. 
\begin{lemma}\label{asym2}
Let $\Sigma_1$ and $\Sigma_2$ be caps that are $\sim 1$ distance apart and suppose that $ \lambda(\Sigma_2)/\lambda_0^{1/2}  \ll b \leq \lambda(\Sigma_2)$. Then
  $$\abs{U_{a,b}} \lesssim \log\left(\frac{\lambda(\Sigma_2)}{b} \frac{NR^{-1/2}}{\lambda_0}\right)^3\max\left(l(\Sigma_2), \frac{\sqrt{R}}{N}\right)\frac{N^2R\lambda(\Sigma_1)\lambda(\Sigma_2)^2}{a^2b^4}.$$
\end{lemma}

We begin by estimating the superlevel sets of $g_{\Sigma_2}$ via a high--low decomposition, while carefully controlling the number of scales that can contribute. Note that $g_{\Sigma_2}$ is essentially locally constant on rectangles $T$ of dimensions $R^{1/2} \times R^{1/2}L_2^{-1}$.

Next, we apply a slightly generalized bilinear restriction estimate \ref{bilresg} on each such rectangle $T$. This yields a decomposition of $f_{\Sigma_1}$ into caps $\omega$ of length $\max\{N^{-1},\allowbreak R^{-1/2}L_2\}$, and of $f_{\Sigma_2}$ into caps of length $L_2$, allowing us to bound the intersection of $U_{a,b}$ with each $T$.

Since we assume $L_2$ is small, the length of the cap $\omega$ we can break $f_{\Sigma_1}$ into is particularly small, which yields a gain of a factor $l(\Sigma_2)$.

Right now to simplify notation, write $g=\sum_{\theta \subset \Sigma_2} |f_\theta|^2$ and $V_\beta=\{x: g(x)>\beta\}$. Note that 
    $$\fint_{B_R} g \sim \fint_{B_R} \sum_{\gamma \subset \Sigma_2}|f_\gamma|^2 \sim \lambda_2.$$
    In the following lemma, we proved a bound for $|V_\beta|$ when $\beta$ is significantly bigger than the average size of $g$ on $B_R$.
\begin{proposition}
        Suppose that $\beta \gg \lambda_2$. Then,
    $$\beta^2 |T|(\#\{T:T\cap V(\beta)\not= \emptyset\})\lesssim \log\left(\frac{\lambda_2 }{\beta}\frac{N}{\sqrt{R}}\right)^2\lambda_0 NR^{-1/2}R^2\lambda_2. $$
    \end{proposition}
    \begin{proof}
    Let $r=\frac{\beta}{2N\lambda_2}$, then
    \begin{eqnarray*}
       &&\beta^2 |T|(\#\{T:T\cap V(\beta)\not= \emptyset\}) \\
       &\leq &\beta^2 |T|(\#\{T:\|g*\check{\eta}_{>r}\|_{L^\infty(T)} \geq \frac{\beta}{2}\}) \\
       &\leq &\beta^2 |T|(\#\{T:\|g*\check{\eta}_{>r}\|_{L^2(W_T)} \geq \frac{\beta}{2}\}) \text{ locally constant  }\\
       &\leq & \|g*\eta_{>r}\|_2^2\\
        &\leq &  \log\left(\frac{R^{-1/2}}{r}\right)\sum_{\substack{r\leq s \leq R^{-1/2}\\ s \text{ dyadic }}}\|g*\check{\eta}_{\sim s}\|_2^2 \text{ Cauchy Schwartz}\\
        &\leq &  \log\left(\frac{R^{-1/2}}{r}\right)\sum_{\substack{r\leq s \leq R^{-1/2}\\ s \text{ dyadic }}}\sum_{\substack{\tau\subset \Sigma_2 \\ l(\tau)=(sR)^{-1}}}\sum_{U \subset \U_{\tau}}|U|\|f_{\tau}\|^4_{\L^2(U)} \text{ by lemma }\ref{hl}\\
        &\leq &  \log\left(\frac{R^{-1/2}}{r}\right)\sum_{\substack{r\leq s \leq R^{-1/2}\\ s \text{ dyadic }}}\lambda_2(s)\lambda_2(s^{-1}R^{-1})\lambda_2 R^2\\
        & \lesssim & \log\left(\frac{\lambda_2 }{\beta}\frac{N}{\sqrt{R}}\right)^2\lambda_0 NR^{-1/2}R^2\lambda_2
   \end{eqnarray*}
        Note $\lambda(s)\leq sN$ and $\lambda(s^{-1}R^{-1}) \leq \lambda_0 s^{-1}R^{-1/2}$. 
    \end{proof}
The following lemma enables us to bound the size of the intersection between $V_\beta$ and $U_{a,b}(f_1,f_2)$. 
 \begin{proposition}\label{bires}
    Let $\omega$ be cap of length $\geq R^{-1/2}L_2$. Then, 
$$\int_{V(\beta)}|f_1|^2|f_2|^2 \lesssim \|\sum_{\omega \subseteq \Gamma_1} |f_{1, \omega}|\|_\infty \beta |T|(\#\{T:T\cap V(\beta)\not= \emptyset\})$$
    \end{proposition}
 \begin{proof}
     \begin{align*}
    \int_{V(\beta)}|f_1|^2|f_2|^2
    & \leq  \sum_{T: T \cap V(\beta) \not=\emptyset} \int_{T}|f_1|^2|f_2|^2\\
    & \lesssim  \sum_{T: T \cap V(\beta) \not=\emptyset} \int  (\sum_\omega |f_\omega|^2) (\sum_\theta |f_\theta|^2)\eta_T \\
    & \lesssim  \beta \|\sum_{\omega \subseteq \Gamma_1} |f_{1, \omega}|\|_\infty\sum_{T: T \cap V(\beta) \not=\emptyset} \int \eta_T \\
    & \lesssim  \beta \|\sum_{\omega \subseteq \Gamma_1} |f_{1, \omega}|\|_\infty(\#\{T: T \cap V(\beta) \not=\emptyset\}) |T| \qedhere
\end{align*}
\end{proof}   
 Now, we are ready to put things together. 
    First, by Cauchy-Schwartz, $|f_2(x)| \geq b$ implies that $g(x) \geq \frac{b^2\lambda_0}{\lambda_2}$.  Thus,
$$U_{a,b} \subseteq \bigcup_{\lambda_2\lambda_0 \geq \beta \geq \frac{b^2\lambda_0}{\lambda_2}} V(\beta) \cap U_{a,b}. $$
For $b \gg  \lambda_2 \lambda_0^{-1/2} $, we need $\beta \gg \lambda_2$.
Note there are just $\log(\frac{\lambda_2^2}{b^2})$ many choices of dyadic $\lambda_2\lambda_0 \geq \beta \geq \frac{b^2\lambda_0}{\lambda_2}$. We shall bound each $|V(\beta) \cap U_{a,b}|$. 

\begin{eqnarray*}
    && |V(\beta) \cap U_{a,b}|a^2b^2 \\&\leq& \int_{V(\beta)} |f_1|^2|f_2|^2 \\
    & \lesssim & \|\sum_{\omega \subseteq \Gamma_1} |f_{1, \omega}|\|_\infty \beta |T|(\#\{T:T\cap V(\beta)\not= \emptyset\}) \text{ by lemma }\ref{bires}\\
    & \lesssim_\epsilon& \frac{1}{\beta}\|\sum_{\omega \subseteq \Gamma_1} |f_{1, \omega}|\|_\infty \log\left(\frac{\lambda_2 }{\beta}\frac{N}{\sqrt{R}}\right)^2\lambda_0 NR^{-1/2}R^2\lambda_2.
\end{eqnarray*}
Suppose that $R^{-1/2}l_2 \geq 1/N$. We let $l(\omega)=R^{-1/2}L_2$, so $\|\sum_{\omega \subset \Gamma_1} |f_{1,\omega}|^2\| \leq \lambda_1\lambda_1(R^{-1/2}l_2)\leq \lambda_1 (l_2N/\sqrt{R})$. Suppose that $R^{-1/2}l_2 \leq 1/N$. We let $l(\omega)=1/N$ so $\|\sum_{\omega \subset \Gamma_1} |f_{1,\omega}|^2\| \leq \lambda_1$. 
Thus, we have plugging in $\beta \geq \frac{b^2\lambda_0}{\lambda_2}$. 
\begin{eqnarray*}
  && |V(\beta) \cap U_{a,b}|a^2b^2 \\
  &\leq& \frac{\lambda_2}{b^2\lambda_0} \lambda_1 \max(1, l_2N/\sqrt{R}) \log\left(\frac{\lambda_2^2 }{b^2}\frac{NR^{-1/2}}{\lambda_0}\right)^2\lambda_0 NR^{-1/2}R^2\lambda_2 \\
& \lesssim & \frac{1}{b^2}\max(1, l_2N/\sqrt{R}) \log\left(\frac{\lambda_2 }{b}\frac{NR^{-1/2}}{\lambda_0}\right)^2\lambda_0 NR^{-1/2}R^2 \lambda_1\lambda_2^2
\end{eqnarray*}
Summing over all the dyadic $\beta$'s, we proved Lemma \ref{asym2}. 

\subsubsection{Refined Broad and Narrow Argument}
Now, we develop a refined broad and narrow argument by combining the broad estimate with the asymmetric estimate to prove lemma \ref{key}. 
\begin{definition}
    We say $x \in Br_{m}(a')$ if there exist  $\tau_m,\tau_m' $ such that $ l_m \leq dist( \tau_m, \tau_m')\leq 2l_{m-1}$ and $x \in U_{a'/10,a'/10}(f_{\tau_m}, f_{\tau_m'})$. We say $x \in Nar_m(a')$ if there exists $\Tilde{\tau}_m$ of length $2l_m$ such that $|f_{\tau_m}(x)|>a'/10$. 
\end{definition}
Now, let's gather some properties of $Br_m$ and $Nar_m$. 
\begin{proposition}
    If $a_1+\cdots+a_k <a/2$ and $l_k<a/(10N)$, then
    $$U_a \subset \bigcup_{j=1}^k Br_m(a_m).$$
\end{proposition}
\begin{proof}
    First, we note that $U_a \subset Br_1(a')\cup Nar_1(a-a')$ for $0<a'<a$. Then, note that for $m \geq 1$, $Nar_m(a') \subset Br_{m+1}(a'')\cup Nar_{m+1}(a'-a'')$ for $0<a''<a'$. Thus, inductively, we have 
    $$U_a \subset Br_1(a_1)\cup \cdots \cup Br_m(a_m) \cup Nar_m(a-\sum_{j=1}^m a_j).$$
    In particular, suppose $a_1+\cdots+ a_k <a/2$, then 
    $$U_a \subset Br_1(a_1)\cup \cdots \cup Br_k(a_k) \cup Nar_k(1-\sum_{j=1}^k a_k).$$
    Note that $Nar_k(a-\sum_{j=1}^k a_k)$ is empty as $a-\sum_{j=1}^k a_k>a/2$ and $\lambda(l_k)<a/10$. 
\end{proof}
\begin{proposition}\label{numbm}
    $$\#\{m: Br_m(a')\not=\emptyset \}\leq \frac{\lambda}{a'}.$$
\end{proposition}
\begin{proof}
    Define $A=\{m: Br_m(a')\not=\emptyset \}$. Then, for each $m \in A$, there exist $\tau_m$ and $\tau_{m'}$ that are $\gtrsim l_m$ distance apart and $\lambda(\tau_m)\geq a'$ and $\lambda(\tau_m')\geq a'$ so $\lambda(\tau_m \cup \tau_{m'}) \geq 2a'$. 
    Suppose $n>m$, $m,n \in A$, then either $(\tau_n \cup \tau_n') \cap \tau_m=\emptyset$ or  $(\tau_n \cup \tau_n') \cap \tau_m'=\emptyset$, $\lambda(\tau_m \cup \tau_{m'} \cup \tau_n \cup \tau_n') \geq 3a'$. In general, $\lambda(\cup_{j=1}^{|A|}(\tau_{m_j} \cup\tau_{m_j}'))\geq |A|a'$. Thus, $|A|\leq \lambda/a'$. 
\end{proof}
From now on, we put
$$h=100\log((\lambda/a)\log(NR^{-1/2}/\lambda_0)).$$
\begin{proposition}\label{faraway}
    If $|m-n|>h$, then $$\abs{Br_m(a(a/\lambda)^3) \cap Br_n(a(a/\lambda)^3)}\lesssim \frac{N^2R\lambda}{a^4}\left(\frac{a}{\lambda}\right)^{20}.$$
\end{proposition}
\begin{proof}
    Suppose $x \in Br_m(a(a/\lambda)^3) \cap Br_n(a(a/\lambda)^3)$, suppose $m>n+h$ then there exist $\tau_m$ and $\tau_n$ such that $dist(\tau_m,\tau_n) \gtrsim l_n$ and $x \in U_{a',a'}(f_{\tau_m},f_{\tau_n})$ with $a'=a(a/\lambda)^3$. Note that 
    $$a'\gg \lambda (\log N)^{-400}\gg \lambda(\tau_m)/\lambda_0^{1/2} $$     
as we assumed that \[a>\lambda (\log N)^{-100}, N>R^\frac{1}{2}(\log N)^{100000} \text{ and } \lambda_0>R^{-1/2}N(\log N)^{-100}. \] 
Thus, we can apply lemma \ref{asym} to get
\begin{align*}
   \abs{U_{a',a'}(f_{\tau_m},f_{\tau_n})}
   &\lesssim\log\left(\frac{\lambda(\tau_2)}{a'} \frac{NR^{-1/2}}{\lambda_0}\right)\max\left(\frac{l_m}{l_n}, \frac{\sqrt{R}}{N}\right)\frac{N^2R\lambda(\tau_m)\lambda(\tau_n)^2}{a'^6} \\
   &\lesssim \log\left(\frac{\lambda}{a'} \frac{NR^{-1/2}}{\lambda_0}\right)\max\left(3^{-h}, (\log N)^{-10000}\right)\frac{N^2R\lambda^3}{a'^6}\\
    &\lesssim\ \log\left(\frac{\lambda}{a'} \frac{N R^{-1/2}}{\lambda_0}\right)
\max\left(\left(\frac{a}{\lambda}\right)^{100}, (\log N)^{-10000}\right) \\
&\quad \times \log\left(\frac{N R^{-1/2}}{\lambda_0}\right)^{-100}
\frac{N^2 R \lambda^3}{a'^6} \\
&\lesssim \frac{N^2 R \lambda}{a^4}\left(\frac{a}{\lambda}\right)^{50}.
\end{align*}  
Note that in the last step we used that  $(a/\lambda)^{100} \geq (\log N)^{-10000}$ as we assume that $a>\lambda (\log N)^{-100}$.

Note in order for $\tau_m$ to contribute, $\lambda(\tau_m) \geq a'$ and the same to $\tau_n$. Thus, there are less than $\lambda/a'$ many $\tau_m$ and $\lambda/a'$ many $\tau_n$ that might contribute. 
Then, we sum things up 
\begin{align*}
 \abs{Br_m(a(a/\lambda)^3) \cap Br_n(a(a/\lambda)^3)}
   &\lesssim  \sum_{\substack{\tau_m,\tau_n\\ \lambda(\tau_m)\geq a',\lambda(\tau_n)\geq a'\\ dist(\tau_m,\tau_n)\gtrsim l_n}}U_{a',a'}(f_{\tau_m},f_{\tau_n})\\
   &\lesssim  \sum_{\substack{\tau_m,\tau_n\\ \lambda(\tau_m)\geq a',\lambda(\tau_n)\geq a'\\ dist(\tau_m,\tau_n)\gtrsim l_n}}\frac{N^2R\lambda}{a^4}\left(\frac{a}{\lambda}\right)^{50}\\
   &\lesssim  (\lambda/a')^2\frac{N^2R\lambda}{a^4}\left(\frac{a}{\lambda}\right)^{50}\\
   &\lesssim \frac{N^2R\lambda}{a^4}\left(\frac{a}{\lambda}\right)^{20} \qedhere
\end{align*}
\end{proof}
\begin{proposition}
    If we define $B_1=\cup_m Br_m(a/(100h))$ and 
    $$B_2=\bigcup_{|m-n|>h}\left(Br_m(a(a/\lambda)^3) \cap Br_n(a(a/\lambda)^3)\right)$$
    Then, $U_a \subset B_1 \cup B_2$. 
\end{proposition}
\begin{proof}
    Suppose $x \not\in B_2$, then there exists $A \in \{1,\cdots, k\}$ such that $|A|<h$ and for $m \not\in A$, $x\not \in Br_m(a(a/\lambda)^3)$. Put $a_m=a(a/\lambda)^3$ for $m\not\in A$ and $a_m=a/(100h)$ for $m\in A$. Note $\sum_{m=1}^k a_m<a/2$, so $x \in \cup_{m\in A}Br_m(a/(10h))$. Thus, $x\in B_1$. 
\end{proof}
Thus, it remains to bound $B_1$ and $B_2$. 
\begin{proposition}(Bounding $B_1$)
$$\abs{B_1}\lesssim \log\left(\frac{\lambda}{a}\frac{NR^{-1/2}}{\lambda_0}\right)^{O(1)} \frac{N^2R\lambda^2}{a^5}\frac{\lambda_0}{NR^{-1/2}}$$
\end{proposition}
\begin{proof}
We write $\tau_m \sim \tau_m'$ if  $l_m<dist(\tau_m', \tau_m)<2l_{m-1}$. Note that 
$$\frac{a}{100h} \gtrsim \lambda (\log N)^{-100} \log((\lambda/a)\log(NR^{-1/2}/\lambda_0))^{-1}\gtrsim \lambda (\log N)^{-4000}.$$
Thus, we can apply lemma \ref{br} to get
    \begin{align*}
        \abs{Br_m(a/100h)}&\lesssim \sum_{\tau_m\sim\tau_m'}\log\left(\frac{\lambda(\tau_2)}{b} \frac{NR^{-1/2}}{\lambda_0}\right)\frac{N^2R(\lambda(\tau_m)+\lambda(\tau_m'))}{(a/(100h))^4}\frac{\lambda_0}{NR^{-1/2}}\\
&\lesssim\sum_{\tau_m\sim\tau_m'}\log\left(\frac{\lambda(\tau_2)}{b} \frac{NR^{-1/2}}{\lambda_0}\right)\frac{N^2R(\lambda(\tau_m)+\lambda(\tau_m'))}{a^4}h^4\frac{\lambda_0}{NR^{-1/2}}\\
&\lesssim\log\left(\frac{\lambda(\tau_2)}{b} \frac{NR^{-1/2}}{\lambda_0}\right)\frac{N^2R\lambda}{a^4}h^4\frac{\lambda_0}{NR^{-1/2}}\\
&\lesssim\log\left(\frac{\lambda(\tau_2)}{b} \frac{NR^{-1/2}}{\lambda_0}\right)^{O(1)}\frac{N^2R\lambda}{a^4}\frac{\lambda_0}{NR^{-1/2}}
    \end{align*}
as we put $h=100\log((\lambda/a)\log(NR^{-1/2}/\lambda_0))$.
Summing things up and by proposition \ref{numbm},   
    \begin{align*}
        \abs{B_1}&\lesssim  \sum_{m=1}^k Br_m(a/100h)\\
        &\lesssim  \frac{\lambda}{a/(100h)}\log\left(\frac{\lambda(\tau_2)}{b} \frac{NR^{-1/2}}{\lambda_0}\right)^{O(1)}\frac{N^2R\lambda}{a^4}\frac{\lambda_0}{NR^{-1/2}}\\
        &\lesssim\log\left(\frac{\lambda}{a}\frac{NR^{-1/2}}{\lambda_0}\right)^{O(1)} \frac{N^2R\lambda^2}{a^5}\frac{\lambda_0}{NR^{-1/2}} \qedhere
    \end{align*}
\end{proof}
\begin{proposition}(Bound $B_2$)
   \begin{equation*}
       \abs{B_2} \lesssim \frac{N^2R\lambda}{a^4}\left(\frac{a}{\lambda}\right)^{10}
   \end{equation*} 
\end{proposition}
\begin{proof}

By proposition \ref{faraway},
    \begin{align*}
        \abs{B_2} &\leq  \sum_{|m-n|>l} \left|Br_m(a(a/\lambda)^3) \cap Br_n(a(a/\lambda)^3)\right|\\
        &\lesssim  \sum_{\substack{m:Br_m(a(a/\lambda)^3)\not=\emptyset\\n:Br_n(a(a/\lambda)^3)\not=\emptyset}}\frac{N^2R\lambda}{a^4}\left(\frac{a}{\lambda}\right)^{20}\\
        &\lesssim  \left(\frac{\lambda}{a}\right)^8\frac{N^2R\lambda}{a^4}\left(\frac{a}{\lambda}\right)^{20}\\
        &\lesssim  \frac{N^2R\lambda}{a^4}\left(\frac{a}{\lambda}\right)^{10} \qedhere
    \end{align*}
\end{proof}
Adding the previous two propositions up, we proved that 
$$|U_a^{\lambda_0}|\lesssim \log\left(\frac{\lambda}{a}\frac{NR^{-1/2}}{\lambda_0}\right)^{O(1)} \frac{N^2R\lambda^2}{a^5}\frac{\lambda_0}{NR^{-1/2}}.$$

\section{Strichartz Estimates in $Y^s$ spaces}\label{sec:4}
\subsection{The Rescaled Torus and $Y^s$ Spaces}\label{sec:4.1}
We begin by introducing the rescaled torus 
 $\T_\lambda:= \R/\left(2\pi \lambda \Z\right)$ with $\lambda\geq 1$ and the associated function spaces that will be used in the sequel. For $f: \T_\lambda \rightarrow \mathbb{C}$ measurable, we define Lebesgue norms by
\begin{equation*}
    \norm{f}_{L^p(\T_\lambda)}^p = \int_{0}^{2\pi\lambda} \abs{f(x)}^p \,\, dx,
\end{equation*}
for $1 \leq p < \infty$ and the usual modification for $p=\infty$. We also define $\left(dk\right)_\lambda$ to be the normalized counting measure on $\Z_\lambda := \frac{1}{\lambda}\Z$
\begin{equation*}
    \int_{\Z_\lambda} a(k) \left(dk\right)_\lambda := \frac{1}{2\pi\lambda} \sum _{k \in \Z_\lambda} a(k).
\end{equation*}
The Fourier coefficients of $f \in L^1(\T_\lambda)$ are given by
\begin{equation*}
    \hat{f}(k) = \int_{0}^{2\pi\lambda} e^{-i k x} f(x) \, \, dx,
\end{equation*}
for $k \in \Z_\lambda$ and the Fourier inversion formula is given by
\begin{equation*}
    f(x)= \int_{\Z_\lambda} e^{i k x} \hat{f}(k)\left(dk\right)_\lambda.
\end{equation*}
Using this convention, the following identities hold
\begin{subequations}
\begin{align}
 \label{def: Parseval}
   & \int_{0}^{2\pi\lambda} f(x) \overline{g(x)} \,\, dx = \int_{\Z_\lambda} \hat{f}(k) \overline{\hat{g}(k)} \,\, (dk)_\lambda,  \tag{Parseval}  \\ 
    \label{def:Plancharel} 
&\norm{f}_{L^2(\T_\lambda)}=\norm{\hat{f}}_{L^2((dk)_\lambda)},\tag{Plancherel} \\
\label{def: Convolution}
   &\widehat{fg}(k) = \int_{\Z_\lambda} \hat{f}(k_1) \hat{g}(k-k_1) \,\, (dk_1)_\lambda.  \tag{Convolution}
\end{align}
\end{subequations}
The Sobolev space, $H^s=H^s(\T_\lambda)$, is defined as the completion of smooth functions under the norm
\begin{equation*}
    \norm{f}_{H^s}=\norm{\langle k \rangle^s \hat{f}(k)}_{L^2((dk)_\lambda)},
\end{equation*}
where $\langle k \rangle := (1+\abs{k}^2)^\frac{1}{2}$. We will frequently make use of Littlewood-Paley theory which allows us to quantitatively separate the rough, high-frequency behavior of a function from the smooth, low-frequency one. In particular, for $A \subset \Z_\lambda$ let $P_A$ denote the Fourier multiplier $\widehat{P_A f}:= \chi_A \hat{f}.$ By a slight abuse of notation, for $N \in 2^{k_0\N_0}$ with $k_0 \geq 1$ we define 
\begin{equation*}
    P_N:= \begin{cases}
        P_{\left[-2^{k_0}N,-N\right) \cup \left(N,2^{k_0}N\right]} &\text{ for } N > 1\\
        P_{\left[-1,1\right]}           &\text{ for } N \leq 1.
    \end{cases}
\end{equation*}
In most of our estimates, we adopt the standard dyadic decomposition with $k_0 = 1$. The sole exception occurs in Lemma \ref{lemma:multilinear_estimates_lwp}, where we take $k_0 = 1/s$ for $0 < s \leq 1$. We also stress that if \(f = f(x,t)\) depends on time, the operator \(P_A\) acts only on the spatial variable: \( (P_A f)(x,t) = P_A(f(\cdot,t))(x) \).

\noindent We define $S_\lambda(t)$ to be the solution operator to the linear Schrödinger equation
\begin{equation*}
    i \partial_t u + \Delta u =0, \text{ } u(x,0)=u_0(x), \text{  } x \in \T_\lambda,
\end{equation*}
that is,
\begin{equation*}
    S_\lambda(t) u_0(x) = \int_{\Z_\lambda} e^{ i(k x+\abs{k}^{2}t)} \hat{u}_0(k) (dk)_\lambda.
\end{equation*}
When $\lambda=1$, we drop the subscript and write $S(t):=S_1(t)$. 

As we will be working with long-time Strichartz estimates on the rescaled torus $\T_\lambda$, it is convenient to prove local well-posedness results in adapted function spaces instead of using $X^{s,b}$ spaces. To this end, we introduce the $V^p$, $U^p$ and $Y^s$ spaces which were first used in the critical regularity well-posedness theory of dispersive equations on periodic domains \cite{Herr_Tartaru_Tzvetkov,Hadac_Herr_Koch,Koch_Tataru}. More recently, the same spaces have been used to prove global well-posedness for the mass critical NLS on $\T$ \cite{Schippa_mass_critical,McConnell_Mass_critical_d_1} and $\TT$ \cite{Herr_Kwak_2024_mass_critical} in the sub-critical setting.

Throughout the remainder of this section, let $p \geq 1$, let $\Hi$ be a separable Hilbert space over $\C$ and define $\mathcal{Z}$ to be set of finite partitions $-\infty <t_0 <t_1 < \cdots<t_K\leq \infty$ of the real line. If $t_K=\infty$, we use the convention that $v(t_K):=0$ for all functions $v: \R \rightarrow \Hi$.
\begin{definition}
\label{Def: V^p spaces}
$V^p=V^p(\R \rightarrow \Hi)$ is defined to be the space of all right continuous functions $f : \R  \rightarrow \Hi $ which satisfy $\lim_{t \rightarrow - \infty} f(t)=0$ and
\begin{equation*}
    \norm{f}_{V^p}^p:= \sup_{\{t_k\}_{k=0}^K \in \mathcal{Z}} \sum_{k=1}^K \norm{f(t_k)-f(t_{k-1})}_{\Hi}^p < \infty.
\end{equation*}
Moreover, for $\Hi=L^2(\T_\lambda)$, we define $V^p_S:=S_\lambda(\cdot) V^p$ and endow it with norm $\norm{f}_{V^p_S}:=\norm{S_\lambda(-\cdot)f}_{V^p}. $
\end{definition}

\begin{definition}
 For $\{t_k\}_{k=0}^K \in \mathcal{Z}$ and $\{f_k\}_{k=0}^{K-1} \subset \Hi$ satisfying $\sum_{k=0}^{K-1} \norm{f_k}_{\Hi}^p =1$ and $f_0=0$, we define $f: \R \rightarrow \Hi$ to be a $U^p$-atom if
\begin{equation*}
    f(t)=\sum_{k=1}^K \chi_{[t_k,t_{k-1})}(t)f_{k-1}.
\end{equation*}
The atomic space $U^p=U^p(\R \rightarrow \Hi)$ comprises functions $f: \R \rightarrow \Hi$ of the form
\begin{equation*}
    f(t)=\sum_{j=1}^\infty \lambda_j f_j(t) \text{ for } U^p\text{-atoms } f_j, \{\lambda_j\} \in \ell^1, 
\end{equation*}
and is endowed with norm
\begin{equation*}
    \norm{f}_{U^p} := \inf\left\{ \sum_{j=1}^\infty \abs{\lambda_j} : f=\sum_{j=1}^\infty \lambda_j f_j, \,\, \lambda_j \in \C, \,\, f_j \text{ } U^p\text{-atom}\right\}.
\end{equation*}
Moreover, for $\Hi=L^2(\T_\lambda)$, we define $U^p_S:=S_\lambda(\cdot) U^p$ and endow it with norm $\norm{f}_{U^p_S}:=\norm{S_\lambda(-\cdot)f}_{U^p}$.
\end{definition}
The spaces $U^p$ and $V^p$ are related by the following continuous embeddings which follow from \cite[Section 2]{Herr_Tartaru_Tzvetkov}.
\begin{lemma}
\label{lemma: U^p/V^p embeddings}
For $1 \leq p < q <\infty$, we have the following continuous embeddings
\begin{equation*}
    U^p \hookrightarrow V^p \hookrightarrow U^q.
\end{equation*}
\end{lemma}

We conclude by the defining the $Y^s$ spaces and stating some well-known properties. For details, see \cite[Section 2]{Herr_Tartaru_Tzvetkov}.
\begin{definition}
    For $s \in \R$, define $Y^s=Y^s(\R \times \T_\lambda \rightarrow \C)$ as the space of functions $f: \R \times \T_{\lambda} \rightarrow \C$ such that $e^{-it\abs{k}^{2\alpha}}\widehat{f(t)}(k)$ lies in $V^2(\R \rightarrow \C)$ for all $k \in \Z_\lambda$ and
\begin{equation*}
    \norm{f}_{Y^s}^2:= \int_{\Z_\lambda} \langle k \rangle^{2s} \norm{e^{-it\abs{k}^{2}}\widehat{f(t)}(k)}_{V^2}^2 (dk)_\lambda<\infty.
\end{equation*}
For a time interval $J_T \subset \R$ of length $T>0$, we define $Y_{J_T}^s$ to be the restriction of $Y^s$ to $J_T \times \T_\lambda$ and endow it with the norm
\begin{equation*}
    \norm{f}_{Y^s_{J_T}} := \inf \left\{\norm{g}_{Y^s} : f= g \text{ on } J_{T} \times \T_\lambda \right\}.
\end{equation*}
When $J_T=[0,T)$, we drop the double subscript and write $Y^s_T$ instead.
\end{definition}
\begin{lemma}
\label{lemma: Y^s properties} For any $s \in \R$, the $Y^s$ spaces satisfy the following properties:
\begin{enumerate}
    \item[(a)] The embeddings $Y^s \hookrightarrow L^\infty(\R\rightarrow H^s)$ and $Y^0 \hookrightarrow V^2_S$ are continuous;
     \item[(b)] For any $A, B$ disjoint subsets of $\Z_\lambda$, 
    \[
    \norm{P_{A\cup B} \phi }_{Y^s}^2=\norm{P_{A} \phi}_{Y^s}^2+\norm{P_{B} \phi}_{Y^s}^2;
    \]
     \item[(c)] For any  $T>0$ and any $\phi \in H^s$,
        \begin{equation*}
            \norm{\chi_{[0,T)}S_\lambda(t)\phi}_{Y^s} \sim \norm{\phi}_{H^s};
        \end{equation*}
     \item[(d)] For any $T>0$ and any $f \in L^1 \left([0,T) \rightarrow H^s\right)$,
        \begin{equation*}
            \norm{\chi_{[0,T)} \int_0^t S_\lambda(t-t') f(t') \,\,dt'}_{Y^s} \lesssim \sup_{\norm{g}_{Y^{-s}} \leq 1} \abs{\int_0^T \int_{\T_\lambda} f\overline{g} dx dt}.
        \end{equation*}
\end{enumerate}
We stress that the constants in the above estimates are independent of $\lambda$, $s$ and $T$.
\end{lemma}
\subsection{Strichartz Estimates}
\noindent We next turn to transferring the Strichartz estimate proved in \cref{sec:3} to $Y^s$ spaces. Throughout this section, all frequency scales $M$ appearing in the projections $P_M$ are assumed to satisfy $M \geq 1$, unless stated otherwise. We also assume throughout that $\lambda \geq 1$.

In view of the scaling symmetry of the problem, it is natural to work on the rescaled torus $\mathbb{T}_\lambda$. To relate this setting to the standard torus $\mathbb{T}$, we begin by recalling a rescaling lemma from \cite[Lemma~3.1]{fractional_global_wellposedness_T}, which allows us to transfer estimates between the two settings.
\begin{lemma}
\label{Lemma: Conversion_between_T_and_T_lambda}
    Let $T>0$, $n\in \N$ and $N_1, \dots N_n >0$. Then the following two statements are equivalent:
    \begin{enumerate}
        \item[(i)] There exists a constant $C_1(T, N_1, \dots , N_n)>0$ such that for all $f_j \in L^2(\T)$ satisfying $\supp \hat{f}_j \subset \left\{k \in \Z:  \frac{N_j}{2} \leq \abs{k} < N_j \right\}$ for $1\leq j \leq n$, we have the following estimate    
    \begin{equation*}
        \int_{[0,T] \times \T} \prod_{j=1}^n \abs{S(t)f_j}^2 \leq C_1(T, N_1, \dots , N_n) \prod_{j=1}^n \norm{f_j}_{2}^2.
    \end{equation*}
    \item[(ii)] There exists a constant $C_\lambda(T, N_1, \dots , N_n)>0$ such that for all $\phi_j \in L^2(\T_\lambda)$ satisfying $\supp \hat{\phi}_j \subset \left\{k \in \Z_\lambda: \frac{N_j}{2} \leq \abs{k} < N_j \right\}$ for $1\leq j \leq n$, we have the following estimate
     \begin{equation*}
        \int_{[0,T] \times \T_\lambda} \prod_{j=1}^n \abs{S_\lambda(t)\phi_j}^2 \leq C_\lambda(T, N_1, \dots , N_n)\prod_{j=1}^n \norm{\phi_j}_{2}^2.
\end{equation*}
\end{enumerate}
Moreover, the constants are related by
\begin{equation*}
    C_\lambda(T, N_1, \dots , N_n)=\lambda^{3-n}C_1(\lambda^{-2}T, \lambda N_1, \dots , \lambda N_n).
\end{equation*}
\end{lemma}
The following lemma follows immediately from Theorem~\ref{theorem:L^6 Strichartz}.
\begin{lemma}
\label{lemma_Strichartz_Linear_unrescaled}
  There exist a constant $C>1$ such that for any $\phi \in L^2(\T)$ and any $N \geq 1$, the following estimate holds
   \[
\norm{S(t)P_M\phi}_{L_{t,x}^6([0,\left(\log N\right)^{-C}] \times \T)} \leq C \left(1+\frac{\log M}{\log N}\right)^\frac{C}{6}  \norm{\phi}_2.
\]
\end{lemma}
Combining Lemma \ref{lemma_Strichartz_Linear_unrescaled} with Lemma \ref{Lemma: Conversion_between_T_and_T_lambda} yields the following Strichartz estimates on the rescaled torus.
\begin{lemma}
    \label{lemma_Strichartz_Linear}
  There exists a constant $C>1$ such that for any $\phi \in L^2(\T_\lambda)$ and any $N \geq 1$, the following estimate holds
  \[ \norm{S_\lambda(t)P_M\phi}_{L_{t,x}^6([0,\lambda^2\left(\log \lambda N\right)^{-C}] \times \T_\lambda)} \leq C \left(1+\frac{\log M}{\log N}\right)^\frac{C}{6} \norm{\phi}_2.
\]
\end{lemma}
\noindent We conclude this section by transferring the linear Strichartz estimate to $Y^s$ spaces and by stating a trilinear estimate on $Y^s$ spaces.
\begin{lemma}
\label{lemma:Linear Y^0 Strichartz}
There exist a constant $C>1$, such that for any $u \in Y^0(\R \times \T_\lambda \rightarrow \C)$ and any $N \geq 1$, the following estimate holds
    \begin{eqnarray*}
\norm{P_Mu}_{L_{t,x}^6([0,\lambda^2\left(\log \lambda N\right)^{-C}]\times \T_\lambda)}\leq C \left(1+\frac{\log M}{\log N}\right)^\frac{C}{6}  \norm{u}_{Y^0}.
    \end{eqnarray*} 
\begin{proof}
Define $T:=\lambda^2\left(\log \lambda N\right)^{-C} $ and let $u$ be a $U^6_S$ atom    
    \begin{equation*}
        u(t)=\sum_{k=1}^K \chi_{[t_{k-1},t_k)}(t) S_\lambda(t)u_{k-1},
    \end{equation*}
where $\{u_k\}_{k=0}^{K-1} \subset L^2(\T_\lambda)$ such that $\sum_{k=0}^{K-1} \norm{u_k}_{2}^6 =1$. Then by Lemma \ref{lemma_Strichartz_Linear} followed by Hölder's inequality, we have
\begin{align*}
    \norm{\chi_{[0,T]}P_Mu}_{L^6_{t,x}}^6 &\leq  \sum_{k=1}^K \norm{\chi_{[0,T]}S_\lambda(t) P_Mu_{k-1}}_{L^6_{t,x}}^6, \\
    &\lesssim \left(1+\frac{\log M}{\log N}\right)^{C}  \sum_{k=1}^K \norm{ u_{k-1}}_{2}^6, \\
    &\lesssim \left(1+\frac{\log M}{\log N}\right)^{C}.
\end{align*}
From the above estimate, it immediately follows that for any $u \in U^6_S$
\begin{equation*}
\norm{\chi_{[0,T]}P_Mu}_{L^6_{x,t}}\lesssim \left(1+\frac{\log M}{\log N}\right)^\frac{C}{6} \norm{u}_{U_S^6}.
\end{equation*}
Therefore, for any $u \in Y^0$, we can use Lemma \ref{lemma: U^p/V^p embeddings} followed by Lemma \ref{lemma: Y^s properties} to conclude that
\begin{equation*}
  \norm{\chi_{[0,T]}P_Mu}_{L^6_{x,t}} \lesssim \left(1+\frac{\log M}{\log N}\right)^\frac{C}{6}  \norm{u}_{V^2_S} \lesssim \left(1+\frac{\log M}{\log N}\right)^\frac{C}{6} \norm{u}_{Y^0}. \qedhere
\end{equation*}
\end{proof}
\end{lemma}
\begin{lemma}
    \label{lemma: Trilinear Y^0 Strichartz}
There exists a constant $C>1$ such that for any functions $u_1, u_2 \in Y^0(\mathbb{R}\times \T_\lambda \rightarrow \C) $ and any $ N \geq 1$, the following estimate holds
    \begin{equation*}
        \norm{P_{N_1}u_1 \left(P_{N_2}u_2\right)^2}_{L_{t,x}^2([0, \lambda^2\left(\log \lambda N\right)^{-C}] \times \T_\lambda)} \leq C\left(1+ \frac{\log N_2}{\log N} \right)^\frac{C}{2}  \norm{u_1}_{Y^0}\norm{u_2}_{Y^0}^2.        \end{equation*}
\begin{proof}
If $N_1 \lesssim N_2$, the statement follows immediately from H\"older inequality and Lemma \ref{lemma:Linear Y^0 Strichartz}. On the other hand, if $N_1 \gg N_2$, we decompose $u_1= \sum_{\substack{\abs{J} \sim  N_2}} P_{J}u_1$ where the summation is taken over intervals $J \subset \{ k \in \Z_\lambda: \frac{N_1}{2}\leq \abs{k} < N_1\}$ of length $\abs{J}=N_2$. Define $T:= \lambda^2\left(\log \lambda N\right)^{-C}$ and note that 
\begin{align*}
    \norm{P_{N_1}u_1 \left(P_{N_2}u_2\right)^2}_{L_{t,x}^2([0,T] \times \T_\lambda)} &\lesssim \left(\sum_J \norm{P_J u_1 \left(P_{N_2}u_2\right)^2}_{L_{t,x}^2([0,T]\times \T_\lambda)}^2\right)^\frac{1}{2}, \\
   &\lesssim  \left(1+\frac{\log N_2}{\log N}\right)^\frac{C}{2}  \left(\sum_J \norm{P_J u_1}_{Y^0}^2 \norm{ u_2}_{Y^0}^4\right)^\frac{1}{2}, \\
   &\lesssim  \left(1+\frac{\log N_2}{\log N}\right)^\frac{C}{2}  \norm{u_1}_{Y^0}\norm{u_2}_{Y^0}^2, 
\end{align*}
where the first line follows from almost orthogonality, the second line follows from Lemma \ref{lemma:Linear Y^0 Strichartz} and last line is due to Lemma \ref{lemma: Y^s properties}.
\end{proof}
\end{lemma}
\section{Proof of Theorem \ref{theorem:main_gwp_theorem}}\label{sec:5}
The aim of this section is to prove Theorem~\ref{theorem:main_gwp_theorem} using the $I$-method. Henceforth, let $0 < s < 1$, and unless otherwise stated, all spatial norms are taken on $\T_\lambda$, i.e.\ $H^s = H^s(\T_\lambda)$, etc. Under the small-mass assumption, it suffices to restrict attention to the defocusing case, i.e. we assume $\mu = 1$ throughout this section.

To prove Theorem~\ref{theorem:main_gwp_theorem}, we follow a strategy similar to that in \cite{mass_critical_gwp_staffilani_tzirakis,Schippa_mass_critical,McConnell_Mass_critical_d_1}. Rather than working directly with the original \eqref{eq:NLS} on $\T$, we consider the $I$-system \eqref{eq:I_system} (see \cref{sec:5.1}) on $\T_\lambda$, in order to exploit the scaling symmetry of the equation. We show that solutions to \eqref{eq:I_system} can be extended to arbitrarily large times, which yields a growth bound for solutions to \eqref{eq:NLS} on the rescaled torus $\T_\lambda$. By inverting the rescaling, we obtain polynomial-in-time bounds for the $H^s(\T)$ norm of solutions to \eqref{eq:NLS}, thereby establishing global well-posedness.

Following \cite{mass_critical_gwp_staffilani_tzirakis}, our proof is divided into two steps: (i) establishing well-posedness for the $I$-system, and (ii) deriving a growth bound for the modified energy $E^1$.

\subsection{The $I$-operator} \label{sec: I-operator}
We begin by constructing the $I$-operator and the first modified energy. Define $m_1$ to be the restriction to $\Z_\lambda$ of the following smooth and monotone multiplier which satisfies
\begin{equation*}
    m_1(x):=\begin{cases}
        1 &\text{ if } \abs{x} \leq 1, \\
        \abs{x}^{-1} &\text{ if } \abs{x} > 2,
    \end{cases}
\end{equation*}
and for which, for every \(k\in\mathbb{N}_0\), there exists a constant $C_k>0$ such that
\begin{equation*}
    \sup_{x\in\mathbb{R}}
    \Bigg|
        \frac{x^k}{m_1(x)}
        \frac{d^k m_1}{dx^k}(x)
    \Bigg|
    \le C_k.
\end{equation*}
The conditions on derivatives of $m_1$ are included to guarantee that the multiplier is sufficiently smooth for subsequent analysis (see Lemma \ref{Lemma: bound on M_6}). For $N>0$ set $m_N(x):=m_1(\frac{x}{N})$ and for $\beta\geq 0$, let $I_N^\beta$ denote the Fourier multiplier $\widehat{I_N^\beta f}=m^\beta_N\hat{f}$. Additionally, note that $I_N^\beta$ is a smoothing operator of degree $\beta$, satisfying
\begin{equation*}
    \norm{u}_{H^{s_0}(\T_\lambda)} \lesssim \norm{I_N^\beta u}_{H^{s_0+\beta}(\T_\lambda)}\lesssim_\beta N^{\beta}  \norm{u}_{H^{s_0}(\T_\lambda)},
\end{equation*}
for all $s_0 \in \R$. We stress that the implicit constant in the first inequality is independent of $\beta$ provided $N \geq 2$.

We also recall an interpolation lemma \cite[Lemma 12.1]{Invariant_Lemma} which is useful for proving local well-posedness estimates.
\begin{lemma} 
\label{lemma: I-method_invariant_lemma}
    Let $\beta>0$ and $n\geq 1$. Suppose that $Z, X_1, \dots X_n$ are translation-invariant Banach spaces and $F$ is a translation invariant $n$-linear operator such that the following estimate holds
    \begin{equation}
    \label{hypothesis: N=1}
        \norm{I_1^\beta F(u_1,\dots,u_n)}_Z \lesssim \prod_{j=1}^n \norm{I_1^\beta u_j }_{X_j},
    \end{equation}
    for all $0 \leq \beta \leq \beta_0$ and $u_1, \dots u_n$. Then one has the estimate,
        \begin{equation}
        \label{conlcusion: all N}
        \norm{I_N^\beta F(u_1,\dots,u_n)}_Z \lesssim \prod_{j=1}^n \norm{I_N^\beta u_j }_{X_j},
    \end{equation}
    for all $0 \leq \beta \leq \beta_0$,  $u_1, \dots u_n$ and $N \geq 1$, with explicit constant independent of $N$.
\end{lemma}
\begin{remark}
We will use Lemma \ref{lemma: I-method_invariant_lemma} with $\beta=1-s$ where $0<s<1$ and with the translation-invariant Banach spaces $Z=X_1=  \dots=X_n=H^1(\T_\lambda)$ or $Z=X_1=  \dots=X_n=Y^1(\R \times \T_\lambda \rightarrow \C)$. For details, see Lemma \ref{lemma:multilinear_estimates_lwp}. Under these assumptions, note that $\norm{I_1^{1-s} f}_{H^1} \sim \norm{f}_{H^1}$ where the implicit constants are independent of $s$. Moreover, if the implicit constant in the hypothesis \eqref{hypothesis: N=1} is independent of $s$, then the constant in the conclusion \eqref{conlcusion: all N} will also not depend on $s$.
\end{remark}
For the remainder of this paper, we define $m:=m_N^{1-s}$ and $I=I_N^{1-s}$ for some $N \gg 1$ to be specified later. The first modified energy associated with the $I$-operator is defined by
\begin{equation*}
    E^1\left(u(t)\right):= E\left(Iu(t)\right)= \frac{1}{2}\int_{\T_\lambda} \abs{\partial_xIu(x,t)}^2 \,\, dx + \frac{1}{6} \int_{\T_\lambda} \abs{Iu(x,t)}^{6}\,\, dx.
\end{equation*}
\subsection{Local well-posedness for the $I$-system}\label{sec:5.1}
The first step in the proof of Theorem \ref{theorem:main_gwp_theorem} is establishing a bound for $\norm{Iu}_{Y^1_T}$ for some long enough $T>0$. To do this we apply the $I$-operator to \eqref{eq:NLS} and use a fixed point argument on the following initial value problem, 
\begin{equation}
    \label{eq:I_system}
    \begin{cases}
       i \partial_t Iu+\Delta Iu=I\left(\left\lvert u \right \rvert^4 u\right), \\
       Iu(x,0)=Iu_0(x) \in H^{1}.
    \end{cases}
    \tag{I-NLS}
\end{equation}
In particular, we prove the following result.
\begin{lemma}
\label{lemma: lwp_ILNS}
There exist constants $\delta>0$ and $c>1$ such that the following holds: for any data $u_0$ satisfying $\norm{I u_0}_{H^1} \leq \delta$ and any $s>0$, there exists a constant $N_0 = N_0(s,c,\delta) > 0$ such that for all $N > N_0$, the solution $u$ to \eqref{eq:NLS} satisfies
\[
\norm{Iu}_{Y^1_T} \leq 2\delta,
\]
where $T := \lambda^{2}\log(\lambda N)^{-c}.$
\end{lemma}
\noindent By standard iterative methods and the properties of $Y^s$ spaces, Lemma \ref{lemma: lwp_ILNS} follows immediately from the following multilinear estimate.
\begin{lemma}
\label{lemma:multilinear_estimates_lwp}
There exits constants $c>1$ and $C>1$ with the following property: for every $0<s \leq 1$ there exists a constant $N_0=N_0(s,c,C)>1$ such that for any $N>N_0$, the following estimate holds
\begin{equation*}
    \norm{\chi_{[0,T)}\int_0^t S_\lambda(t-t')I(u_1u_2u_3u_4u_5) dt'}_{Y^1} \leq C \prod_{j=1}^5 \norm{Iu_j}_{Y^1},
\end{equation*}
where $T:=\lambda^{2}\log(\lambda N)^{-c}$. Moreover, in the above estimate, each $u_j$ may be replaced by its complex conjugate $\overline{u}_j$.
\begin{proof}
First note that by Lemma \ref{lemma: I-method_invariant_lemma}, it suffices to show that
\begin{equation*}
    \norm{\chi_{[0,T)}\int_0^t S_\lambda(t-t')u_1u_2u_3u_4u_5 dt'}_{Y^s} \lesssim \prod_{j=1}^5 \norm{u_j}_{Y^s}.
\end{equation*}
By Lemma \ref{lemma: Y^s properties}, we have that
\begin{equation*}
     \norm{\chi_{[0,T)}\int_0^t S_\lambda(t-t')u_1u_2u_3u_4u_5 dt'}_{Y^s} \lesssim \sup_{\norm{v}_{Y^{-s}}\leq 1} \abs{\int_0^T\int_{\T_\lambda}u_1u_2u_3u_4u_5\overline{v} dx dt}.
\end{equation*}
We perform a Paley-Littlewood decomposition, $u_j=\sum_{L_j} P_{L_j}u_j$ for $j=1,\dots, 5$ and $v=\sum_{M} P_M v$ where $L_1, \dots L_5, M \in  2^{k_s\N_0}$ and $k_s:=\frac{1}{s}$. We therefore need to bound
\begin{equation*}
    \sum_{L,M} \mathcal{I}(L,M) := \sum_{L,M} \int_0^T\int_{\T_\lambda}\abs{P_{L_1}u_1P_{L_2}u_2P_{L_3}u_3P_{L_4}u_4P_{L_5}u_5P_M v} dx dt,
\end{equation*}
where $L=(L_1,L_2,L_3,L_4,L_5)$ and the summation is restricted to $M \lesssim L_1+\dots+L_5$ due to orthogonality. Without loss in generality we can assume that $L_1 \geq \cdots \geq L_5$ and so $M \lesssim L_1$. Also note that each $u_j$ can be replaced by its complex conjugate $\overline{u}_j$ in the definition of $\mathcal{I}(L,M)$.

We split the summation into two regions: (a) $L_1 \lesssim  N$ and (b) $L_1 \gg N$. First consider case (a) $L_1 \lesssim  N$. Then,
\begin{align*}
    \sum_{\substack{L,M\\ L_1 \lesssim N}} \mathcal{I}(L,M) &\leq \sum_{L,M} \prod_{j=1}^5 \norm{\chi_{[0,T)}P_{L_j}u_j}_{L^6_{t,x}} \norm{\chi_{[0,T)}P_{M}v}_{L^6_{t,x}}, \\
    &\lesssim \sum_{L,M}\prod_{j=1}^5 \norm{P_{L_j}u_j}_{Y^0} \norm{P_{M}v}_{Y^0}, \\
     &\lesssim \prod_{j=2}^5 \norm{u_j}_{Y^s} \sum_{M\lesssim L_1} \left(\frac{M}{L_1}\right)^s \norm{P_{L_1}u_1}_{Y^s} \norm{P_{M}v}_{Y^{-s}}, \\
&\lesssim \prod_{j=1}^5 \norm{u_j}_{Y^{s}} \norm{v}_{Y^{-s}},
\end{align*}
where the first line follows from H\"older's inequality, the second line follows from Lemma \ref{lemma:Linear Y^0 Strichartz} and the last two lines follow from summing geometric progressions and the fact that $s>0$. The selection of $k_s=\frac{1}{s}$ ensures that the implicit constants in the above estimates do not depend on $s$.

Next consider case (b): $L_1 \gg N$. We further split of case (b) into two regions: (i) $M \lesssim L_2$ and (ii) $M \gg L_2$. In sub-case (i), $M \lesssim L_2$ and so $L_1 \sim L_2$. Therefore, an analogous argument to case (a) yields
\begin{align*}
    \sum_{\substack{L,M \\ L_1 \sim L_2}} \mathcal{I}(L,M) &\leq \sum_{\substack{L,M \\ L_1 \sim L_2}}\prod_{j=1}^5 \norm{\chi_{[0,T)}P_{L_j}u_j}_{L^6_{t,x}} \norm{\chi_{[0,T)}P_{M}v}_{L^6_{t,x}}, \\
    &\lesssim \sum_{\substack{L,M \\ L_1 \sim L_2}} \left(\frac{\log L_2}{\log N}\right)^{c} \prod_{j=1}^5 \norm{P_{L_j}u_j}_{Y^0} \norm{P_{M}v}_{Y^0}, \\
     & \lesssim \frac{N^{-s}}{s} \prod_{j=2}^5 \norm{u_j}_{Y^s} \sum_{M \lesssim  L_1} \left(\frac{M}{L_1}\right)^s \norm{P_{L_1}u_1}_{Y^s} \norm{P_{M}v}_{Y^{-s}}, \\
    &\lesssim \frac{N^{-s}}{s} \prod_{j=1}^5 \norm{u_j}_{Y^s} \norm{v}_{Y^{-s}},
\end{align*}
where the first line follows from H\"older's inequality, the second line follows from Lemma \ref{lemma:Linear Y^0 Strichartz} and the fact that $L_2 \sim L_1 \gg N$. The last two estimates follow from summing geometric progressions and the fact that $s>0$.

On the other hand, for region (ii): $M \gg L_2$ so we have $L_1 \sim M$. Therefore, using the trilinear bound from Lemma \ref{lemma: Trilinear Y^0 Strichartz} and noting that $L_1 \gg L_2 \geq L_3$, we obtain
\begin{align*}
\norm{\chi_{[0,T)} \prod_{j=1}^3 P_{L_j}u_j}_{L^2_{t,x}}^2 &\leq \norm{\chi_{[0,T)}P_{L_1} u_1\left(P_{L_2}u_2\right)^2}_{L^2_{t,x}} \norm{\chi_{[0,T)}P_{L_1}u_1\left(P_{L_3}u_3\right)^2}_{L^2_{t,x}}, \\
& \lesssim \left(1+\frac{\log L_2}{\log N}\right)^{c}  \norm{P_{L_1}u_1}_{Y^0}^2 \norm{P_{L_2}u_2}_{Y^0}^2\norm{P_{L_3}u_3}_{Y^0}^2.
\end{align*}
An identical argument and the fact that $M \gg L_4 \geq L_5$ yields,
\[
\norm{\chi_{[0,T)} P_M v \prod_{j=4}^5 P_{L_j}u_j}_{L^2_{t,x}} \lesssim \left(1+\frac{\log L_4}{\log N}\right)^\frac{c}{2}   \norm{P_{L_M}v}_{Y^0} \norm{P_{L_4}u_4}_{Y^0}\norm{P_{L_5}u_5}_{Y^0}.
\]
We can thus conclude,
\begin{align*}
\sum_{ L_1 \sim M \gg L_2}  \mathcal{I}(L,M) &\leq \sum_{ L_1 \sim M \gg L_2} \norm{\chi_{[0,T)} \prod_{j=1}^3 P_{L_j}u_j}_{L^2_{t,x}} \norm{\chi_{[0,T)} P_M v \prod_{j=4}^5 P_{L_j}u_j}_{L^2_{t,x}}, \\
&\lesssim \sum_{ L_1 \sim M \gg L_2 } \left(1+\frac{\log L_2}{\log N}\right)^{c} \prod_{j=1}^5 \norm{P_{L_j}u_j}_{Y^0}, \\
&\lesssim \left(1+\frac{N^{-s}}{s}\right) \prod_{j=1}^5 \norm{u_j}_{Y^s} \norm{v}_{Y^{-s}}. \qedhere
\end{align*}
\end{proof}
\end{lemma}

\subsection{Growth of $E^1$}
Next, we proceed with bounding the growth of the modified energy. Note that if $u$ is a smooth solution to \eqref{eq:NLS}, then
\begin{align*}
    \partial_t E^1(u(t)) &= \Re \int_{\T_\lambda} \overline{\partial_t Iu } \left( \abs{Iu}^4 Iu - \Delta Iu \right),\\ 
    &= \Re \int_{\T_\lambda} \overline{\partial_t Iu} \left( \abs{Iu}^4 Iu - I(\abs{u}^4 u) \right), \\
    &= \Im \int_{\T_\lambda} \left(\overline{\Delta I u}- \overline{I(\abs{u}^4u)} \right) \left( \abs{Iu}^4 Iu - I(\abs{u}^4 u) \right).
\end{align*}
Therefore, by the fundamental theorem of calculus, for any $t_0 \in \R$ and any $T>0$, we have
\begin{equation}
    \label{eq: fundamental_theorem_calculus_E1}
    E^1(u(t_0+T))-E^1(u(t_0)) = \int_{t_0}^{t_0+T} \left(\Lambda_6(u(t)) +\Lambda_{10}(u(t) \right)\,\,dt, 
\end{equation}
where we define  
\begin{align*}
    \Lambda_6(u) &:= \Im \int_{\Gamma_6} M_6(k_1, \dots, k_6) \widehat{\overline{\Delta Iu}} (k_1) \widehat{u}(k_2) \widehat{\overline{u}}(k_3) \widehat{u}(k_4) \widehat{\overline{u}}(k_5) \widehat{u}(k_6), \\
    \Lambda_{10}(u) &:=  \frac{i}{2}\int_{\Gamma_{10}} M_{10}(k_1, \dots, k_{10} ) \prod_{j=1}^5 \widehat{u} (k_{2j-1}) \widehat{\overline{u}}(k_{2j}).
\end{align*}
Here the hyperplane $\Gamma_n=\left\{ (k_1, \dots, k_n) \in \Z_\lambda^n : k_1+\cdots+k_n=0 \right\}$ is endowed with measure $\prod_{j=1}^{n-1} (dk_j)_\lambda$ and the multipliers $M_6: \Gamma_6 \rightarrow \R$ and $M_{10}: \Gamma_{10} \rightarrow \R$ are given by
\begin{align*}
    M_6(k_1, \dots, k_6)&:=\prod_{j=2}^6 m(k_j) - m(k_1),  \\
     M_{10}(k_1,\dots,k_{10}) &:=m(\sum_{j=6}^{10} k_j) \prod_{j=1}^{5} m(k_j)- m(\sum_{j=1}^5 k_j)\prod_{j=6}^{10} m(k_j).
     \end{align*}
Our goal is to bound the right-hand side of \eqref{eq: fundamental_theorem_calculus_E1} by using the long-time linear and trilinear Strichartz estimates from Lemma \ref{lemma:Linear Y^0 Strichartz} and Lemma \ref{lemma: Trilinear Y^0 Strichartz}. The remainder of this section will focus on proving following estimate.
\begin{lemma}
\label{lemma: growth E1}
Fix $c>1$ as in Lemma~\ref{lemma: lwp_ILNS} and set $T := \lambda^2 \log(\lambda N)^{-c}$. Then, for every $s>0$ and every $0<\epsilon \ll_s 1$, there exists a constant $C=C(s,\epsilon)>1$ such that for any smooth solution $u:\T_\lambda\times \R \rightarrow \mathbb{C}$ to \eqref{eq:NLS} and any $t_0\in \R$, the following estimate holds
\begin{equation*}
    E^1(u(t_0+T)) -E^1(u(t_0)) \leq C N^{-1+\epsilon}\left(\norm{Iu}_{Y^1_{J_T}}^6+\norm{Iu}_{Y^1_{J_T}}^{10}\right),
\end{equation*}
where $J_T:=[t_0,t_0+T]$.
\end{lemma}
We prove Lemma \ref{lemma: growth E1} by separately estimating the contributions of $\Lambda_6$ and $\Lambda_{10}$ in \eqref{eq: fundamental_theorem_calculus_E1}. To this end, we first we first state a helpful lemma. 
\begin{lemma}
\label{Lemma: bound on M_6}
Let $\mathcal{M}: \Gamma_6 \rightarrow \R$ and suppose there exists an extension $\widetilde{\mathcal{M}}: \R^6 \rightarrow \R$ satisfying 
\[\widetilde{\mathcal{M}}(\xi_1,\dots, \xi_6)= \mathcal{M}(\xi_1, \dots, \xi_6) \text{ for all } (\xi_1, \dots, \xi_6) \in \Gamma_6,
\]
and for dyadic $N_1,\dots,N_6\in 2^{\N_0}$ and all $0\le\alpha\le7$, there exists a constant 
$C_1(N_1,\dots,N_6)>0$ such that
\begin{equation}
    \sup_{\substack{|\xi_k|\sim N_k\\ 1\le k\le6}} \abs{\partial_{\xi_j}^\alpha \widetilde{\mathcal{M}}(\xi_1, \dots, \xi_6)} \lesssim_{\alpha} \frac{C_1(N_1,\dots, N_6)}{N_j^\alpha} \text{ for all } 1\leq j \leq 6.
    \label{eq: hypothesis_1}
\end{equation}
Moreover, suppose that there exists a constant $C_2(N_1, \dots, N_6)>0$ such that
\begin{equation}
    \sup_{x_1, \dots, x_6 \in \R}\int_{0}^T \int_{\T_\lambda} \abs{ \prod_{j=1}^6 P_{N_j} f_j(x+x_j,t)} dx dt  \leq C_2(N_1, \dots, N_6) \prod_{j=1}^6 \norm{f_j}_{Y^0}.
    \label{eq: hypothesis_2}
\end{equation}
Then, 
\begin{equation*}
    \abs{\int_{0}^T \int_{\Gamma_6} \mathcal{M}(k_1,\dots, k_6)\prod_{j=1}^6 P_{N_j} \hat{f}_j(k_j,t) dk_\lambda dt} \leq C_3(N_1, \dots, N_6) \prod_{j=1}^6 \norm{f_j}_{Y^0},
\end{equation*}
where $C_3(N_1, \dots, N_6)\sim C_1(N_1, \dots, N_6)C_2(N_1, \dots, N_6)$.
\begin{proof}
This lemma follows from the Fourier series expansion of $\widetilde{\mathcal{M}}$ followed by Plancherel's theorem. For details, see Proposition 5.4 and Remark 5.5 in \cite{Schippa_mass_critical}.
\end{proof}
\end{lemma}
By combining Lemma \ref{Lemma: bound on M_6} with the long-time Strichartz estimates from Lemmata \ref{lemma:Linear Y^0 Strichartz} and \ref{lemma: Trilinear Y^0 Strichartz} we obtain the following estimates.
\begin{lemma}
\label{lemma: growth E1_six}
Fix $c>1$ as in Lemma~\ref{lemma: lwp_ILNS} and set $T := \lambda^2 \log(\lambda N)^{-c}$. Then, for every $s>0$ and every $0<\epsilon \ll_s 1$, there exists a constant $C=C(s,\epsilon)>1$ such that for any function $u: \T_\lambda\times \R \rightarrow \mathbb{C}$ satisfying $Iu \in Y^1$ and any $t_0\in \R$, the following estimate holds
\begin{equation*}
   \abs{\int_{t_0}^{t_0+T} \Lambda_6(u(t)) \,\, dt }\leq C N^{-1+\epsilon}\norm{Iu}_{Y^1}^6.
\end{equation*}
\begin{proof}
    Without loss in generality, we may assume $t_0=0$. We dyadically decompose $u = \sum_{j\geq 1} u_j$ where $u_j:=P_{N_j} u$ for dyadic $N_j \in 2^{\N_0}$. With this decomposition in hand, it suffices to prove that
    \[ \sum_{N_1, \dots N_6} \abs{\int_0^T \int_{\Gamma_6} M_6 \widehat{\overline{\Delta Iu}}_1 \widehat{u}_2 \widehat{\overline{u}}_3 \widehat{u}_4 \widehat{\overline{u}}_5 \widehat{u}_6} \lesssim N^{-1+} \norm{Iu}_{Y^1}^6.  \]
We will use Lemma \ref{Lemma: bound on M_6} to prove the above estimate. As left-hand side of \eqref{eq: hypothesis_2} is invariant under complex conjugation and $M_6$ is symmetric under permutations of $k_2, \dots, k_6$, we may assume without loss of generality that $N_2 \geq \cdots \geq N_6$. Furthermore, since we are integrating on $\Gamma_6$, $N_1 \lesssim N_2$ also holds. We also stress that $M_6$ vanishes unless $N_2 \gtrsim N$. Thus we split the summation into two cases: (a) $N_3 \ll N \lesssim N_2$ and (b) $N_2 \geq N_3 \gtrsim N$.

For case (a), we have that $N_1 \sim N_2 \gtrsim N \gg N_3$ and $k_1k_2<0$. We write $M_6(k_1, \dots, k_6) = m(k_2)-m(k_2+\dots+k_6)$ and note that by the mean value theorem,
\[
\abs{m(k_2)-m(k_2+\dots+k_6)} \lesssim  \abs{k_3+\dots+k_6} m'(N_2) \lesssim \frac{N_3}{N_2} m(N_2). 
\]
Using a similar argument, we can bound the partial derivatives by the same bound and conclude that $C_1(N_1,\dots, N_6) \lesssim \frac{N_3}{N_2}m(N_2)$ where $C_1$ is the constant defined in hypothesis \eqref{eq: hypothesis_1} of Lemma \ref{Lemma: bound on M_6}. We next use H\"older's inequality followed by Lemma \ref{lemma:Linear Y^0 Strichartz}, to obtain
\begin{align*}
\int_0^T \int_{\T_\lambda} \abs{\Delta I u_1 \prod_{j=2}^6 u_j} &\leq \norm{\chi_{[0,T)}\Delta Iu_1}_{L^6_{t,x}} \prod_{j=2}^6 \norm{\chi_{[0,T)}u_j}_{L^6_{t,x}} \\
&\lesssim N_2^{0+}\norm{\Delta Iu_1}_{Y^0} \prod_{j=2}^6 \norm{u_j}_{Y^0}.
\end{align*}
Note that since Strichartz estimate are translation invariant, we conclude that $C_2(N_1, \dots, N_6) \lesssim N_2^{0+}$ where $C_2$ is the constant defined in hypothesis \eqref{eq: hypothesis_2} of Lemma \ref{Lemma: bound on M_6}. Therefore, by Lemma \ref{Lemma: bound on M_6}, we have shown that when $N_3 \ll N \lesssim N_2$, we have the estimate
\begin{align*}
    \abs{\int_0^T \int_{\Gamma_6} M_6 \widehat{\overline{\Delta Iu}}_1 \widehat{u}_2 \widehat{\overline{u}}_3 \widehat{u}_4 \widehat{\overline{u}}_5 \widehat{u}_6} &\lesssim N_2^{0+}\frac{N_3}{N_2} m(N_2) \norm{\Delta Iu_1}_{Y^0} \prod_{j=2}^6 \norm{u_j}_{Y^0}  \\
    &\lesssim N_2^{0+}\frac{N_3}{N_2}\frac{N_1}{N_2N_3N_4N_5N_6} \prod_{j=1}^6 \norm{Iu_j}_{Y^1}  \\
     &\lesssim N^{-1+}N_2^{0-} \prod_{j=1}^6 \norm{Iu_j}_{Y^1_T},
\end{align*}
where the second line follows by noting that $N_3 \ll N $ and so $m(N_j)=1$ for $j=3,4,5,6$. The last line follows since $N_1 \sim N_2 \gtrsim N$ and $N_4,N_5, N_6 \geq 1$. The desired bound is then obtained by summing geometric progressions.

Now consider case (b): $N_3 \gtrsim N $ and further split the summation into two sub-cases (i) $ N_3 \sim N_2$ and (ii) $N_3 \ll N_2 \sim N_1$. For both sub-cases, we have the trivial bound 
\[\abs{M_6(k_1, \dots, k_6)} \lesssim m(k_1)\] 
and the same upper bound holds for the partial derivatives so  $C_1(N_1,\dots, N_6) \lesssim m(N_1)$ where $C_1$ is the constant defined in hypothesis \eqref{eq: hypothesis_1} of Lemma \ref{Lemma: bound on M_6}. 

For sub-case (b)(i): $N_2 \sim N_3 \gtrsim N$, we proceed in the similar way as we did in case (a) to obtain $C_2(N_1, \dots, N_6) \lesssim N_2^{0+}$ where $C_2$ is the constant defined in hypothesis \eqref{eq: hypothesis_2} of Lemma \ref{Lemma: bound on M_6}. An application of Lemma \ref{Lemma: bound on M_6} yields,
\begin{align*}
    \abs{\int_0^T \int_{\Gamma_6} M_6 \widehat{\overline{\Delta Iu}}_1 \widehat{u}_2 \widehat{\overline{u}}_3 \widehat{u}_4 \widehat{\overline{u}}_5 \widehat{u}_6} &\lesssim N_2^{0+}m(N_1) \norm{\Delta Iu_1}_{Y^0_T} \prod_{j=2}^6 \norm{u_j}_{Y^0_T}  \\
    &\lesssim N_2^{0+}\frac{m(N_1)N_1}{\prod_{j=2}^6 N_jm(N_j)} \prod_{j=1}^6 \norm{Iu_j}_{Y^1_T}  \\
     &\lesssim N^{-1+}N_2^{0-} \prod_{j=1}^6 \norm{Iu_j}_{Y^1_T},
\end{align*}
where in the last line we used the fact $m(N_j)N_j^{1-\epsilon} \gtrsim \min(N,N_j)^{1-\epsilon}$ for every $\epsilon \in [0,s]$. We then obtain the desired bound by summing geometric progressions.

Next consider sub-case (b)(ii): $N_1 \sim N_2 \gg N_3 \gtrsim N$. Here we cannot use the same bounding strategy as sub-case (b)(i) as $N_2 \gg N_3$ and the derivative loss of $N_2^{0+}$ cannot be removed by the $N_3$ factor. Instead, we use H\"older's inequality followed by the trilinear estimate from Lemma \ref{lemma: Trilinear Y^0 Strichartz} to obtain
\begin{align*}
    \int_0^T \int_{\T_\lambda} \abs{\Delta I u_1 \prod_{j=2}^6 u_j} &\leq \norm{\chi_{[0,T)}\Delta Iu_1 u_3 u_4}_{L^2_{t,x}} \norm{\chi_{[0,T)}u_2 u_5 u_6}_{L^2_{t,x}},  \\
    &\leq \prod_{j=3}^4 \norm{\chi_{[0,T)}\Delta Iu_1 u_j^2 }_{L^2_{t,x}}^\frac{1}{2} \prod_{j=5}^6 \norm{\chi_{[0,T)}u_2 u_j^2 }_{L^2_{t,x}}^\frac{1}{2}, \\ 
    &\lesssim N_3^{0+} \norm{\Delta Iu_1}_{Y^0_T} \prod_{j=2}^6 \norm{u_j}_{Y^0}.
\end{align*}
Hence, an application of Lemma \ref{Lemma: bound on M_6} yields
\begin{align*}
    \abs{\int_0^T \int_{\Gamma_6} M_6 \widehat{\overline{\Delta Iu}}_1 \widehat{u}_2 \widehat{\overline{u}}_3 \widehat{u}_4 \widehat{\overline{u}}_5 \widehat{u}_6} &\lesssim N_3^{0+} m(N_1) \norm{\Delta Iu_1}_{Y^0_T} \prod_{j=2}^6 \norm{u_j}_{Y^0_T},  \\
    &\lesssim N_3^{0+}\frac{m(N_1)N_1}{\prod_{j=2}^6 N_jm(N_j)} \prod_{j=1}^6 \norm{Iu_j}_{Y^1_T}, \\
     &\lesssim N^{-1+}N_3^{0-} \prod_{j=1}^6 \norm{Iu_j}_{Y^1_T}.
\end{align*}
We conclude by summing geometric progressions and observing that $N_1 \sim N_2$, so the lack of decay in $N_2$ is harmless.
\end{proof}
\end{lemma}
\begin{lemma}
\label{lemma: growth E1_ten}
Fix $c>1$ as in Lemma~\ref{lemma: lwp_ILNS} and set $T := \lambda^2 \log(\lambda N)^{-c}$. Then, for every $s>0$ and every $0<\epsilon \ll_s 1$, there exists a constant $C=C(s,\epsilon)>1$ such that for any function $u: \T_\lambda\times \R \rightarrow \mathbb{C}$ satisfying $Iu \in Y^1$ and any $t_0\in \R$, the following estimate holds
\begin{equation*}
   \abs{\int_{t_0}^{t_0+T} \Lambda_{10}(u(t)) \,\, dt }\leq C N^{-2+\epsilon}\norm{Iu}_{Y^1}^{10}.
\end{equation*}
\begin{proof}
    Just as in the proof of Lemma \ref{lemma: growth E1_six}, it suffices to consider $t_0=0$. We dyadically decompose $u = \sum_{j\geq 1} u_j$ where $u_j:=P_{N_j} u$ for dyadic $N_j \in 2^{\N_0}$. Let $N_1^\star \geq \cdots \geq N_{10}^\star$ denote the decreasing ordering of  $(N_1, \dots, N_{10})$ and define $u_j^\star:=P_{N_j^\star}u$. We will prove that,
\[ \sum_{N_1, \dots, N_{10}} \abs{\int_0^T \int_{\Gamma_6} M_{10}\widehat{u}_1 \widehat{\overline{u}}_2 \widehat{u}_3 \widehat{\overline{u}}_4 \widehat{u}_5 \widehat{\overline{u}}_{6} \widehat{u}_7 \widehat{\overline{u}}_8 \widehat{u}_{9}\widehat{\overline{u}}_{10}} \lesssim N^{-2+} \norm{Iu}_{Y^1}^{10}. \]
First note that $M_{10}$ vanishes unless $N_1^\star \gtrsim N$. Therefore, since we are integrating in $\Gamma_{10}$ we may restrict the summation to $N_1^\star \sim N_2^\star \gtrsim N$. Next, by the triangle inequality, it suffices to show that,
\[ \max_{j=1,2} \sum_{N_1, \dots, N_{10}} \abs{\int_0^T \int_{\Gamma_6} M_{10}^{(j)}\widehat{u}_1 \widehat{\overline{u}}_2 \widehat{u}_3 \widehat{\overline{u}}_4 \widehat{u}_5 \widehat{\overline{u}}_{6} \widehat{u}_7 \widehat{\overline{u}}_8 \widehat{u}_{9}\widehat{\overline{u}}_{10}} \lesssim N^{-2+} \norm{Iu}_{Y^1}^{10}, \]
where
\[M_{10}^{(j)}(k_1,k_2, \dots, k_{10}):= m\left(\sum_{l \in J_j} k_l\right) \prod_{l \in J_j^c} m(k_l),\]
$J_{1}:= \{6, \dots 10\},$ $J_{2}:=\{1,\dots 5\}$ and $J_j^c$ denotes the complement of $J_j$ in $\{1,\dots, 10\}.$ We view $M_{10}^{(j)}$ as a multiplier acting on $\Gamma_6$ and apply Lemma \ref{Lemma: bound on M_6}. By symmetry, it suffices to consider $j=1$. To this end, we further decompose the summation
\[ \sum_{N_1, \dots, N_{10},K } \abs{\int_0^T \int_{\Gamma_6} M_{10}^{(1)}\widehat{u}_1 \widehat{\overline{u}}_2 \widehat{u}_3 \widehat{\overline{u}}_4 \widehat{u}_5 \widehat{f}_K}:= \sum_{N_1, \dots, N_{10},K } \mathcal{I}(N_1, \dots, N_{10}, K) , \]
where $\hat{f}_K(k):= \sum_{\substack{ k_6+\dots+ k_{10}=k \\ \abs{k} \sim K}} \widehat{\overline{u}}_6(k_{6}) \widehat{u}_7(k_{7}) \widehat{\overline{u}}_8(k_{8}) \widehat{u}_9(k_{9}) \widehat{\overline{u}}_{10}(k_{10})$. By the invariance of left-hand side in hypothesis \eqref{eq: hypothesis_2} under complex conjugation together with the permutation symmetry of $M_{10}^{(1)}$, we may further restrict the summation to {$N_1 \geq N_2 \geq \cdots \geq N_5$} and $N_6 \geq N_7 \geq \cdots \geq N_{10}$. We split the summation into two cases: (a) $N_1^\star=N_1, N_2^\star=N_2$ and (b) at least one of $N_1^\star, N_2^\star$ lies in $\{N_6, \dots, N_{10}\}$.

In both case we have the trivial bound $m(K) \leq 1$, so it is clear that 
\[\abs{M_{10}^{(1)}(k_1, \dots, k_{10})} \leq \prod_{j=1}^{5} m(N_j),\]
and analogous bounds hold for the partial derivatives, so that 
\[C_1(N_1, \dots, N_5, K) \lesssim \prod_{j=1}^{5} m(N_j),\]
where $C_1$ is the constant defined in hypothesis \eqref{eq: hypothesis_1} of Lemma \ref{Lemma: bound on M_6}. 

Next consider case (a): $N_1^\star=N_1, N_2^\star=N_2$. By H\"older's inequality followed by Lemma \ref{lemma: Trilinear Y^0 Strichartz}, we have that,

\begin{align*}
\int_0^T \int_{\T_\lambda} \abs{\prod_{j=1}^{5} u_j\, f_K}
&\le \Bigl\| \prod_{j=1}^5 \chi_{[0,T)} u_j \Bigr\|_{L^2_{t,x}}
      \, \| \chi_{[0,T)} f_K \|_{L^2_{t,x}}, \\
&\le \prod_{j=1}^3 \|\chi_{[0,T)} u_j\|_{L^6_{t,x}}
     \prod_{j=4}^5 \|\chi_{[0,T)} u_j\|_{L^\infty_{t,x}} \\
&\qquad \times
     \prod_{j=6}^8 \|\chi_{[0,T)} u_j\|_{L^6_{t,x}}
     \prod_{j=9}^{10} \|\chi_{[0,T)} u_j\|_{L^\infty_{t,x}}, \\
&\lesssim N_1^{0+} \left(N_4 N_5 N_9 N_{10}\right)^{1/2}
           \prod_{j=1}^{10} \|u_j\|_{Y^0}.
\end{align*}


Therefore, by Lemma \ref{Lemma: bound on M_6} we obtain
\begin{align*}
  \mathcal{I} &\lesssim N_1^{0+} \left(N_4 N_5 N_9 N_{10}\right)^{1/2} \prod_{j=1}^{5} \norm{Iu_j}_{Y^0} \prod_{j=6}^{10} \norm{u}_{Y^0}, \\
    &\lesssim N_1^{0+} \frac{\left(N_4 N_5 N_9 N_{10}\right)^{1/2}}{\prod_{j=1}^5 N_j \prod_{j=6}^{10} m(N_j)N_j} \prod_{j=1}^{10} \norm{Iu_j}_{Y^1}, \\
    &\lesssim N_1^{0+} \frac{m(N_1)m(N_2)}{m(N_1) N_1 m(N_2) N_2 m(N_9) N_9^{\frac{1}{2}} m(N_{10}) N_{10}^{\frac{1}{2}}} \prod_{j=1}^{10} \norm{Iu_j}_{Y^1}, \\
    &\lesssim N^{-2+} N_1^{0-} \prod_{j=1}^{10}\norm{Iu_j}_{Y^1}, 
\end{align*}
where the third line follows from the fact that $m(N_j)N_j^{1-\epsilon} \gtrsim \min(N,N_j)^{1-\epsilon}$ for all $\epsilon \in [0,s]$ and the last line follows from the aforementioned fact and since $N_1 \sim N_2 \geq N_9 \geq N_{10}$ implies that $m(N_1) \sim m(N_2) \leq m(N_9) \leq m(N_{10})$.
Summing geometric progressions concludes the proof for case (a). 

We now turn to case (b): at least one of $N_1^\star, N_2^\star$ lies in $\{N_6, \dots, N_{10}\}$. Here we use H\"older's inequality again but the derivative loss is placed at the frequencies $N_2, \dots, N_5$ where it can be compensated by the $m(N_2) \cdots m(N_5)$ gain. In fact,

\begin{align*}
\int_0^T \int_{\T_\lambda} \abs{\prod_{j=1}^{5} u_j\, f_K}
&\le \Bigl\| \prod_{j=1}^5 \chi_{[0,T)} u_j \Bigr\|_{L^6_{t,x}}
      \, \| \chi_{[0,T)} f_K \|_{L^\frac{6}{5}_{t,x}}, \\
&\le \|\chi_{[0,T)} u_1\|_{L^6_{t,x}}
     \prod_{j=2}^5 \|\chi_{[0,T)} u_j\|_{L^\infty_{t,x}}
     \prod_{j=6}^{10} \|\chi_{[0,T)} u_j\|_{L^6_{t,x}}, \\
&\lesssim \left(N_1^\star\right)^{0+} \left(N_2 N_3 N_4 N_{5}\right)^{1/2}
           \prod_{j=1}^{10} \|u_j\|_{Y^0}.
\end{align*}
An application of Lemma \ref{Lemma: bound on M_6} and an analogous argument to the one used above conclude the proof.
\end{proof}
\end{lemma}
\subsection{Proof of Global Well-posedness}
We are now in a position to prove Theorem \ref{theorem:main_gwp_theorem} using Lemmata \ref{lemma: lwp_ILNS} and \ref{lemma: growth E1}. Let $u: \T \times \R \rightarrow \C$ denote the solution to \eqref{eq:NLS} with initial condition $u_0 \in H^s(\T)$. By assumption, $u_0$ has small mass, so we rescale the initial data so that $Iu_0^\lambda$ has small $H^1$-norm, where \[ u_0^\lambda= \lambda^{-\frac{1}{2}}u_0\left(\frac{x}{\lambda}\right),\]
and $u^\lambda : \T_\lambda \times \R \rightarrow \C$ denotes the corresponding solution to \eqref{eq:NLS} with initial data $u_0^\lambda$.

Applying Lemma \ref{lemma: lwp_ILNS} yields a bound for the $H^1$ norm of $Iu^\lambda(t)$ for all $t \in [0,T_0]$ for some fixed $T_0>0$. The strategy is to iterate this local result: we use Lemma \ref{lemma: growth E1} to bound the growth of the modified energy $E^1(u^\lambda(t))$ on $[0,T_0]$ which in turn controls the $\dot{H}^s$-norm of $Iu^\lambda(T_0)$. Combined with the conservation of mass, this provides uniform control of the $H^s$ norm of $Iu^\lambda(T_0)$ for sufficiently large $N$, allowing Lemma \ref{lemma: lwp_ILNS} to be applied successively with initial condition $Iu^\lambda(T_0)$. By repeating this argument and reversing the rescaling, we extend the original solution $u$ globally in time by letting $N \rightarrow \infty$. We now turn to details.
\begin{proof}[Proof of Theorem  \ref{theorem:main_gwp_theorem}] 
Suppose that $0<s<1$ and let $u_0 \in H^s(\T)$ be such that $\norm{u_0}_2 = \delta$ where $0<\delta < 1$. Let $u^\lambda$ denote the solution to (\ref{eq:NLS}) with initial condition $u^\lambda(x,0)=u^\lambda_0(x):= \lambda^{-\frac{1}{2}} u_0(\frac{x}{\lambda})$. We stress that $u^\lambda_0 : \T_\lambda \rightarrow 
\C$ whereas $u_0: \T \rightarrow \C$ and hence all $L^p$ norms ($p \geq 1$) in the remainder of the proof are defined on the appropriate domain ($\T_\lambda$ for $u^\lambda_0$ and $\T$ for $u_0$).

We select $\lambda \sim_s  N^{\frac{1-s}{s}} \norm{u_0}_{\dot{H}^s}^{\frac{1}{s}} \norm{u_0}_2^{-\frac{1}{s}}$ such that $E^1(u_0^\lambda) \leq 2 \norm{u_0}_2^2$. In particular note that
\begin{equation*}
    \norm{\partial_x Iu_0^\lambda}_{2}^2=\norm{m(k) k \hat{u}_0^\lambda}_{L^2((dk)_\lambda)}^2 \lesssim \frac{N^{2-2s}}{\lambda^{2s}} \norm{u_0}_{\dot{H}^s}^2 \lesssim \norm{u_0}_2^2,
\end{equation*}
and by the Gagliardo-Nirenberg inequality \cite{Nirenberg}, 
\begin{align*}
    \norm{Iu_0^\lambda}_{6}^6 &\lesssim \norm{\partial_xIu_0^\lambda}_{2}^{2} \norm{Iu_0^\lambda}_{2}^{4} + \norm{Iu_0^\lambda}_{2}^{6}, \notag \\ 
    &\lesssim \frac{N^{2-2s}}{\lambda^{2s}} \norm{u_0}_{\dot{H}^s}^{2}\norm{u_0}_{2}^{4}+\norm{u_0}_{2}^{6},  \notag\\
      &\lesssim  \norm{u_0}_2^6.
\end{align*}
Therefore, since $E^1(u_0^\lambda) \leq 2 \delta^2$, we have  
\[\norm{Iu_0^\lambda}_{H^1}^2 \leq \norm{Iu_0^\lambda}_2^2+E^1(u_0^\lambda) \leq 3\delta^2.\]
We select $\delta \ll 1$ small enough and $N=N(s,\delta) \gg 1$ large enough so that Lemma \ref{lemma: lwp_ILNS} can be applied and let $T_0\sim \lambda^{2}\log(\lambda N)^{-c}$. Next, by Lemma \ref{lemma: growth E1}, we have that
\begin{align*}
    E^1(u^\lambda(T_0)) &\leq  E^1(u^\lambda(0)) + C_1 N^{-1+} \left(\norm{Iu^\lambda}_{Y^1_{T_0}}^6+\norm{Iu^\lambda}_{Y^1_{T_0}}^{10}\right), \\
    &\leq 2\delta^2+C_2 N^{-1+}\delta^6,
\end{align*} 
where the second line follows from Lemma \ref{lemma: lwp_ILNS} and $C_1, C_2 >0$ are constants depending on $s$. By possibly increasing $N=N(s,\delta) \gg 1$ one can ensure that $E^1(u^\lambda(T_0)) \leq 3\delta^2$ so that Lemma \ref{lemma: lwp_ILNS} can be applied once more on the subsequent interval $[T_0,2T_0]$. We can repeat this process and obtain solutions satisfying
\[
\norm{Iu^\lambda(nT_0)}_{H^1} \leq 2\delta,
\]
as long as $n \ll N^{1-}$. Undoing the rescaling, the corresponding solutions to \eqref{eq:NLS} satisfy
\begin{align*}
    \norm{u(t)}_{H^s(\mathbb{T})}^2 &=  \norm{u_0}_2^2+ \lambda^{2s} 
 \norm{u^\lambda(t\lambda^{2})}_{\dot{H}^s(\T_\lambda)}^2, \\
 &\lesssim \left(1 +\lambda^{2s}\right) \delta^2. \\
 &\lesssim \left(1 + N^{2(1-s)}\right) \norm{u_0}_{H^s(\T)}^2.
\end{align*}
 for $t\in [0,T]$ where $T \sim N^{1-}\log(N)^{-c} $. Therefore, we deduce that solutions can be continued for arbitrary time intervals by selecting $N \rightarrow \infty$.
\end{proof}
\appendix
\section{Refined Trilinear Strichartz Estimates via Decoupling}\label{sec:appendix A}
In this appendix, we give an alternative proof of the refined trilinear Strichartz estimate first established in \cite[Proposition 5.10]{McConnell_Mass_critical_d_1}.

\begin{theorem}\label{theorem: asymetric_strichartz}
For every $\epsilon>0$ there exists a constant $C_\epsilon>0$ such that for all $N_1 \gg N_2 \geq 1$, all $N_1^{-1} \leq T \leq 1$ and all $\phi_1,\phi_2 \in L^2(\T)$ with $\supp \hat{\phi}_j \subset [N_j,2N_j]$, $j=1,2$, the following estimate holds, 
$$\norm{S(t)\phi_1\left(S(t) \phi_2\right)^2}_{L^2_{t,x}([0,T] \times \T )}^2 \leq C_\epsilon N_2^\epsilon \left(T^{1/2}+\frac{N_2}{N_1}\right)\|\phi_1\|_2^2 \|\phi_2\|_2^4.$$
\end{theorem}
The corresponding result in \cite{McConnell_Mass_critical_d_1} is stronger, as it does not contain the $N_2^\epsilon$ loss. Their proof uses number-theoretic methods, whereas ours uses decoupling tools. Thus, our method allows more flexible choices of the locations for the frequencies. We include an alternative proof here for completeness and to highlight our method, which may be useful in higher-dimensional settings.

We shall show that Theorem \ref{asymthm} follows quickly from the asymmetric estimate in Lemma \ref{asym2} by combining it with an induction argument similar to the one in \cite{schippa2023_trilinear}. We also emphasize that Theorem \ref{asymthm} is sharp by considering Example \ref{example1} and Example \ref{example2}.  

By a standard rescaling and periodization argument, it suffices to establish the following decoupling estimate.
\begin{theorem}\label{asymthm} 
Let \[
\Gamma_1, \Gamma_2 \subseteq N_{1/R}(\P)
\] be two caps with $l(\Gamma_i)=L_i \leq 1$ such that $dist(\Gamma_1,\Gamma_2) \sim 1$. Let $f_1$ and $f_2$ be two functins with
\[
\spp \,\widehat{f_i} \subseteq \Gamma_i, 
\qquad 
f_i = \sum_{\gamma \subseteq \Gamma_i} f_\gamma,
\]
where each $f_\gamma$ satisfies the following:
\begin{enumerate}
    \item $\spp \,\widehat{f_\gamma} \subset \gamma$ a cap of radius $1/N$ on $N_{1/R}(\P)$.   
    \item For each $\gamma$, 
    \[
    |f_\gamma| \sim a_\gamma \quad \text{on } B_R, 
    \] for some $a_\gamma \in \R$
    and $f_\gamma$ decays rapidly off $B_R$.
\end{enumerate}
Assume also that $N \leq R \leq N^2$. Then, for any $\epsilon>0$, there exists a constant $C_\epsilon$ such that
\begin{align*}
    \fint_{B_R} 
    & |f_1|^2 |f_2|^4 \\
    &\leq C_\epsilon \,(L_2 N R)^{\epsilon} 
    \max\!\left( L_2, \frac{\sqrt{R}}{N} \right)
    \frac{N^2}{R}\Bigg(\sum_{\gamma \in \Gamma_1} |a_{\gamma}|^2 \Bigg)
    \Bigg(\sum_{\gamma \in \Gamma_1}  |a_{\gamma}|^2 \Bigg)^{2}.
\end{align*}
\end{theorem}

\subsection{Examples}
We see that lemma \ref{asymthm} is sharp up to a $(L_2NR)^\epsilon$ factor, if we consider the following two examples. In both examples, we shall let $a_\xi=1$ for all $\xi \in \Xi_1$ and $\Xi_2$. 
Put 
$$g_1(x,t):=\sum_{j=1}^{L_1N}e(\frac{j}{N}x+\frac{j^2}{N^2}t) \text{ and }g_2(x,t):=\sum_{j=N}^{N+L_2N}e(\frac{j}{N}x+\frac{j^2}{N^2}t).$$
\begin{example}\label{example1}
    Let 
$L_1=L_2^2\geq 1/N.$ Note that $g_i(k/N)\sim L_iN$ for integers $k$. As $g_i$ has Fourier support in a $L_i \times L_i^2$ rectangle, we know that $g_i$ is locally constant on $L_i^{-1} \times L_i^{-2}$ tubes  whose long axes are in the same direction as $d(\Gamma_i)$ where $d(\Gamma_i)$ is the direction of the normal to the paraboloid at the center of $\Gamma_i$. 
Thus, on $R/N$ many $L_2^{-1} \times L_2^{-2}$ tubes in $B_R$, we have $g_i \sim L_iN$. Thus,
\begin{eqnarray*}
    &&\fint_{B_R}|\sum_{\xi\in \Xi_1}a_{\xi}e((x,t)\cdot (\xi,\xi^2))|^2|\sum_{\xi\in \Xi_2}a_{\xi}e((x,t)\cdot (\xi,\xi^2))|^4\\
    &\geq & L_2^{-3}\left(\frac{R}{N}\right)(L_1N)^2(L_2N)^4\\    
    &\sim& L_2N^2/R (\sum_{\xi\in \Xi_1}|a_{\xi}|^2)(\sum_{\xi\in \Xi_1}|a_{\xi}|^2)^2.
\end{eqnarray*}

\end{example}

\begin{example}\label{example2}
    Let 
$L_1=L_2= R^{-1/2}.$ Note that $g_i(k/N)\sim L_iN$ for integers $k$. We know that $g_i$ is locally constant on a $R \times R^{1/2}$ tube whose long axis is in the same direction as $d(\Gamma_i)$ where $d(\Gamma_i)$ is the direction of the normal to the paraboloid at the center of $\Gamma_i$. 
Thus, on $(R/N)^2$ many $R^{1/2} \times R^{1/2}$ squares, we have $g_i \sim L_iN$. Thus,
\begin{eqnarray*}
    &&\fint_{B_R}|\sum_{\xi\in \Xi_1}a_{\xi}e((x,t)\cdot (\xi,\xi^2))|^2|\sum_{\xi\in \Xi_2}a_{\xi}e((x,t)\cdot (\xi,\xi^2))|^4\\
    &\geq & (R^{1/2})^2\left(\frac{R}{N}\right)^2(L_1N)^2(L_2N)^4\\    
    &\sim& (N/\sqrt{R}) (\sum_{\xi\in \Xi_1}|a_{\xi}|^2)(\sum_{\xi\in \Xi_1}|a_{\xi}|^2)^2.
\end{eqnarray*}
\end{example}
\subsection{Proof of Theorem \ref{asymthm}}
We shall focus on the special case when we are in the situation such that for each $\gamma$ either $a_\gamma \sim 1$ or $a_\gamma=0$. In this case, we say that $f_1$ and $f_2$ satisfy $\Cn(L_1,L_2)$. The general case would just follow from it from a standard pigeonholing argument, giving rise to an extra $\log R$ loss. 
\begin{definition}
    Define $D(L_1,L_2,R,N)$ to be the smallest constant such that for any $(R,1/M)$ normalized $(f_1,f_2)$ satisfying $\Cn(L_1,L_2)$, we have 
    \begin{equation}\label{maineq}
        \int_{B_R}|f_1|^2|f_2|^4\lesssim D(L_1,L_2,R,N)\lambda_1 \lambda_2^2R^2.
    \end{equation}
\end{definition}
To prove Theorem \ref{asymthm}, it is enough to prove $D(L_1,L_2,R,N)\lessapprox \max(L_2, \sqrt{R}/N) \frac{N^2}{R}$. To do so, we shall bound $|U_{a,b}(f_1,f_2)|a^2b^4$ for the case when $b$ is big and when $b$ is small separately. 

Suppose that $b \geq R^\epsilon \lambda_2 \lambda_2(R^{-1/2})^{-1/2}$, then Lemma \ref{asym2} implies that 
$$|U_{a,b}(f_1,f_2)|a^2b^4 \lessapprox \max\!\left( L_2, \frac{\sqrt{R}}{N} \right)
    \frac{N^2}{R}\lambda_1
    \lambda_2^2.$$
Suppose that $b \leq R^\epsilon \lambda_2 \lambda_2(R^{-1/2})^{-1/2}$, Lemma \ref{oldsup2} immediately implies the following bound on $|U_b(f_2)|b^6$. 
\begin{lemma}\label{smallb}
    Suppose that  $b \leq R^\epsilon \lambda_2 \lambda_2(R^{-1/2})^{-1/2}$, then 
    $$b^6|B_b| \lesssim_\epsilon R^{O(\epsilon)} N/\sqrt{R}\lambda_2^3R^2.$$
\end{lemma}
Next, we prove the following lemma relating $D(L_1,L_2,R,N)$ to $D(L_2^2,L_2,R,N)$. 
\begin{lemma}(Almost Orthogonality 1)\label{orth1}
    $$D(L_1,L_2,R,N) \lesssim D(L_2^2,L_2,R,N)$$
\end{lemma}
\begin{proof}
  As $D(L_1,L_2,R,N)$ is monotonically increasing in $L_1$ it is enough to prove the case when $L_1 \geq L_2^2$. In the definition of $D(L_1,L_2,R,N)$, let's replace $e^{it\Delta}f_i\eta_{B_R}$ by $g_i$ to simplify the notation. 
  Thus, our  goal is to prove that
  $$\int |g_1|^2|g_2|^4 \lesssim \sum_{\tau\subseteq \Gamma_1}\int|g_{1,\tau}|^2|g_2|^4 $$
  where we break $\Gamma_1$ into caps $\tau$ of length $L_2^2$. 
  \begin{eqnarray*}
   &&\int |g_1|^2|g_2|^4\\
   &=&\int \widehat{|g_1|^2}\widehat{|g_2|^4} \text{ by Plancherel }\\
   &=&\sum_{\tau \subseteq \Gamma_1}\sum_{\tau' \subseteq \Gamma_1}\int \widehat{g_\tau}*\widehat{\overline{g_{\tau'}}}\widehat{|g_2|^4} 
  \end{eqnarray*}
  In order for the integrand to be non-zero, we need 
  $$\tau-\tau' \cap \spp \widehat{|g_2|^4}\not=\emptyset.$$
  Note that $ \spp \widehat{|g_2|^4} \subseteq \Gamma_2+\Gamma_2-\Gamma_2-\Gamma_2$ which is contained in a $O(L_2) \times O(L_2^2)$ rectangle whose long axis is in direction $d(\Gamma_2)$. As $d(\Gamma_1)$ is transversal to $d(\Gamma_2)$, we need $|c(\tau)-c(\tau')| \lesssim L_2^2$. In this case, we call $\tau\sim \tau'$. Thus, 
  \begin{align*}
   \int |g_1|^2|g_2|^4
   &=\sum_{\tau \subseteq \Gamma_1}\sum_{\tau' \sim \tau}\int \widehat{g_\tau}*\widehat{\overline{g_{\tau'}}}\widehat{|g_2|^4} \\
   &=\sum_{\tau \subseteq \Gamma_1}\sum_{\tau' \sim \tau}\int g_\tau\overline{g_{\tau'}}|g_2|^4 \\
   &\lesssim\int |g_\tau|^2|g_2|^4. \qedhere
  \end{align*}
\end{proof}
Combining the two lemmas above, we obtain the following inductive relation. 
\begin{lemma}\label{ind1}
    Suppose that $b \leq R^\epsilon \lambda_2 \lambda_2(R^{-1/2})^{-1/2}$, then 
    $$|U_{a,b}|a^2b^4 \lesssim_\epsilon R^{o(\epsilon)} (D(L_2,L_2^2))^{1/2}(N/\sqrt{R})^{1/2} \lambda_1\lambda_2^2R^2. $$
\end{lemma}
\begin{proof}
Suppose that $b \leq R^\epsilon \lambda_2 \lambda_2(R^{-1/2})^{-1/2}$, by lemma \ref{smallb},
    \begin{align*}
        |U_{a,b}|a^2b^4 & \leq  (|U_{a,b}|a^4b^2)^{1/2} (|W_{2,b}|b^6)^{1/2}\\
        &\leq (D(L_2,L_2^2)\lambda_1^2\lambda_2)^{1/2} (|W_{2,b}|b^6)^{1/2}\\
        &\lesssim_\epsilon R^{o(\epsilon)}(D(L_2,L_2^2)\lambda_1^2\lambda_2)^{1/2} (N/\sqrt{R}\lambda_2^3R^2)^{1/2}\\
        &= R^{o(\epsilon)} (D(L_2,L_2^2))^{1/2}(N/\sqrt{R})^{1/2} \lambda_1\lambda_2^2R^2. \qedhere
    \end{align*}
\end{proof}
Combining lemma \ref{ind1} and lemma \ref{asym2} and orthogonality, we have the following induction lemma 
\begin{lemma}\label{Ind}
    $$D(L_2^2,L_2) \lesssim_\epsilon R^{O(\epsilon)} \max(N^2/R l_2, N/\sqrt{R}, D(L_2^4,L_2^2)^{1/2}(N/\sqrt{R})^{1/2})\lambda_1\lambda_2^2R^2.$$
\end{lemma}
\begin{lemma}(Base Case)\label{base}\\
$$D(L_1,L_2) \leq L_2N.$$  
In particular, we have $D(L_1,L_2) \leq R^{-1/2}N$ for $L_2 \leq R^{-1/2}$. 
\end{lemma}
\begin{proof}
By bilinear restriction, 
    \begin{align*}
        |U_{a,b}|a^2b^4 & \leq  \int_{B_R} |f_1|^2|f_2|^4\\
        & \leq \|f_2\|^2_\infty \int_{B_R} |f_1|^2|f_2|^2\\
        & \leq  \lambda_2^2 R^2\lambda_1 \lambda_2\\
        & \leq  L_2N R^2\lambda_1 \lambda_2^2. \qedhere
    \end{align*}
\end{proof}
\begin{proof}[Proof of Theorem \ref{asymthm}]
    Suppose that $L_2 \leq R^{-1/4}$, then $L_2^2 \leq R^{-1/2}$, by lemma \ref{base} and lemma \ref{Ind}, 
    $$D(L_2^2, L_2) \lesssim_\epsilon R^{O(\epsilon)} \max(L_2N^2/R , D(L_2, L_2^2)^{1/2}(N/\sqrt{R})^{-1/2})\lesssim_\epsilon R^{O(\epsilon)} \max(L_2N^2/R, N/\sqrt{R}).$$
    Suppose that for some $L_2$, 
    $$D(L_2^4, L_2^2) \lesssim_\epsilon R^{O(\epsilon)}  \max(L_2^2N^2/R, N/\sqrt{R}),$$
    then 
    \begin{eqnarray*}
        D(L_2^2, L_2) &\lesssim_\epsilon& R^{o(\epsilon)} \max(L_2N^2/R, (\max(L_2^2N^2/R, N/\sqrt{R}))^{1/2}(N/\sqrt{R})^{1/2})\\
       &\lesssim_\epsilon& R^{o(\epsilon)} \max(L_2N^2/R, N/\sqrt{R})
    \end{eqnarray*}
    Putting these together proves Theorem \ref{asymthm}. 
\end{proof}

\section*{Acknowledgment}
The authors would like to thank Professor Larry Guth for his unwavering support and helpful suggestions throughout the preparation of this paper. This work was completed under his guidance. The authors are also grateful to Professor Gigliola Staffilani for many helpful discussions and valuable insights related to this work. In addition, [NS] would like to thank his thesis supervisor, Professor Alexandre Megretski, for his continued support, guidance, and encouragement throughout this project.
\bibliography{references.bib}

@article{Nirenberg,
     author = {Nirenberg, L.},
     title = {On elliptic partial differential equations},
     journal = {Annali della Scuola Normale Superiore di Pisa - Scienze Fisiche e Matematiche},
     pages = {115--162},
     publisher = {Scuola normale superiore},
     howpublished={\url{chrome-extension://efaidnbmnnnibpcajpcglclefindmkaj/http://archive.numdam.org/item/ASNSP_1959_3_13_2_115_0.pdf}},
     volume = {Ser. 3, 13},
     number = {2},
     year = {1959},
     zbl = {0088.07601},
     mrnumber = {109940},
     url = {http://archive.numdam.org/item/ASNSP_1959_3_13_2_115_0/}
}

@article{Bourgain_lwp,
author = {Bourgain, J.},
journal = {Geometric and functional analysis},
number = {3},
pages = {107-156},
title = {Fourier transform restriction phenomena for certain lattice subsets and applications to nonlinear evolution equations: Part I: Schrödinger equations},
volume = {3},
year = {1993},
}

@article{Bourgain_trilinear,
author = {Bourgain, J.},
journal = {Journal d’Analyse Mathematique},
pages= {125-157},
title = {A remark on normal forms and the “{I}-method” for periodic {NLS}},
volume = {94},
year = {2004},
}

@article{schippa2023_trilinear,
  author = {Schippa, Robert},
  title = {Refinements of Strichartz Estimates on Tori and Applications},
  journal = {Mathematische Annalen},
  year = {2025},
  volume = {391},
  number = {3},
  pages = {3245--3294}
}

@article{Burq_Gerard_Tzvetkov,
   title={Bilinear eigenfunction estimates and the nonlinear {S}chrödinger equation on surfaces},
   volume={159},
   url={http://dx.doi.org/10.1007/s00222-004-0388-x},
   DOI={10.1007/s00222-004-0388-x},
   number={1},
   journal={Inventiones mathematicae},
   publisher={Springer Science and Business Media LLC},
   author    = {N. Burq and P. G{\'e}rard and N. Tzvetkov},
   year={2004},
   month=jul, pages={187–223} }

@article{I-method_1,
author = {Colliander, J. and Keel, M. and Staffilani, G. and Takaoka, H. and Tao, T.},
issn = {1073-2780},
journal = {Mathematical Research Letters},
number = {5},
pages = {659-682},
publisher = {International Press of Boston},
title = {Almost Conservation Laws and Global Rough Solutions to a Nonlinear {S}chrödinger Equation},
volume = {9},
year = {2002},
}

@article{I-method_2,
author = {Colliander, J. and Keel, M. and Staffilani, G. and Takaoka, H. and Tao, T.},
address = {Philadelphia},
copyright = {[Copyright] © 2001 Society for Industrial and Applied Mathematics},
issn = {0036-1410},
journal = {SIAM journal on mathematical analysis},
number = {3},
pages = {649-669},
publisher = {Society for Industrial and Applied Mathematics},
title = {Global Well-Posedness for {S}chrödinger Equations with Derivative},
volume = {33},
year = {2001},
}

@article{KdV,
author = {Colliander, J. and Keel, M. and Staffilani, G. and Takaoka, H. and Tao, T.},
issn = {0894-0347},
journal = {Journal of the American Mathematical Society},
number = {3},
pages = {705-749},
publisher = {American Mathematical Society (AMS)},
title = {Sharp Global Well-Posedness for the {KdV} and Modified {KdV} Equations on \(\mathbb{R}\) and \(\mathbb{T}\)},
volume = {16},
year = {2003},
}

@article{Schippa_mass_critical,
author = {Schippa, Robert},
title = {Improved global well-posedness for mass-critical nonlinear {S}chr\"odinger equations on tori},
year = {2024},
volume = {412},
pages = {87-139},
journal={Journal of Differential Equations},
}

@article{Invariant_Lemma,
author = {Colliander, J. and Keel, M. and Staffilani, G. and Takaoka, H. and Tao, T.},
copyright = {2003 Elsevier Inc.},
issn = {0022-1236},
journal = {Journal of functional analysis},
number = {1},
pages = {173-218},
publisher = {Elsevier Inc},
title = {Multilinear estimates for periodic {KdV} equations, and applications},
volume = {211},
year = {2004},
}

@article{mass_critical_gwp_staffilani_tzirakis,
title = {Global well-posedness for a periodic nonlinear {S}chrödinger equation in 1D and 2D},
journal = {Discrete and Continuous Dynamical Systems},
volume = {19},
number = {1},
pages = {37-65},
year = {2007},
issn = {1078-0947},
doi = {10.3934/dcds.2007.19.37},
url = {https://www.aimsciences.org/article/id/793e8cd2-d8af-4c34-b523-2b4c3d70c05e},
author = {Daniela {De Silva} and Nataša Pavlović and Gigliola Staffilani and Nikolaos Tzirakis}
}

@article{Staffalani_mass_critical_errata,
number = {3/4},
pages = {399-400},
publisher = {Khayyam Publishing, Inc},
title = {Errata to "The {C}auchy problem for the semi-linear quintic {S}chrödinger equation in 1D", Differential Integral Equations, 18 (2005), No. 8, 947-960},
volume = {23},
year = {2010},
journal = {Differential Integral Equations},
author = {Tzirakis, Nikolaos},
copyright = {Copyright 2010 Khayyam Publishing, Inc.},
}

@article{Herr_Kwak_2024_mass_critical,
title={Strichartz estimates and global well-posedness of the cubic {NLS} on $\mathbb {T}^{2}$}, 
volume={12}, 
DOI={10.1017/fmp.2024.11}, 
journal={Forum of Mathematics, Pi},
author={Herr, Sebastian and Kwak, Beomjong}, 
year={2024}, 
pages={e14}
}

@misc{Herr_Kwak_2025_mass_critical_large_mass,
  title = {Global Well-Posedness of the Cubic Nonlinear Schrödinger Equation in the Mass-Critical Case on the Torus},
  author = {Herr, Sebastian and Kwak, Beomjong},
  year = {2025},
  note = {Preprint}
}

@article{Bourgain_Demeter_decoupling,
author = {Bourgain, Jean and Demeter, Ciprian},
copyright = {Copyright © 2015 Princeton University (Mathematics Department)},
issn = {0003-486X},
journal = {Annals of mathematics},
number = {1},
pages = {351-389},
publisher = {Department of Mathematics at Princeton University},
title = {The proof of the l2 Decoupling Conjecture},
volume = {182},
year = {2015},
}

@article{Herr_Tartaru_Tzvetkov,
author = {Herr, Sebastian and Tataru, Daniel and Tzvetkov, Nikolay},
copyright = {Copyright 2011 Duke University Press},
journal = {Duke mathematical journal},
number = {2},
pages = {329-349},
publisher = {DUKE University Press},
title = {Global well-posedness of the energy-critical nonlinear {S}chrödinger equation with small initial data in ${H}^1( \mathbb{T}^3)$},
volume = {159},
year = {2011},
}

@article{Hadac_Herr_Koch,
title = {Well-posedness and scattering for the KP-II equation in a critical space},
journal = {Annales de l'Institut Henri Poincaré C, Analyse non linéaire},
volume = {26},
number = {3},
pages = {917-941},
year = {2009},
issn = {0294-1449},
doi = {https://doi.org/10.1016/j.anihpc.2008.04.002},
url = {https://www.sciencedirect.com/science/article/pii/S0294144908000474},
author = {Martin Hadac and Sebastian Herr and Herbert Koch},
}

@article{Koch_Tataru,
author = {Koch, Herbert and Tataru, Daniel},
address = {[Durham, N.C.] :},
journal = {International mathematics research notices : IMRN.},
lccn = {95648111},
publisher = {Duke University Press,},
title = {A Priori Bounds for the 1D Cubic {NLS} in Negative Sobolev Spaces},
pages={rnm053-rnm053},
volume = {2007},
year = {2007-01-01},
}

@article{Burq_illposedness_T,
  title={An instability property of the nonlinear {S}chr{\"o}dinger equation on ${S}^{d}$},
  author={Nicolas Burq and Patrick G{\'e}rard and Nikolay Tzvetkov},
  journal={Mathematical Research Letters},
  year={2002},
  volume={9},
  pages={323-335},
  url={https://api.semanticscholar.org/CorpusID:53308123}
}

@misc{fractional_global_wellposedness_T,
      title={Global Well-posedness for the periodic fractional cubic {NLS} in 1D}, 
      author={Alexandre Megretski and Nikolaos Skouloudis},
      year={2025},
      note={Preprint},
      eprint={2508.01204},
      archivePrefix={arXiv},
      primaryClass={math.AP},
      url={https://arxiv.org/abs/2508.01204}, 
}

@article{McConnell_Mass_critical_d_1,
pages = {368-405},
title = {On lattice points, short-time estimates, and global well-posedness of the quintic {NLS} on $\mathbb{T}$},
volume = {47},
year = {2026},
issn = {1078-0947},
journal = {Discrete and continuous dynamical systems. Series A},
author = {McConnell, Ryan},
}

@article{LiWuXu_Mass_critical_d_1,
  author  = {Li, Y. and Wu, Y. and Xu, G.},
  title   = {Global well-posedness for the mass-critical nonlinear {N}onlinear {S}chr{\"o}dinger equation on $\mathbb{T}$},
  journal = {Journal of Differential Equations},
  volume  = {250},
  number  = {6},
  year    = {2011},
  pages   = {2715--2736},
}

@article{BurqGerardTzvetkov_Strichartz,
  author    = {N. Burq and P. G{\'e}rard and N. Tzvetkov},
  title     = {Strichartz inequalities and the nonlinear {S}chr\"odinger equation on compact manifolds},
  journal   = {Amer. J. Math.},
  volume    = {126},
  number    = {3},
  pages     = {569--605},
  year      = {2004}
}

@article{Kishimoto_sharp_CN,
journal = {Proceedings of the American Mathematical Society},
number = {8},
pages = {2649-2660},
publisher = {American Mathematical Society},
title = {Remark on the periodic mass critical nonlinear {S}chrödinger equation},
volume = {142},
year = {2014},
author = {Kishimoto, Nobu},
address = {Providence, Rhode Island},
copyright = {Copyright 2014 American Mathematical Society; reverts to public domain 28 years from publication},
}

@article{GuoLiYung_Strichartz,
  author    = {Shaoming Guo and Zane Kun Li and Po-Lam Yung},
  title     = {Improved discrete restriction for the parabola},
  journal   = {Mathematical Research Letters},
  volume    = {30},
  number    = {5},
  pages     = {1375--1409},
  year      = {2024},
  doi       = {10.4310/MRL.2023.v30.n5.a4},
}

@article{GuthMaldagueWang_Improved_Dec_Parabola,
journal = {Journal of the European Mathematical Society : JEMS},
number = {3},
pages = {875-917},
title = {Improved decoupling for the parabola},
volume = {26},
year = {2024},
author = {Guth, Larry and Maldague, Dominique and Wang, Hong},
issn = {1435-9855},
}

@article {FGM,
    AUTHOR = {Fu, Yuqiu and Guth, Larry and Maldague, Dominique},
     TITLE = {Sharp superlevel set estimates for small cap decouplings of
              the parabola},
   JOURNAL = {Rev. Mat. Iberoam.},
  FJOURNAL = {Revista Matem\'atica Iberoamericana},
    VOLUME = {39},
      YEAR = {2023},
    NUMBER = {3},
     PAGES = {975--1004},
      ISSN = {0213-2230,2235-0616},
   MRCLASS = {42B15 (42B20)},
  MRNUMBER = {4603641},
       DOI = {10.4171/rmi/1393},
       URL = {https://doi.org/10.4171/rmi/1393},
}

@article {DGW,
    AUTHOR = {Demeter, Ciprian and Guth, Larry and Wang, Hong},
     TITLE = {Small cap decouplings},
      NOTE = {With an appendix by D. R. Heath-Brown},
   JOURNAL = {Geom. Funct. Anal.},
  FJOURNAL = {Geometric and Functional Analysis},
    VOLUME = {30},
      YEAR = {2020},
    NUMBER = {4},
     PAGES = {989--1062},
      ISSN = {1016-443X,1420-8970},
   MRCLASS = {42B10 (11L07 53A07)},
  MRNUMBER = {4153908},
MRREVIEWER = {Ruixiang\ Zhang},
       DOI = {10.1007/s00039-020-00541-5},
       URL = {https://doi.org/10.1007/s00039-020-00541-5},
}

@misc{amplitude,
      title={Amplitude dependent wave envelope estimates for the cone in $\mathbb{R}^3$}, 
      author={Dominique Maldague and Larry Guth},
      year={2022},
      eprint={2206.01093},
      archivePrefix={arXiv},
      notes={Preprint},
      primaryClass={math.CA},
      url={https://arxiv.org/abs/2206.01093}, 
}

@article {Guth_Solomon_Wang,
    AUTHOR = {Guth, Larry and Solomon, Noam and Wang, Hong},
     TITLE = {Incidence estimates for well spaced tubes},
   JOURNAL = {Geom. Funct. Anal.},
  FJOURNAL = {Geometric and Functional Analysis},
    VOLUME = {29},
      YEAR = {2019},
    NUMBER = {6},
     PAGES = {1844--1863},
      ISSN = {1016-443X,1420-8970},
   MRCLASS = {52C10 (42B25)},
  MRNUMBER = {4034922},
MRREVIEWER = {Xiumin\ Du},
       DOI = {10.1007/s00039-019-00519-y},
       URL = {https://doi.org/10.1007/s00039-019-00519-y},
}

@article{cor77,
 ISSN = {00029327, 10806377},
 URL = {http://www.jstor.org/stable/2374006},
 author = {Antonio Córdoba},
 journal = {American Journal of Mathematics},
 number = {1},
 pages = {1--22},
 publisher = {Johns Hopkins University Press},
 title = {The {K}akeya Maximal Function and the Spherical Summation Multipliers},
 urldate = {2026-05-05},
 volume = {99},
 year = {1977}
}

@inproceedings{Dodson_JEDP_2011,
  author       = {Benjamin Dodson},
  title        = {Global well-posedness and scattering for the mass-critical {NLS}},
  booktitle    = {Journ{\'e}es {\'E}quations aux d{\'e}riv{\'e}es partielles, Biarritz, 6 juin--10 juin 2011},
  series       = {GDR 2434 (CNRS)},
  year         = {2011},
  pages        = {1--11},
  url          = {https://proceedings.centre-mersenne.org/articles/10.5802/jedp.76/},
  note         = {Lecture notes}
}

@article{CKSTT2010_energy_transfer,
  author  = {Colliander, J. and Keel, M. and Staffilani, G. and Takaoka, H. and Tao, T.},
  title   = {Transfer of energy to high frequencies in the cubic defocusing nonlinear Schr\"odinger equation},
  journal = {Inventiones Mathematicae},
  volume  = {181},
  number  = {1},
  pages   = {39--113},
  year    = {2010},
  doi     = {10.1007/s00222-010-0242-2}
}
\bibliographystyle{amsplain2.bst}

\end{document}